\documentclass[reqno,english]{amsart}
\usepackage{amsfonts,amsmath,latexsym,verbatim,amscd,mathrsfs,color,array}
\usepackage[colorlinks=true]{hyperref}

\usepackage{amsmath}

\usepackage{amssymb,amsthm,graphicx,color}
\usepackage[hmargin=3cm, vmargin=2.9cm]{geometry}
\usepackage{bbm}
\usepackage{amssymb}
\usepackage{pdfsync}
\usepackage{epstopdf}
\usepackage{cite}
\usepackage{graphicx}
\usepackage{multirow}
\usepackage[font=bf,aboveskip=15pt]{caption}
\usepackage[toc,page]{appendix}
\usepackage[colorlinks=true]{hyperref}
\usepackage{bookmark}
\usepackage{todonotes}

\usepackage{pgfplots}
\pgfplotsset{compat=1.18}

\allowdisplaybreaks

\DeclareFontShape{OT1}{cmr}{bx}{sc}{<-> cmbcsc10}{}
\newcommand{\la}{\lambda}

\newcommand{\R}{\mathbb R}

\newcommand{\bt}{\begin{theorem}}
\newcommand{\et}{\end{theorem}}
\newcommand{\bl}{\begin{lemma}}
\newcommand{\el}{\end{lemma}}
\newcommand{\bd}{\begin{definition}}
\newcommand{\ed}{\end{definition}}
\newcommand{\bc}{\begin{corollary}}
\newcommand{\ec}{\end{corollary}}
\newcommand{\bp}{\begin{proof}}
\newcommand{\ep}{\end{proof}}
\newcommand{\bx}{\begin{example}}
\newcommand{\ex}{\end{example}}
\newcommand{\bi}{\begin{exercise}}
\newcommand{\ei}{\end{exercise}}
\newcommand{\bo}{\begin{prop}}
\newcommand{\eo}{\end{prop}}
\newcommand{\br}{\begin{remark}}
\newcommand{\er}{\end{remark}}
\newcommand{\be}{\begin{equation}}
\newcommand{\ee}{\end{equation}}
\newcommand{\ba}{\begin{align}}
\newcommand{\ea}{\end{align}}
\newcommand{\bn}{\begin{enumerate}}
\newcommand{\en}{\end{enumerate}}
\newcommand{\bg}{\begin{align*}}
\newcommand{\bcs}{\begin{cases}}
\newcommand{\ecs}{\end{cases}}

\newcommand{\bean}{\begin{eqnarray*}}
\newcommand{\eean}{\end{eqnarray*}}

\newtheorem{definition}{Definition}[section]
\newtheorem{theorem}{Theorem}[section]
\newtheorem{lemma}{Lemma}[section]

\newtheorem{prop}{Proposition}[section]
\newtheorem{remark}{Remark}[section]

\numberwithin{equation}{section}

\begin{document}

\title[Collapsing-tube blow-up]
{Collapsing-tube type II blow-up for the energy-supercritical heat equation}

\author[M. del Pino]{Manuel del Pino}
\address{\noindent
Department of Mathematical Sciences,
University of Bath, Bath BA2 7AY, United Kingdom}
\email{mdp59@bath.ac.uk}

\author[M. Musso]{Monica Musso}
\address{\noindent
Department of Mathematical Sciences,
University of Bath, Bath BA2 7AY, United Kingdom}
\email{mm2683@bath.ac.uk}

\author[J. Wei]{Juncheng Wei}
\address{\noindent
Department of Mathematics,
Chinese University of Hong Kong,
Shatin, NT, Hong Kong}
\email{wei@math.cuhk.edu.hk}

\author[Y. Zhou]{Yifu Zhou}
\address{\noindent
School of Mathematics and Statistics, Wuhan University, 
Wuhan 430072, China}
\email{yifuzhou@whu.edu.cn}

\begin{abstract}
We construct a new type II finite-time blow-up mechanism for the
energy-supercritical heat equation
\[
u_t=\Delta u+u^3,
\qquad n\geq 5.
\]
The solution is positive and blows up only at the origin, but in a highly
anisotropic fashion. As \(t\nearrow T\), the solution concentrates in a
thin tubular region around an \((n-4)\)-dimensional sphere whose radius
shrinks to zero at the self-similar scale
\[
\xi_r(t)\sim \sqrt{2(n-4)(T-t)}.
\]
At the same time, concentration takes place transversely to the sphere at
the much smaller scale
\[
\lambda(t)\sim
\kappa_*
\frac{T-t}{|\log(T-t)|^{\frac n{n-2}}},
\]
for some \(\kappa_*>0\). More precisely, in cylindrical coordinates
\(r=|x'|\), \(z\in\mathbb R^3\), the leading profile is
\[
u(x,t)
\sim
\frac{1}{\lambda(t)}
U\left(
\frac{r-\xi_r(t)}{\lambda(t)},
\frac{z}{\lambda(t)}
\right),
\]
where \(U\) is the Aubin--Talenti bubble in \(\mathbb R^4\).

The construction reveals a two-scale singularity mechanism in which a
critical transverse bubble concentrates around a geometric set that
itself collapses. The concentration tube evolves at the parabolic scale
\(\sqrt{T-t}\), whereas its transverse thickness is governed by the much
smaller type II scale \(\lambda(t)\). The logarithmic blow-up law is
determined by a nonlocal modulation equation arising from the interaction
between the four-dimensional critical bubble and the axisymmetric heat
kernel. To our knowledge, this seems to be the first Type II blowup that quantifies the effect of a self-similar collapsing tube.

The exponent \(p=3\) is energy-supercritical in dimensions \(n\geq5\),
but lies below the Joseph--Lundgren exponent for $5\leq n\leq 12$, in a regime where positive
radial type II blow-up is ruled out. The present result provides the first
example of a positive type II, single-point blow-up through a collapsing thin-tube
geometry.
\end{abstract}
\maketitle



\section{Introduction}

\medskip
{ 

Semilinear heat equations have been a classical model for the study of
singularity formation in parabolic problems since the work of Fujita
\cite{Fujita66}. We refer to the monograph \cite{Souplet07book} for a
comprehensive account of the theory. For a finite-time blow-up solution
of
\[
u_t=\Delta u+u^p,
\qquad p>1,
\]
one distinguishes between type I blow-up, for which
\[
\limsup_{t\to T}
(T-t)^{\frac{1}{p-1}}
\|u(\cdot,t)\|_{\infty}
<+\infty,
\]
and type II blow-up, where this quantity is infinite. Type I blow-up is
at most of the order predicted by the ODE \(u_t=u^p\), while type II
blow-up is driven by a more delicate concentration mechanism.

The occurrence of type II blow-up depends strongly on the relation
between \(p\) and the Sobolev critical exponent. In the subcritical
range, type I behavior is predominant under broad assumptions; see, for
instance,
\cite{Giga87Indiana,Giga04Indiana,Giga04MMAS,Quittner21Duke,Merle97Duke}.
In the energy-critical case
\[
p=\frac{n+2}{n-2},
\]
type II blow-up is governed by concentration of the Aubin--Talenti
bubble. In low dimensions, finite-time type II blow-up was predicted
formally in \cite{FHV00} and subsequently constructed rigorously in
several settings; see, among others,
\cite{Schweyer12JFA,ni4d,type25D,harada6D}. In contrast,
positive finite-time type II blow-up is excluded in dimensions
\(n\geq 7\) under the hypotheses of \cite{wang2021refined}.  See \cite{Collot17CMP} for the classification of near-soliton dynamics and \cite{KimMerle,aryan} for the classification of bubble decomposition  in large dimensions.

\medskip
The energy-supercritical range
\[
p>\frac{n+2}{n-2}
\]
is more rigid in the radial class and at the same time more flexible
under nonradial perturbations. Herrero and Vel\'azquez constructed
radial type II blow-up solutions in the Joseph--Lundgren range
\[
n\geq 11,
\qquad
p>p_{JL}(n):=\begin{cases}
\infty~&\mbox{ if }~3\leq n\leq 10\\
1+\frac{4}{n-4-2\sqrt{n-1}}~&\mbox{ if }~n\geq11
\end{cases},
\]
see \cite{herrero1992blow,Velazquez94pJL}, while nonexistence and
classification results in the complementary Matano--Merle range $\frac{n+2}{n-2}<p<p_{JL}$ were
obtained in
\cite{Merle04CPAM,Merle09JFA,Mizoguchi11JDE}. See \cite{Mizoguchi05Indiana} when the domain is a ball and also \cite{Collot17APDE} with the restriction that $p$ is an odd integer.
Nonradial constructions show that geometry can create type II
mechanisms which are invisible in the radial class. Type II blow-up on a
fixed positive-dimensional set was constructed in \cite{17type2}. A
different strongly anisotropic type II mechanism was discovered by
Collot, Merle and Rapha\"el in \cite{CollotMerleRaphael}: the singularity
forms at an isolated point, but concentration occurs at different rates
in different spatial directions.

\medskip
The present paper gives a new anisotropic single-point blow-up mechanism that encompasses both the sub-threshold regime in lower dimensions and the super-threshold regime in higher dimensions, relative to the Joseph–Lundgren threshold. In particular, it provides the
first positive type II finite-time blow-up solutions for the cubic
equation \(u_t=\Delta u+u^3\) in dimensions \(n\geq5\), for instance in
dimension \(5\). The singularity is reached through a thin tubular
profile around a shrinking sphere, which collapses to the origin at the
blow-up time.  Note that the nonlinearity $ p=3$ falls within  the so-called  supercritical {\it degenerate} range outlined in \cite[Problem 2]{CollotMerleRaphael}. Within this specific regime, we discover a novel, two-scale blowup mechanism. See the conjectures in Section \ref{sec-1.3} below.

\medskip
The key point is that, in the symmetry class
considered below, the problem reduces in the variables
\[
r=\sqrt{x_1^2+\cdots+x_{n-3}^2},
\qquad
z=(x_{n-2},x_{n-1},x_n),
\]
to a four-dimensional critical heat equation with the additional
geometric drift
\[
\frac{n-4}{r}u_r.
\]
Thus the inner profile is the critical Aubin--Talenti bubble in
\(\mathbb R^4\), while the ambient dimension \(n\) enters through the
cylindrical drift and through the axisymmetric heat kernel governing the
outer correction.

This interaction between a four-dimensional critical bubble and a
shrinking higher-dimensional concentration set is responsible for the
new blow-up law. The concentration sphere collapses at the parabolic self-similar
scale,
\[
\xi_r(t)\sim \sqrt{2(n-4)(T-t)},
\]
while the transverse scale of the bubble is much smaller:
\[
\lambda(t)\sim
\kappa_*
\frac{T-t}{|\log(T-t)|^{\frac n{n-2}}}.
\]
Thus the solution develops a thin tubular profile around a sphere whose
radius itself tends to zero.

The case treated in this paper, corresponding to a four-dimensional
critical transverse bubble, is not meant to be the only possible
shrinking-sphere mechanism. Rather, it is the natural first case in
which this geometry can be connected with several already known critical
parabolic concentration phenomena. The logarithmic law
\[
\lambda(t)\sim
\frac{T-t}{|\log(T-t)|^2}
\]
appears, for instance, in the critical four-dimensional heat equation,
in harmonic map heat flow;  see
\cite{ni4d,17HMF,Schweyer12JFA} and the references therein. Other related logarithmic laws were found, for instance, in 2D Keller--Segel system and in 3D axially symmetric Keller--Segel
concentration on a sphere of fixed positive radius;  see
\cite{Collot-KS-2023,ringKS1,Collot-KS,gluingKS,KS-finite} and the references therein. In this sense, the present result should be
viewed as a geometric variant of a robust critical-parabolic mechanism,
now coupled to the collapse of the concentration set itself.

At the level of the leading ansatz, the concentration is obtained by
placing a \(4D\) bubble in the transverse variables
\[
(r,z)=\big(|x'|,z\big),
\]
centered at
\[
(r,z)=(\xi_r(t),0).
\]
Thus the main profile is equivalently written as
\[
U_{\lambda,\xi}(x',z,t)
=
\lambda^{-1}(t)
U\left(
\frac{
\sqrt{\big(|x'|-\xi_r(t)\big)^2+|z|^2}
}
{\lambda(t)}
\right),
\]
where \(x'\in\mathbb R^{n-3}\) and \(z\in\mathbb R^3\). This expression
makes explicit that the singular set is the sphere
\[
|x'|=\xi_r(t),
\qquad
z=0.
\]

}

{

\subsection{Type II singularity along a shrinking tube}

\medskip

We now state the main result. We consider
\begin{equation}\label{eqn}
\begin{cases}
u_t=\Delta u+u^{3},~&\mbox{ in } \Omega\times (0,T),\\
u(x,t)=u|_{\partial\Omega},~&\mbox{ on } \partial\Omega\times (0,T),\\
u(x,0)=u_0(x),~&\mbox{ in } \Omega,
\end{cases}
\end{equation}
where \(n\geq 5\) and \(\Omega=\R^n\), or \(\Omega\subset\R^n\) is a
smooth bounded domain satisfying the symmetries described below. We
write
\[
x=(x^*,x^{**})\in\R^{n-3}\times\R^3,
\qquad
r=|x^*|,
\qquad
z=x^{**},
\]
and look for solutions of the form
\begin{equation}\label{symmetryclass}
u(x,t)=\tilde u(r,z,t),
\end{equation}
even in the three \(z\)-variables. 

The corresponding \((r,z)\)-domain is
\begin{equation}\label{def-domain}
\mathcal D
:=
\left\{
(r,z):
r=\sqrt{x_1^2+\cdots+x_{n-3}^2},
\ z=(x_{n-2},x_{n-1},x_n),
\ x\in\Omega
\right\}.
\end{equation}

\begin{theorem}\label{thm}
Let \(n\geq 5\). Let \(\Omega=\R^n\), or let
\(\Omega\subset\R^n\) be a smooth bounded domain, radial in the first
\(n-3\) variables and even in the remaining three variables. Then, for
\(T>0\) sufficiently small, there exist initial data, and boundary data
in the bounded-domain case, such that the solution of \eqref{eqn} blows
up at time \(T\) along a shrinking \((n-4)\)-dimensional sphere.
More precisely, as \(t\nearrow T\),
\[
u(x,t)
\sim
\lambda^{-1}(t)
U\left(
\frac{(r,z)-(\xi_r(t),0)}
{\lambda(t)}
\right),
\]
where \(U\) is the standard bubble in \(\R^4\). Moreover, for some
\(\kappa_*>0\),
\[
\xi_r(t)\sim \sqrt{2(n-4)(T-t)}
\]
and
\[
\lambda(t)
\sim
\kappa_*
\frac{T-t}{|\log(T-t)|^{\frac n{n-2}}}.
\]
\end{theorem}

When \(n=4\), the shrinking sphere degenerates to a point and the
geometric drift disappears. The rate in Theorem \ref{thm} then formally
reduces to
\[
\lambda(t)
\sim
\frac{T-t}{|\log(T-t)|^2},
\]
which is the stable \(k=1\) rate predicted in \cite{FHV00} for the cubic
critical problem in \(\R^4\).

\medskip

Theorem \ref{thm} exhibits a new type II blow-up mechanism in which
concentration takes place along a positive-dimensional manifold whose
size shrinks at the self-similar scale. More precisely, the
concentration set is an \((n-4)\)-dimensional sphere of radius
\[
\xi_r(t)
\sim
\sqrt{2(n-4)(T-t)},
\]
which collapses to the origin as \(t\nearrow T\). At the same time, the
transverse concentration scale satisfies
\[
\lambda(t)\ll \sqrt{T-t}.
\]
Thus the solution develops a thin tubular profile around a sphere whose
radius itself tends to zero. A schematic depiction of this evolution for
\(n=5\), where the concentration set is a shrinking circle, is given in
Figure \ref{Fig.1}.

\begin{figure}[htbp]
\centering
\includegraphics[width=110mm]{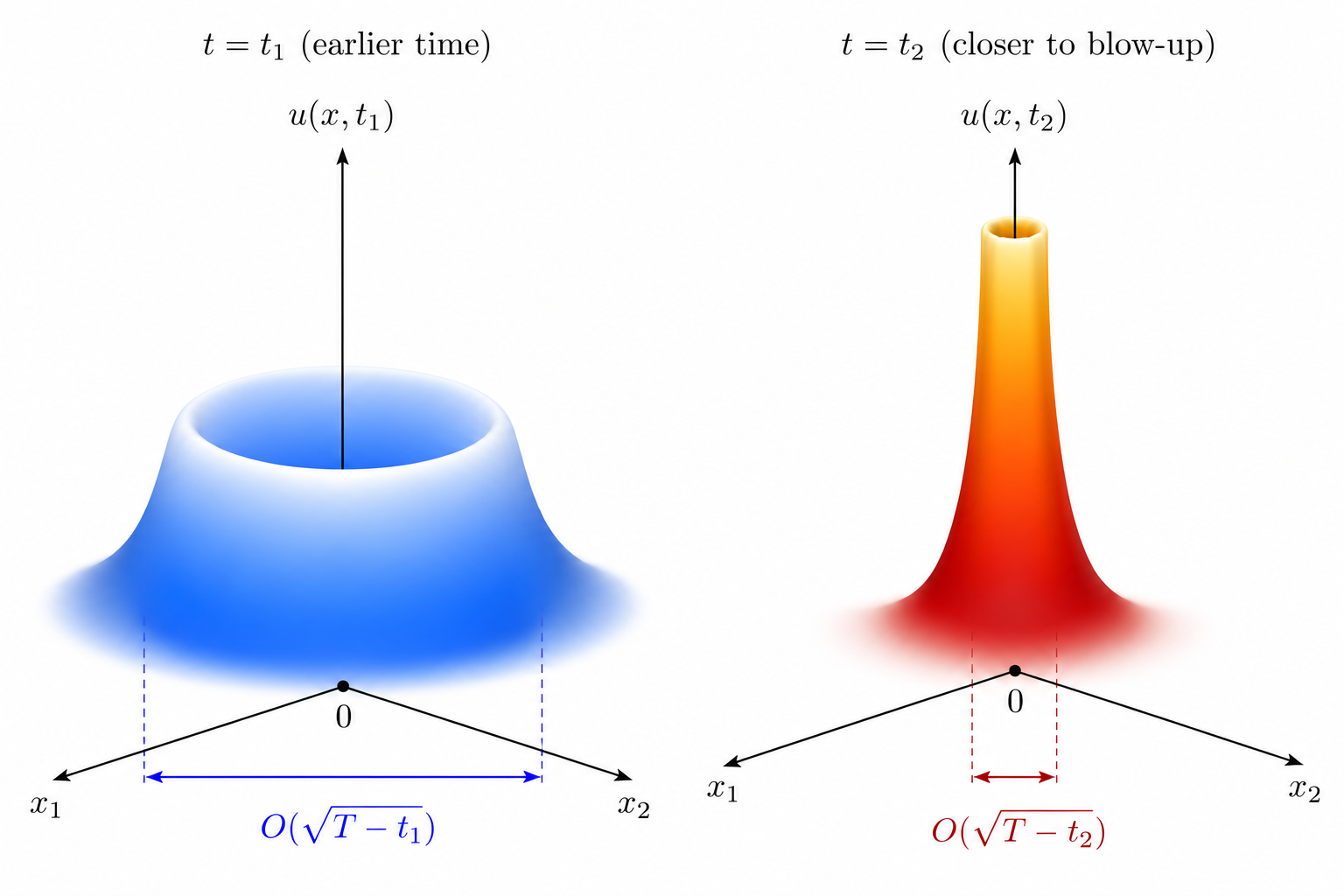}
\caption{A schematic depiction of the evolution of the ``thin tube'' in
Theorem \ref{thm} for \(n=5\). In this case, the concentration set is a
shrinking circle. For \(0<t_1<t_2<T\), the concentration set evolves
from the picture on the left to the picture on the right, and finally
collapses to a point as \(t\nearrow T\).}
\label{Fig.1}
\end{figure}

The construction is related to earlier examples of type II blow-up along
positive-dimensional sets. In the energy-supercritical heat equation
with the second Sobolev exponent, type II blow-up concentrating on a
fixed circle was constructed in \cite{17type2}. The present result is
different in that the concentration set itself collapses at the
self-similar scale. 

{
A closely related two-scale concentration pattern appears in the
numerical scenario of Hou and Huang for potential singularity formation
in the three-dimensional incompressible axisymmetric Euler equations
\cite{HouHuangEuler}. In cylindrical coordinates \((r,\theta,z)\), their
scenario describes a concentrating ring travelling toward the origin.
In notation parallel to ours, its center may be written as
\[
(r,z)=(\xi(t),0),
\qquad
\xi(t)=O((T-t)^{1/2}),
\]
while its transverse thickness is of order
\[
\lambda(t)=O(T-t).
\]
At the level of the angular vorticity, this corresponds formally to a
local ansatz of the form
\[
\omega^\theta(r,z,t)
\sim
\frac{1}{T-t}
\Omega\left(
\frac{r-\xi(t)}{\lambda(t)},
\frac{z}{\lambda(t)}
\right).
\]

The geometry of Theorem \ref{thm} has the same two-scale character. In
our construction,
\[
\xi_r(t)\sim \sqrt{2(n-4)(T-t)},
\]
whereas
\[
\lambda(t)\sim
\kappa_*
\frac{T-t}{|\log(T-t)|^{\frac n{n-2}}}.
\]
Thus, in both settings, the radial location of the concentration set lies at a distance of order \((T-t)^{1/2}\) from the origin, while concentration in the transverse directions occurs on a much smaller scale, of order (T-t) up to a logarithmic correction in the present problem.

\medskip
The possible relevance of collapsing-ring scenarios in boundaryless axisymmetric flows is discussed in \cite{ChaeTsaiLuoHou}. Although the underlying mechanisms are different, the present construction provides a rigorous realization, in a scalar parabolic model, of a two-scale collapsing-ring geometry closely resembling the scenario observed numerically by Hou and Huang in \cite{HouHuangEuler}. In particular, the law
\[
\xi_r(t)\sim \sqrt{T-t}
\]
is obtained here from an explicit modulation equation, suggesting that this radial scale may be a robust geometric feature of axisymmetric singularity formation.

Related nearly self-similar scenarios have been proposed for generalized axisymmetric Navier--Stokes and Boussinesq equations in \cite{HouNSBoussinesq}. In addition, Tao \cite{TaoEuler} constructed a similar ``neck-pinch'' singularity, with self-similar size \(\sqrt{T-t}\), for a generalized three-dimensional Euler equation.

The proof is based on an inner--outer gluing scheme. We refer to the second author
\cite{MussoNAMS} for a survey of this method in parabolic problems.
The approximate solution is built from a sharply scaled \(4D\)
Aubin--Talenti bubble centered at \(r=\xi_r(t)\), \(z=0\).

The inner problem is governed by the linearized operator around the
bubble in \(\mathbb R^4\), while the outer problem is solved in the
original \(n\)-dimensional variables within the imposed symmetry class.
The solvability conditions for the inner problem determine the
modulation equations for \(\lambda(t)\) and \(\xi(t)\). }

\medskip
\subsection{Main difficulties and novelties}

In the symmetry class \eqref{symmetryclass}, problem \eqref{eqn} is reduced to solving
\begin{equation}\label{eqn''}
\begin{cases}
\tilde u_t=\Delta_{(r,z)} \tilde u+\frac{n-4}{r}\tilde u_r+\tilde u^{3},~&\mbox{ in } \mathcal D\times (0,T)\\
\tilde u_r=0,~&\mbox{ on } (\mathcal D\cap\{r=0\})\times (0,T)\\
\tilde u=\tilde u|_{\partial\Omega},~&\mbox{ on } (\partial \mathcal D\backslash \{r=0\})\times (0,T)\\
\tilde u(\cdot,0)=\tilde u_0,~&\mbox{ in } \mathcal D\\
\end{cases}
\end{equation}
where $\mathcal D$ is defined in \eqref{def-domain} and $\Delta_{(r,z)}:=\partial_r^2+\Delta_z$ is the Laplacian in $\R^4$. Note that $p=3$ is the critical Sobolev exponent in $\R^4$. So problem \eqref{eqn''} can be viewed as the energy critical problem in $\R^4$ with drift term $\frac{n-4}{r}\tilde u_r$. Major difficulties and novelties of this supercritical problem come from several aspects. 

\medskip

\noindent (1) {\it Moving singularity with self-similar rate.} The term $\frac{n-4}{r}\tilde u_r$ plays a crucial role in producing the shrinking concentration set through aforementioned orthogonality condition at mode $1$. The translation dynamics read
\begin{equation*}
\left\{
\begin{aligned}
&\dot{\xi}_r(t) +\frac{n-4}{\xi_r(t)}=o(1),\\
&\dot\xi_{z_j}(t)=o(1),\quad j=2,3,4,
\end{aligned}
\right.
\end{equation*}
and the anisotropy between $\xi_r$ and $\xi_{z_j}$ is due to the presence of above drift term, making the concentration location shrinking in $r$ with the self-similar rate. The fact that the type II rate $\la(t)\ll \sqrt{T-t}$ plays a key role here.

\medskip 
{

\noindent (2) {\it Improvement of the approximation via the axisymmetric heat kernel.}
The leading error of the \(4D\) bubble, when projected onto the scaling
mode, contains a slowly decaying part which cannot be treated as a
perturbative remainder. We therefore introduce a global correction
\(\Psi_0\), represented by the axisymmetric heat kernel derived in
Appendix \ref{App-heatkernel}. The leading mechanism by which this correction contributes
to the scaling modulation equation is described in Section \ref{sec-laxi}, while the
full asymptotic estimates are carried out in Appendix \ref{App-nonlocal}.

\medskip

\noindent (3) {\it Nonlocal scaling law.}
The reduced equation for the scaling parameter \(\lambda(t)\) is not a
local ODE, but a nonlocal integro-differential equation. Its main term
has the form
\[
\int_0^{t-(T-t)}
\frac{\dot\lambda(s)}{t-s}\,ds
+
c_n^*
\int_{t-(T-t)}^{t-\lambda^2(t)}
\frac{\dot\lambda(s)}{t-s}\,ds
=
-c+o(1),
\]
where
\[
c_n^*=\frac{n-2}{2}.
\]
Solving this equation to sufficient precision gives the logarithmic
law stated in Theorem \ref{thm}. The control of the remainder requires
quantitative H\"older estimates and is carried out in Section \ref{subsec-la}.

\medskip

\noindent (4) {\it Refined estimates in the linearized problems.} 
Since the desired blow-up solution is located in the self-similar regime
$r=O(\sqrt{T-t})$, the estimates for the linear outer problem are much
more delicate than in the case of a fixed concentration set, where
$r=O(1)$. In particular, the shrinking effect $r\sim \sqrt{T-t}\to0$
makes the term
\[
\frac{n-4}{r}\tilde u_r
\]
difficult to control. To carry out the fixed point argument in suitable
weighted topologies, refined estimates are therefore needed both for the
outer and the inner problems.

For the outer problem, we use Duhamel's formula in $\R^n$ within the
symmetry class. For the inner problem, motivated by
\cite{Green16JEMS,17HMF}, we decompose the linearized problem into
scaling, translation and higher modes. The scaling and translation modes
are the most delicate ones. To handle them, we introduce another
inner--outer gluing procedure (``re-gluing'') together with a careful
blow-up argument in Section \ref{sec-linearinner}. This is needed because
the inner solution deteriorates near the blow-up set, and without this
refinement the nonlinear terms in the coupled inner and outer equations
cannot be controlled.

Finally, the radius $R(t)$ of the localized inner problem and the weighted
topologies have to be chosen carefully. The logarithmic factor in
\[
R(t)=(T-t)^{-1/2}|\log(T-t)|^{\theta},
\qquad
\theta\in\left(\frac{n-1}{n-2},\frac n{n-2}\right),
\]
is essential. It ensures that the inner region grows in the scaled
variables, while its physical size satisfies
\[
\lambda(t)R(t)
=
\sqrt{T-t}\,|\log(T-t)|^{\theta-\frac n{n-2}}
\ll \sqrt{T-t}.
\]
Thus the inner problem remains strictly inside the self-similar scale of
the shrinking sphere. The lower bound on $\theta$ is needed to close the
inner--outer estimates and the relevant time integrability.
}

\medskip

\subsection{Further results and conjectures}\label{sec-1.3}

The construction of Theorem \ref{thm} is nonradial in an essential way.
Indeed, the supercritical problem \eqref{eqn} lies in the
Matano--Merle range
\[
\frac{n+2}{n-2}<p<p_{JL},
\]
for $5\leq n\leq 12$,
where finite-time type II blow-up is ruled out for radial solutions,
under additional assumptions, in a ball or in the entire space; see
\cite{Merle04CPAM,Merle09JFA,Mizoguchi11JDE}. The solution constructed
here avoids this obstruction through the translation mode, which drives
the collapse of the concentration sphere.

A related construction gives type II blow-up on a sphere of fixed
radius. More precisely, under the same assumptions as in Theorem
\ref{thm}, there exist initial and boundary data such that the solution
of \eqref{eqn} blows up along a fixed \((n-4)\)-dimensional sphere:
\[
u(x,t)
\sim
\mu^{-1}(t)
U\left(
\frac{(r,z)-(\zeta_r(t),0)}
{\mu(t)}
\right),
\qquad t\nearrow T,
\]
where, for some \(\kappa_*>0\) and some fixed \(r_0>0\),
\[
\mu(t)
\sim
\kappa_*
\frac{T-t}{|\log(T-t)|^2},
\qquad
\zeta_r(t)\sim r_0.
\]
We do not give the details, since this follows by a simpler variant of
the present argument. In that case the concentration set does not enter
the self-similar region \(r=O(\sqrt{T-t})\), and the logarithmic
correction is the same as in the critical four-dimensional problem.

We close with a formal conjectural picture for analogous constructions
based on critical bubbles of other transverse dimensions. Consider
\[
u_t=\Delta u+|u|^{\frac4{k-2}}u
\qquad\mbox{in } \R^n,
\qquad
3\leq k<n.
\]
Writing
\[
x=(x^*,x^{**})
\in
\R^{\,n-k+1}\times\R^{\,k-1},
\qquad
\mathtt r=|x^*|,
\qquad
\mathtt z=x^{**},
\]
and imposing the corresponding cylindrical symmetry, the equation
reduces formally to
\[
v_t
=
\Delta_{(\mathtt r,\mathtt z)}v
+
\frac{n-k}{\mathtt r}v_{\mathtt r}
+
|v|^{\frac4{k-2}}v.
\]
One may then ask whether there exist solutions concentrating along a
shrinking \((n-k)\)-dimensional sphere whose radius is of order
\(\sqrt{T-t}\).

The formal reason for the rates below is the following. A critical
bubble in \(\R^k\) satisfies
\[
U_\lambda(x)
=
\lambda^{-\frac{k-2}{2}}
U\left(\frac{x-\xi}{\lambda}\right),
\qquad
U_\lambda(x)
\sim
\lambda^{\frac{k-2}{2}}|x-\xi|^{2-k}
\]
away from the core. Hence the tail of the scaling error has the form
\[
\dot\lambda\,\partial_\lambda U_\lambda
\sim
\dot\lambda\,
\lambda^{\frac{k-4}{2}}
|x-\xi|^{2-k}.
\]
After heat evolution, the Newtonian tail produces a memory term of size
\[
\dot\lambda(t)
\lambda(t)^{\frac{k-4}{2}}
\int_{\lambda^2(t)}^{T-t}
\tau^{-\frac{k-2}{2}}\,d\tau.
\]
Thus the behavior depends on whether this integral is dominated by the
macroscopic endpoint, is logarithmic, or is dominated by the microscopic
endpoint.

For \(k=3\),
\[
\int_{\lambda^2}^{T-t}\tau^{-1/2}\,d\tau
\sim
\sqrt{T-t},
\]
and balancing with a nonzero leading constant gives
\[
\dot\lambda
\sim
-C\frac{\lambda^{1/2}}{\sqrt{T-t}},
\qquad
\lambda(t)\sim T-t.
\]
For \(k=4\), the integral is logarithmic:
\[
\int_{\lambda^2}^{T-t}\frac{d\tau}{\tau}
=
\log\left(\frac{T-t}{\lambda^2}\right),
\]
which is the borderline mechanism realized in Theorem \ref{thm}. For
\(k>4\),
\[
\int_{\lambda^2}^{T-t}
\tau^{-\frac{k-2}{2}}\,d\tau
\sim
\lambda^{-(k-4)}.
\]
The reduced equation then becomes formally
\[
\dot\lambda
\sim
-C\lambda^{\frac{k-4}{2}}.
\]
This yields finite-time concentration for \(k=5\), with
\[
\lambda(t)\sim (T-t)^2,
\]
while \(k=6\) is the threshold case,
\[
\lambda(t)\sim e^{-Ct},
\]
and \(k>6\) gives infinite-time algebraic concentration,
\[
\lambda(t)\sim t^{-\frac{2}{k-6}}.
\]
The resulting formal picture is summarized in the following table:
\[
\begin{array}{c|c|c|c}
k
&
\text{formal reduced equation}
&
\text{conjectural scale}
&
\text{expected regime}
\\[0.3em]
\hline
3
&
\displaystyle
\dot\lambda
\sim
-C\frac{\lambda^{1/2}}{\sqrt{T-t}}
&
\displaystyle
T-t
&
\text{finite-time type II}
\\[1.2em]
4
&
\text{logarithmic memory}
&
\displaystyle
\frac{T-t}{|\log(T-t)|^\alpha}
&
\text{finite-time type II}
\\[1.2em]
5
&
\displaystyle
\dot\lambda
\sim
-C\lambda^{1/2}
&
\displaystyle
(T-t)^2
&
\text{finite-time type II}
\\[1.2em]
6
&
\displaystyle
\dot\lambda
\sim
-C\lambda
&
\displaystyle
e^{-Ct}
&
\text{infinite-time concentration}
\\[1.2em]
k>6
&
\displaystyle
\dot\lambda
\sim
-C\lambda^{\frac{k-4}{2}}
&
\displaystyle
t^{-\frac{2}{k-6}}
&
\text{infinite-time concentration}
\end{array}
\]

\medskip
Here \(C>0\), and the table should be understood only at the level of
formal leading-order modulation. The case \(k=4\), realized in Theorem
\ref{thm}, is the logarithmic borderline case. In the present geometry,
the logarithmic exponent is determined by the ambient dimension and is
\[
\alpha=\frac n{n-2}.
\]

\medskip

This paper will be organized as follows. In Section \ref{Sec-Approximation}, we analyze the error of the first approximation and improve it by adding corrections. We set up the gluing scheme in Section \ref{sec-inneroutergluingscheme} and develop its linear theories in Section \ref{sec-linearouter} and Section \ref{sec-linearinner}. The leading dynamics of $\la(t)$ and $\xi_r(t)$ are derived in Section \ref{sec-laxi} through corresponding orthogonality conditions. In Section \ref{sec-innerouter}, we solve the full nonlinear coupled system. The proofs of some technical ingredients are postponed to Appendix \ref{Appendix-outer} and Appendix \ref{App-B}.

\medskip


\medskip

\section{Approximate solutions and error estimates}\label{Sec-Approximation}

\medskip

In the symmetry class \eqref{symmetryclass}, problem \eqref{eqn} becomes
$$u_t=u_{rr}+\frac{n-4}{r}u_r+\Delta_{z} u+u^3,$$
where $(r,z) \in \mathcal D$ with $\mathcal D$ defined in \eqref{def-domain}. Define the error operator
\begin{equation*}
\mathcal S (u):=-u_t+\Delta_{(r,z)} u+\frac{n-4}{r}u_r+u^3,
\end{equation*}
where $\Delta_{(r,z)}:=\partial_r^2+\Delta_z$ is the Laplacian in $\R^4$. Our first approximate solution is based on the Aubin--Talenti bubble 
$$
U(y)=\frac{\alpha_{0}}{1+|y|^2},
$$ 
which solves the Yamabe problem
\begin{equation*}
\Delta_y U+U^3=0~\mbox{ in }~\R^4.
\end{equation*}
Here $\alpha_0=2\sqrt{2}.$ It is well-known that the linearized operator around the bubble
\begin{equation}\label{def-L0}
L_0(\phi):=\Delta \phi+3U^2\phi
\end{equation}
is non-degenerate in the sense that all bounded solutions to $L_0(\phi)=0$ are the linear combination of
\begin{equation}\label{kernels}
Z_i(y):=\partial_{y_i} U(y),~~i=1,2,3,4,~Z_5(y):=U(y)+\nabla U(y)\cdot y.
\end{equation}

\subsection{First approximate solution}
We write
\[
U_{\la(t),\xi(t)}(r,z)
=
\la^{-1}(t)
U\left(
\frac{(r,z)-\xi(t)}{\la(t)}
\right),
\]
where
\[
\xi(t)=(\xi_r(t),0),
\qquad
\xi_r(t)=\xi_{r,*}(t)+\xi_{r,1}(t),
\qquad
\xi_{r,1}(t)=o(\xi_{r,*}(t)).
\]
In the sequel, we denote
\[
y=\frac{(r,z)-\xi(t)}{\la(t)}.
\]

Then the first error of $U_{\la(t),\xi(t)}$ is
\begin{equation}\label{firsterror}
\begin{aligned}
\mathcal S(U_{\la(t),\xi(t)})=&~- \partial_t U_{\la,\xi}+\frac{n-4}{r} \partial_r  U_{\la,\xi}\\
=&~  \la^{-2}(t)\dot{\la}(t)\left(-\frac{\alpha_0}{1+|y|^2}+\frac{2\alpha_0}{(1+|y|^2)^2}\right)+\la^{-2}(t)\nabla U(y)\cdot \dot{\xi}(t)+\frac{n-4}r \partial_r  U_{\la,\xi},
\end{aligned}
\end{equation}
where $\nabla:=(\partial_r, \nabla_z)$.

\subsection{Corrected approximate solution}

Observe that the slow decaying error in \eqref{firsterror} is
\begin{equation}\label{def-mE0}
\mathcal E_0=-\frac{\alpha_0\dot\la(t)}{\la^2(t)+\rho^2}\approx -\frac{\alpha_0\dot \la(t)}{\rho^2},
\end{equation}
where $\rho=|(r,z)-\xi(t)|.$
In order to reduce the size of the first error, we shall choose $\Psi_0(r,z,t)$ that solves
\begin{equation*}
\left\{
    \begin{aligned}
    &~\partial_t \Psi_0=\partial_{rr}\Psi_0+\frac{n-4}{r}\partial_r \Psi_0+\Delta_z \Psi_0-\frac{\alpha_0\dot \la(t)}{\rho^2}~&\mbox{ in }~\R_+\times\R^3\times(0,T),\\
    &~\Psi_0(r,z,0)=0~&\mbox{ in }~\R_+\times\R^3.
    \end{aligned}
\right.
\end{equation*}
By the axisymmetric heat kernel derived in Appendix \ref{App-heatkernel}, one can express
$$
\begin{aligned}
\Psi_0(r,z,t)=-2\sqrt2\int_0^t \int_{\R^3}\int_{\R_+}\Gamma_n(r,z;\tilde r,\tilde z;t-s)\frac{\dot\la(s)}{|(\tilde r,\tilde z)-\xi(s)|^2} d\tilde r d\tilde z ds,
\end{aligned}
$$
where $\Gamma_n$ is defined in \eqref{axi-heatkernel}.

Now we choose the corrected approximation as
$$u^*=U_{\la(t),\xi(t)}+ \Psi_0$$
and compute the error
$$
\begin{aligned}
\mathcal S(u^*)
=&~\mathcal S(U_{\la(t),\xi(t)})+\frac{\alpha_0\dot\la(t)}{\rho^2}+(U_{\la(t),\xi(t)}+\Psi_0)^3-U_{\la(t),\xi(t)}^3\\
=&~ \frac{2\alpha_0\la^{-2}(t)\dot{\la}(t)}{(1+|y|^2)^2}+\la^{-2}(t)\nabla U(y)\cdot \dot{\xi}(t)+\frac{\alpha_0\dot\la(t)\la^2(t)}{\rho^2(\la^2+\rho^2)}+\frac{n-4}r \partial_r U_{\la,\xi}\\
&~+(U_{\la,\xi}+\Psi_0)^3-U_{\la,\xi}^3\\
=&~\mathcal K[\la,\xi]+(U_{\la,\xi}+\Psi_0)^3-U_{\la,\xi}^3,
\end{aligned}
$$
where $\mathcal K[\la,\xi]$ is defined as
\begin{equation}\label{def-mK}
\begin{aligned}
\mathcal K[\la,\xi]:=&~\frac{2\alpha_0 \la^{-2}(t)\dot{\la}(t)}{(1+|y|^2)^2}+\frac{\alpha_0\dot\la(t)\la^2(t)}{\rho^2(\la^2+\rho^2)}+\la^{-2}(t)\nabla U(y)\cdot \dot{\xi}(t)\\
&~+\frac{n-4}{\la(t)y_1+\xi_r(t)}\la^{-2}(t)\partial_{y_1}U(y).
\end{aligned}
\end{equation}
We define
\begin{equation}\label{def-Sout}
\begin{aligned}
\mathcal S_{\rm out}[\la,\xi]:=&~(1-\eta_R)\left(\frac{\alpha_0\dot\la(t)\la^2(t)}{\rho^2(\la^2+\rho^2)}+\la^{-2}\nabla U(y)\cdot \dot{\xi}+\frac{n-4}{\la y_1+\xi_r }\la^{-2} \partial_{y_1}U(y)\right).
\end{aligned}
\end{equation}
Here
$$
\eta_R=\eta_{R(t)}(r,z,t)=\eta\left(\frac{|(r,z)-\xi(t)|}{\la(t)R(t)}\right),
$$
the smooth cut-off function $\eta$ is defined by 
\begin{equation}\label{def-cutoff}
\eta(s)=\begin{cases}
1,~&s<1,\\
0,~&s>2,
\end{cases}
\end{equation}
and $R(t)$ will be specified later. To further reduce the size of the error, we introduce the leading orders of $\la$ and $\xi$
$$\la_*=\frac{2}{n-2}\frac{|\log T|^{\frac{2}{n-2}}(T-t)}{|\log(T-t)|^{\frac{n}{n-2}}},\quad \xi_*=(\xi_{r,*},~\xi_{z,*})=(\sqrt{2(n-4)(T-t)},~z_0),$$
which will be derived in Section \ref{sec-laxi}. Here $z_0:=(0,0,0)\in\R^3$. Let $\Psi_1$ be the solution solving
\begin{equation}\label{def-Psi1}
\begin{cases}
\partial_t\Psi_1=\Delta_{(r,z)} \Psi_1 +\frac{n-4}{r}\partial_r\Psi_1+\mathcal S_{\rm out}[\la_*,\xi_*],~&\mbox{ in }~\mathcal D\times (0,T)\\
\Psi_1=0,~&\mbox{ on }~(\partial\mathcal D\backslash \{r=0\})\times (0,T)\\
\partial_r \Psi_1=0,~&\mbox{ on }~(\mathcal D\cap \{r=0\})\times (0,T)\\
\Psi_1(r,z,0)=0,~&\mbox{ in }~\mathcal D\\
\end{cases}
\end{equation}
where $\mathcal S_{\rm out}[\la_*,\xi_*]$ is defined by replacing $\la$, $\xi$ in $\mathcal S_{\rm out}[\la,\xi]$
by $\la_*$ and $\xi_*$, respectively. Note that the new error produced by $\Psi_1$ turns out to be of smaller order and will not change the leading order terms $\la_*$, $\xi_*$, and in fact, this is the reason for choosing $R(t)$ above. We shall show this in Section \ref{sec-innerouter}.

In conclusion, the corrected approximation we finally choose is
$U^*+\Psi_0+\Psi_1.$
In the sequel, we shall find a perturbation $\mathtt P$ such that $u=U_{\la,\xi}+\Psi_0+\Psi_1+\mathtt P$ is the desired solution, namely,
$$\mathcal S(U_{\la,\xi}+\Psi_0+\Psi_1+\mathtt P)=0.$$


\medskip

\section{The inner--outer gluing scheme}\label{sec-inneroutergluingscheme}

\medskip

We look for solution of the form
$$u=U_{\la(t),\xi(t)}+\Psi_0(r,z,t)+\Psi_1(r,z,t)+\mathtt P,$$
where $\mathtt P$ is a small perturbation consisting of inner and outer parts
\begin{align*}
&~\mathtt P=\la^{-1}(t)\eta_R \phi(y,t)+\psi(r,z,t)+Z^*(x,t),
\end{align*}
where $Z^*$ satisfies
$$  
\begin{cases}
Z_t^*=\Delta_x Z^*,~&\mbox{ in }\Omega\times (0,T)\\
Z^*(\cdot,t)=0,~&\mbox{ on }\partial \Omega\times (0,T)\\
Z^*(\cdot,0)=Z_0^*,~&\mbox{ in }\Omega
\end{cases}
$$ 
in the original variables $x\in\R^n$. Throughout the paper, we choose $R(t)$ such that $\la(t)R(t)\ll \sqrt{T-t}$ for $T\ll 1$.
Denote
$$
\mathcal D_{2R}=\left\{(r,z)\in \R^4:|(r,z)-(\xi_r,0)|\leq 2\la R\right\}$$
and $\Psi^*=\psi+Z^*$. Then $u$ is a solution to the original problem \eqref{eqn} if

\medskip

\noindent $\bullet$ $\phi$ solves the {\em inner problem}
\begin{equation}\label{inner}
\begin{aligned}
\la^2 \phi_t = \Delta_y \phi +3U^2(y)\phi+\mathcal H(\phi,\psi,\la,\xi)~~\mbox{ in }\mathcal D_{2R}\times (0,T)
\end{aligned}
\end{equation}
where
\begin{align} \notag
\mathcal H(\phi,\psi,\la,\xi)(y,t):=&~3\la U^{2}(y)[\Psi_0+\psi+Z^*](\la y+\xi,t)+\frac{(n-4)\la}{\la y_1+\xi_r}\phi_{y_1}\\ \notag
&~+\la\left[\dot\la (\nabla_y\phi\cdot y+\phi)+\nabla_y \phi\cdot\dot\xi\right]+\frac{(n-4)\la^2}{r}\phi\partial_r \eta_R\\ \label{def-mH}
&~+\la^3 \mathcal N(\mathtt P+\Psi_0+\Psi_1)+\la^3 \mathcal K[\la,\xi]+3\la U^2(y)\Psi_1
\end{align}
with $\mathcal K[\la,\xi]$ defined in \eqref{def-mK}.

\medskip

\noindent $\bullet$  $\psi$ solves the {\em outer problem}
\begin{equation}\label{outer}
\psi_t =\Delta_{(r,z)} \psi+\frac{n-4}{r}\partial_r \psi  + \mathcal G(\phi,\psi,\la,\xi)~~\mbox{ in }\mathcal D\times (0,T)
\end{equation}
with
\begin{equation}\label{def-mG}
\begin{aligned}
\mathcal G(\phi,\psi,\la,\xi):=&~3\la^{-2}(1-\eta_R)U^2(y)(\psi+Z^*+\Psi_0+\Psi_1)\\
&~+\la^{-3}\left[(\Delta_y \eta_R) \phi+2\nabla_y\eta_R\cdot\nabla_y \phi-\la^2\phi\partial_t \eta_R\right]\\
&~+(1-\eta_R)\mathcal K_1[\la,\xi]+\mathcal S_{\rm out}[\la,\xi]-\mathcal S_{\rm out}[\la_*,\xi_*]+(1-\eta_R)\mathcal N(\mathtt P+\Psi_0+\Psi_1),
\end{aligned}
\end{equation}
where
\begin{equation}\label{def-mK1}
\mathcal K_1[\la,\xi]:=\frac{2\alpha_0 \la^{-2}(t)\dot{\la}(t)}{(1+|y|^2)^2},
\end{equation}
$\mathcal S_{\rm out}$ is defined in \eqref{def-Sout} and
$$\mathcal N(\mathtt P+\Psi_0+\Psi_1):=(U_{\la,\xi}+\mathtt P+\Psi_0+\Psi_1)^3-U_{\la,\xi}^3-3U_{\la,\xi}^2(\mathtt P+\Psi_0+\Psi_1).
$$ 

\medskip
We now describe our strategy to solve the inner and outer problems. We shall first develop linear theories for the associated linear problems of \eqref{outer} and \eqref{inner}. Since the solution we want to construct concentrates on an $(n-4)$-dimensional sphere with shrinking size $\sqrt{T-t}$, suitable estimates for the outer solution $\psi$ are very delicate to find. To achieve this, we find solutions in the symmetry class \eqref{symmetryclass} by using the Duhamel's formula in $\R^n$. For the linear inner problem, we want to find inner solution with proper space-time decay.
 Since the inner--outer gluing relies on delicate analysis of the space-time decay of solutions, we shall further decompose the inner problem \eqref{inner} into three different spherical harmonic modes and construct solution in each mode. To get more refined estimates for the gluing to work, we carry out a new inner--outer gluing scheme for the linear inner problem, where certain orthogonality conditions are of course needed due to the existence of the nontrivial kernels (see \eqref{kernels}) of the linearized operator $L_0$ in \eqref{def-L0}. This will give us the reduced equations for the parameter functions $\la(t)$ and $\xi_r(t)$. 
 The reduced equation for $\xi_r(t)$ is explicit.
 However, the reduced equation for $\la(t)$ turns out to be an integro-differential equation due to the non-local correction $\Psi_0$ in \eqref{def-Psi0}, and it is more involved. Finally, by using the Schauder fixed point theorem, we solve the inner--outer gluing system and prove the existence of the desired blow-up solution.

\medskip

The rest of the paper is organized as follows. In Section \ref{sec-laxi}, we derive the leading orders for the parameter functions $\la(t)$ and $\xi(t)$. In Section \ref{sec-linearouter}, we establish the estimates for the linear outer problem with different right hand sides which appear in $\mathcal G$ defined in \eqref{def-mG}. The proof is postponed to the Appendix. In Section \ref{sec-linearinner}, we develop the linear theory for the inner problem by spherical harmonic decomposition. In Section \ref{sec-innerouter}, the inner--outer gluing system is formulated, and we shall solve $(\phi,\psi,\la,\xi_r)$ from the full system by the linear theories developed in Section \ref{sec-linearouter}, Section \ref{sec-linearinner} and the Schauder fixed point theorem.

\bigskip
\noindent{{\bf Notation}}. \, Throughout the paper, we shall use the symbol  $``\,\lesssim\,"$ to denote $``\,\leq\, C\,"$ for a positive constant $C$ independent of $t$ and $T$. Here $C$ might be different from line to line.


\medskip

{

\section{The choices of \texorpdfstring{$\lambda_*$}{lambda} and \texorpdfstring{$\xi_*$}{xi}}\label{sec-laxi}

\medskip

In this section, we shall choose the leading orders $\la_*(t)$, $\xi_*(t)=(\xi_{r,*}(t), 0)$ of the parameter functions $\la(t)$ and $\xi(t)$. In Section \ref{sec-linearinner}, a linear theory for the inner problem concerning the solvability and estimates of the associated linear problem will be developed, where approximately the following orthogonality conditions
\begin{equation}\label{oc}
\int_{\R^4} \mathcal H(\phi,\psi,\la,\xi) Z_j (y) dy=0~~\mbox{ for all } j=1,\cdots,5,~  t\in (0,T)
\end{equation}
are needed to guarantee the existence of an inner solution $\phi$ with desired space-time decay. Here $Z_j$ are the kernel functions (c.f. \eqref{kernels}) of the linearized operator $L_0$. Basically, we will derive the scaling and translation parameters $\la(t)$ and $\xi(t)$ at main order from the orthogonality conditions \eqref{oc}.

Recall $\mathcal H(\phi,\psi,\la,\xi)(y,t)$ defined in \eqref{def-mH}. In this section, we shall single out the leading term $\mathcal H_*$ in $\mathcal H$ to derive $\la_*$ and $\xi_*$. We define
\[
\mathcal H_*[\la,\xi,\Psi^*]
:=3\la U^{2}(y)[\Psi_0+\Psi^*](\la y+\xi,t)+\la^3 \mathcal K[\la,\xi],
\]
where $\Psi^*=\psi+Z^*$. The contribution of the rest terms $\mathcal H-\mathcal H_*$ in the orthogonality conditions turns out to be negligible compared to the leading term $\mathcal H_*$. We shall deal with this in Section \ref{sec-innerouter} when we finally solve the inner--outer gluing system.

Then
\[
\int_{\mathbb R^4} \mathcal H_*[\la,\xi,\Psi^*] Z_1(y)dy=0
\]
implies that
\[
\dot{\xi}_r(t) +\frac{n-4}{\xi_r(t)}=o(1),
\]
where $o(1)\to 0$ as $t\nearrow T.$ So the choice of $\xi_r(t)$ at main order is
\[
\xi_{r,*}(t)=\sqrt{2(n-4)(T-t)}.
\]
Next, we consider the reduced equation for $\la(t)$ from
\begin{equation}\label{oc0}
\int_{\R^4} \mathcal H_*[\la,\xi,\Psi^*] Z_5(y)dy=0.
\end{equation}
Recall that the correction $\Psi_0$, improving the slowly decaying error, solves the problem
\[
\partial_t \Psi_0-\partial_{rr}\Psi_0-\frac{n-4}{r}\partial_r \Psi_0-\Delta_{z}\Psi_0
=-2\sqrt2\frac{\dot\la(t)}{|(r,z)-\xi(t)|^2}
\]
in $\R^n\times(0,T)$ with zero initial data. Then using \eqref{axi-heatkernel} one has
\begin{equation}\label{def-Psi0}
\begin{aligned}
\Psi_0(r,z,t)=&~-2\sqrt2\int_0^t \int_{\R^3}\int_{\R_+}\Gamma_n(r,z;\tilde r,\tilde z;t-s)\frac{\dot\la(s)}{|(\tilde r,\tilde z)-\xi(s)|^2} d\tilde r d\tilde z ds\\
=&~-\frac{\sqrt2}{8\pi^{3/2}} r^{\frac{5-n}2}\int_0^t \frac{\dot\la(s)}{(t-s)^{5/2}}\exp\left\{-\frac{r^2}{4(t-s)}\right\} ds\int_{\R^3} \exp\left\{-\frac{|z-\tilde z|^2}{4(t-s)}\right\}d\tilde z\\
&~\qquad\int_{\R_+} \exp\left\{-\frac{\tilde r^2}{4(t-s)}\right\}\frac{\tilde r^{\frac{n-3}{2}}}{(\tilde r-c_n\sqrt{T-s})^2+|\tilde z|^2} I_{\frac{n-5}{2}}\left(\frac{r\tilde r}{2(t-s)}\right) d\tilde r.
\end{aligned}
\end{equation}
Here $I_{\frac{n-5}{2}}$ is the modified Bessel function, 
\[
|(\tilde r,\tilde z)-\xi(s)|^2=(\tilde r-c_n\sqrt{T-s})^2+|\tilde z|^2,
\]
and
\[
\xi(t)\sim\xi_*(t)=(c_n\sqrt{T-t},0)\in\R\times\R^3,
\quad c_n:=\sqrt{2(n-4)}.
\]

\subsection{The contribution of the global correction to the scaling mode}

We now explain the mechanism by which the global correction $\Psi_0$ determines the leading order law of $\lambda(t)$. This is the main new feature in the reduced scaling equation. We shall prove that
\begin{equation}\label{nonlocal-oc0}
\begin{aligned}
    &~\int_{B_{4R}} 3 U^2(y_1,y') \Psi_0\big(\la y_1+c_n\sqrt{T-t},\la y',t\big)  Z_5(y_1,y')dy_1 dy'\\
    =&~C\left(\int_0^{t-(T-t)} \frac{\dot\la(s)}{t-s}\,ds
    +c_n^*\int_{t-(T-t)}^{t-\la^2(t)} \frac{\dot\la(s)}{t-s}\,ds\right)+O(|\dot\la(t)|),
\end{aligned}
\end{equation}
where $C>0$ and
$$
c_n^*=\frac{n-2}{2}.
$$
The proof of the precise asymptotic expansion is given in Appendix \ref{App-nonlocal}. We record here the leading order computation, since it explains the origin of the logarithmic correction in the blow-up rate.

The two asymptotic regimes of the modified Bessel function
\begin{equation}\label{Bessel-main}
    I_{\frac{n-5}{2}}(z)\sim
    \begin{cases}
        \displaystyle
        \frac{z^{\frac{n-5}{2}}}{2^{\frac{n-5}{2}}\Gamma\left(\frac{n-3}{2}\right)},
        & z\to 0,\\[1ex]
        \displaystyle
        \frac{1}{\sqrt{2\pi}}\frac{e^z}{\sqrt z},
        & z\to \infty,
    \end{cases}
\end{equation}
are responsible for the two different coefficients in \eqref{nonlocal-oc0}. Substituting the representation formula \eqref{def-Psi0} into the projection onto $Z_5$, and using $\xi_r(t)\sim c_n\sqrt{T-t}$, we obtain, up to a nonzero dimensional constant,
\begin{align*}
&\int_{B_{4R}} 3U^2(y)Z_5(y)
\Psi_0(\la y_1+c_n\sqrt{T-t},\la y',t)\,dy  \\
&\sim (T-t)^{\frac{5-n}{4}}
\int_0^t \frac{\dot\la(s)}{(t-s)^{3/2}}\,ds
\int_{\mathbb R^3}\int_{-c_n\sqrt{\frac{T-s}{t-s}}}^{\infty}
\mathcal A(\hat r,\hat z,s,t)\,d\hat r\,d\hat z ,
\end{align*}
where
\[
\hat r=\frac{\tilde r-c_n\sqrt{T-s}}{\sqrt{t-s}},
\qquad
\hat z=\frac{\tilde z}{\sqrt{t-s}}.
\]
The relevant point is that the argument of the Bessel function is
\[
\frac{c_n\sqrt{T-t}\,
(\sqrt{t-s}\hat r+c_n\sqrt{T-s})}{2(t-s)}.
\]
Hence the regimes $t-s\gtrsim T-t$ and $\lambda^2(t)\ll t-s\ll T-t$ lead to different Bessel asymptotics.

We split the time integral as
\[
\int_0^t=\int_0^{t-(T-t)}
+\int_{t-(T-t)}^{t-\lambda^2(t)}
+\int_{t-\lambda^2(t)}^t
=:I+J+K.
\]
In the first region the small-argument asymptotic in \eqref{Bessel-main} gives
\[
I=C\int_0^{t-(T-t)}\frac{\dot\lambda(s)}{t-s}\,ds
+O(|\dot\lambda(t)|).
\]
In the intermediate region, the large-argument asymptotic applies. The exponential factor in $I_{\frac{n-5}{2}}$ cancels the corresponding part of the Gaussian kernel, and the remaining axisymmetric Jacobian produces the coefficient
\[
c_n^*=\frac{n-2}{2}.
\]
Consequently,
\[
J=Cc_n^*
\int_{t-(T-t)}^{t-\lambda^2(t)}
\frac{\dot\lambda(s)}{t-s}\,ds
+O(|\dot\lambda(t)|).
\]
Finally, the microscopic region $t-\lambda^2(t)<s<t$ contributes only
\[
K=O(|\dot\lambda(t)|).
\]
Combining these three estimates gives \eqref{nonlocal-oc0}.

Therefore, \eqref{oc0} is reduced to
\begin{equation}\label{oc5''''}
C\left(\int_0^{t-(T-t)} \frac{\dot\la(s)}{t-s}\,ds
+c_n^*\int_{t-(T-t)}^{t-\la^2(t)} \frac{\dot\la(s)}{t-s}\,ds\right)
=-c_0[Z_0^*(q)+\psi(q,0)]+o(1),
\end{equation}
where
\[
c_0:=3\int_{\R^4} U^2(y) Z_5(y) dy<0.
\]
Since $\la(t)$ decreases to $0$ as $t\nearrow T,$
we impose
\[
a_*:=Z_0^*(q)+\psi(q,0)<0.
\]
Now we claim that a good choice of $\la(t)$ at main order is
$$
\dot{\la}(t)=-\frac{c}{|\log(T-t)|^{\frac{n}{n-2}}},
$$ 
where $c>0$ is a constant to be determined later. Indeed, by approximation we get
\begin{align*}
        &~\int_0^{t-(T-t)} \frac{\dot\la(s)}{t-s}\,ds
        +c_n^*\int_{t-(T-t)}^{t-\la^2(t)} \frac{\dot\la(s)}{t-s}\,ds\\
        =&~\int_0^{t-(T-t)} \frac{\dot\la(s)}{t-s}\,ds
        +c_n^*\int_{t-(T-t)}^{t-\la^2(t)} \frac{\dot\la(t)}{t-s}\,ds
        +c_n^*\int_{t-(T-t)}^{t-\la^2(t)} \frac{\dot\la(s)-\dot\la(t)}{t-s}\,ds\\
        =&~\int_0^{t-(T-t)} \frac{\dot\la(s)}{t-s}\,ds
        +c_n^*\dot\la(t)[\log(T-t)-2\log\la(t)]
        +c_n^*\int_{t-(T-t)}^{t-\la^2(t)} \frac{\dot\la(s)-\dot\la(t)}{t-s}\,ds\\
        \approx&~\int_0^t \frac{\dot\la(s)}{T-s}\,ds-c_n^*\dot\la(t)\log(T-t):=\gamma(t).
\end{align*}
Then
\begin{align*}
    |\log(T-t)|^{\frac1{c_n^*}}\dot\gamma(t)
    =&~c_n^*\left(\ddot\la(t)|\log(T-t)|^{\frac{c_n^*+1}{c_n^*}}
    +\frac{c_n^*+1}{c_n^*}\frac{\dot\la(t)}{T-t} |\log(T-t)|^{\frac1{c_n^*}}\right)\\
    =&~c_n^*\frac{d}{dt}\left(\dot\la(t)|\log(T-t)|^{\frac{c_n^*+1}{c_n^*}}\right)\sim 0.
\end{align*}
Therefore, $\la(t)$ can be well approximated by
\[
    \dot{\la}(t)=-\frac{c}{|\log(T-t)|^{\frac{c_n^*+1}{c_n^*}}}
    =-\frac{c}{|\log(T-t)|^{\frac{n}{n-2}}},
    \quad n\geq 5.
\]
The constant $c$ is chosen so that the leading order part of \eqref{oc5''''} is satisfied. Equivalently,
\[
-c\int_{-T}^T \frac{ds}{(T-s)|\log(T-s)|^{\frac{n}{n-2}}}
=\kappa_*:=-\frac{c_0 a_*}{c_n^*C}.
\]
At main order, we obtain
\[
\dot{\la}(t)=\kappa_*\dot{\la}_*(t)
\]
with
\[
\dot{\la}_*(t)=-\frac{2}{n-2}\frac{|\log T|^{\frac{2}{n-2}}}{|\log(T-t)|^{\frac{n}{n-2}}}.
\]
By imposing $\la_*(T)=0$, we finally get
\[
\la_*(t)=\frac{2}{n-2}\frac{|\log T|^{\frac{2}{n-2}}(T-t)}{|\log(T-t)|^{\frac{n}{n-2}}}(1+o(1))~\mbox{ as }~t\nearrow T.
\]
}

\medskip

{
}

\medskip

\section{Linear theory for the outer problem}\label{sec-linearouter}

\medskip

In order to solve the outer problem \eqref{outer}, we need a linear theory for the associated linear problem. We consider
\begin{equation}\label{outer-linear}
\begin{cases}
\psi_t=\Delta_{(r,z)} \psi+\frac{n-4}{r}\partial_r \psi +\mathbf {f}_{\rm out},~&\mbox{ in }\mathcal D\times (0,T)\\
\psi=0,~&\mbox{ on }(\partial \mathcal D\backslash \{r=0\})\times (0,T)\\
\psi_r=0,~&\mbox{ on } (\mathcal D\cap\{r=0\})\times (0,T)\\
\psi(r,z,0)=0,~&\mbox{ in }\mathcal D\\
\end{cases}
\end{equation}
where the non-homogeneous term $\mathbf f_{{\rm out}}$ in \eqref{outer-linear} is assumed to be bounded with respect to the weights appearing in the outer problem \eqref{outer}. Define the weights
\begin{equation}\label{def-weights}
\left\{
\begin{aligned}
&\varrho_1:=\la_*^{\nu-3}(t)R^{-2-\alpha}(t)\chi_{\{|(r,z)-\xi(t)|\leq 2\la_* R\}}\\
&\varrho_2:=\frac{\la_*^{\nu_2}}{|(r,z)-\xi(t)|^2}\chi_{\{\la_* R \leq|(r,z)-\xi(t)|\}}\\
&\varrho_3:=1\\
\end{aligned}
\right.
\end{equation}
where we choose $R=(T-t)^{-1/2}|\log(T-t)|^{\theta}$ with $\theta\in\left(\frac{n-1}{n-2},\frac{n}{n-2}\right)$.

Define the norms
\begin{equation}\label{def-norm**}
\|f\|_{**}:=\sup_{(r,z,t)\in \mathcal D\times(0,T)} \left(\sum\limits_{i=1}^3 \varrho_i(r,z,t)\right)^{-1}|f(r,z,t)|
\end{equation}

\begin{equation}\label{def-norm*}
\begin{aligned}
\|\psi\|_{*}:=&~\frac{\la_*^{-\nu}(0)R^{\alpha-2}(0)}{|\log T|}\|\psi\|_{L^{\infty}(\Omega\times(0,T))}+\frac{\la_*^{\frac12-\nu}(0)R^{\alpha-2}(0)}{|\log T|}\|\nabla\psi\|_{L^{\infty}(\Omega\times(0,T))}\\
&~+\sup_{(r,z,t)\in \mathcal D\times(0,T)}\left[ \frac{\la_*^{-\nu}(t)R^{\alpha-2}(t)}{|\log(T-t)|}\left|\psi(r,z,t)-\psi(r,z,T)\right|\right]\\
&~+\sup_{(r,z,t)\in \mathcal D\times(0,T)}\left[\frac{\la_*^{\frac12-\nu}(t)R^{\alpha-2}(t)}{|\log (T-t)|}\left|\nabla\psi(r,z,t)-\nabla\psi(r,z,T)\right|\right]\\
&~+\sup_{\mathcal D\times I_T} \frac{\la_*^{2-\nu+\gamma}(t_2)R^{2+\alpha}(t_2)}{(t_2-t_1)^{\gamma}} |\psi(r,z,t_2)-\psi(r,z,t_1)|,
\end{aligned}
\end{equation}
where $y=\frac{(r,z)-\xi(t)}{\la(t)}$, $\nu,\alpha\in(0,1)$, $\gamma\in(0,1)$,  and the last supremum is taken over
$$\mathcal D\times I_T=\left\{(r,z,t_1,t_2): ~(r,z)\in\mathcal D, ~0\leq t_1\leq t_2\leq T,~t_2-t_1\leq \frac{1}{10}(T-t_2)\right\}.$$
For problem \eqref{outer-linear}, we have the following proposition.
\begin{prop}\label{outer-apriori}
Let $\psi$ be the solution to problem \eqref{outer-linear} with $\|\mathbf {f}_{\rm out}\|_{**}<+\infty$. Then it holds that
\begin{equation}\label{est-outer-apriori}
\|\psi\|_*\lesssim \|\mathbf {f}_{\rm out}\|_{**}.
\end{equation}
\end{prop}
In order to establish Proposition \ref{outer-apriori}, we consider
\begin{equation}\label{outer-linear''}
\begin{cases}
\psi_t=\Delta_{\R^n} \psi +f,~&\mbox{ in }\Omega\times (0,T)\\
\psi=0,~&\mbox{ on }\partial \Omega\times (0,T)\\
\psi(x,0)=0,~&\mbox{ in }\Omega\\
\end{cases}
\end{equation}
which is defined in $\Omega\subset\R^n$ in the symmetry class \eqref{symmetryclass}. For problem \eqref{outer-linear''}, we prove the following three lemmata concerning the a priori estimates with different right hand sides.

\begin{lemma}\label{lemma-rhs1}
Let $\psi$ solve problem \eqref{outer-linear''} with right hand side
$$|f(x,t)|\lesssim \la_*^{\nu-3}(t)R^{-2-\alpha}(t)\chi_{\{|(r,z)-\xi(t)|\leq 2\la_* R\}}.$$ If
$$\nu-3+\frac12(2+\alpha)<0,~\nu-1+\frac{\alpha}2>0,$$
then
\begin{equation}\label{outer-rhs1}
|\psi(x,t)|\lesssim\la_*^{\nu}(0) R^{2-\alpha}(0)|\log T|,
\end{equation}
\begin{equation}\label{outerT-rhs1}
|\psi(x,t)-\psi(x,T)|\lesssim \la_*^{\nu}(t) R^{2-\alpha}(t)|\log(T-t)|,
\end{equation}
\begin{equation}\label{outergradient-rhs1}
|\nabla \psi(x,t)|\lesssim\la_*^{\nu-\frac12}(0)R^{2-\alpha}(0)|\log T|,
\end{equation}
\begin{equation}\label{outergradientT-rhs1}
|\nabla \psi(x,t)-\nabla \psi(x,T)|\lesssim\la_*^{\nu-\frac12}(t)R^{2-\alpha}(t)|\log(T-t)|,
\end{equation}
and
\begin{equation}\label{outerholder-rhs1}
|\psi(x,t_2)-\psi(x,t_1)|\lesssim  \la_*^{\nu+\frac{\mu}{2}-3}(t_2)R^{-2-\alpha}(t_2) (t_2-t_1)^{1-\mu/2},
\end{equation}
where $0\leq t_1\leq t_2\leq T$ with $t_2-t_1\leq \frac{1}{10}(T-t_2)$ and $\mu\in(0,1)$.
\end{lemma}

\begin{lemma}\label{lemma-rhs2}
Let $\psi$ solve problem \eqref{outer-linear''} with right hand side
$$|f(x,t)|\lesssim \frac{\la_*^{\nu_2}}{|(r,z)-\xi(t)|^2}\chi_{\{\la_* R\leq|(r,z)-\xi(t)|\}},$$
where $\nu_2\in(0,1)$, $\delta>0$.
Then
$$
|\psi(x,t)|\lesssim  \la_*^{\nu_2-1}(0)R^{-2}(0)|\log T|,
$$ 
$$
|\psi(x,t)-\psi(x,T)|\lesssim  \la_*^{\nu_2-1}(t)R^{-2}(t)|\log (T-t)|,
$$ 
$$|\nabla \psi(x,t)|\lesssim \la_*^{\nu_2-2}(t)R^{-2}(t)\sqrt{t},
$$ 
$$|\nabla \psi(x,t)-\nabla \psi(x,T)|\lesssim \la_*^{\nu_2-\frac32}(t)R^{-2}(t)|\log (T-t)|,
$$ 
and
$$ 
|\psi(x,t_2)-\psi(x,t_1)|\lesssim \la_*^{\nu_2-1-\gamma}(t_2) R^{-2}(t_2) (t_2-t_1)^{\gamma},
$$ 
where $0\leq t_1\leq t_2\leq T$ with $t_2-t_1\leq \frac{1}{10}(T-t_2)$ and $\gamma\in(0,1)$.
\end{lemma}

\begin{lemma}\label{lemma-rhs3}
Let $\psi$ solve problem \eqref{outer-linear''} with right hand side
$$|f(x,t)|\lesssim 1.$$
Then
$$|\psi(x,t)|\lesssim t,
$$ 
$$|\psi(x,t)|\lesssim (T-t)|\log(T-t)|,
$$
$$ 
|\nabla\psi(x,t)|\lesssim T^{1/2},
$$ 
$$ 
|\nabla\psi(x,t)-\nabla\psi(x,T)|\lesssim (T-t)^{1/2},
$$ 
and
$$|\psi(x,t_2)-\psi(x,t_1)|\lesssim (t_2-t_1)|\log (t_2-t_1)|,
$$ 
where $0\leq t_1\leq t_2\leq T$ with $t_2-t_1\leq \frac{1}{10}(T-t_2)$.
\end{lemma}

\medskip

\begin{proof}[Proof of Proposition \ref{outer-apriori}]
We denote $\psi[\mathbf {f}_{\rm out}]$ by the solution to problem \eqref{outer-linear''} with the right hand side $\mathbf {f}_{\rm out}$ satisfying $\|\mathbf {f}_{\rm out}\|_{**}<+\infty$. Decompose
$$\mathbf {f}_{\rm out}=\sum\limits_{i=1}^3 f_i~\mbox{ with }~|f_i|\lesssim \varrho_i\|f_i\|_{**}.$$
Let $1-\frac{\mu}{2}=\gamma$ in Lemma \ref{lemma-rhs1}. Then by the linearity, \eqref{est-outer-apriori} follows from Lemma \ref{lemma-rhs1}, Lemma \ref{lemma-rhs2} and Lemma \ref{lemma-rhs3}.
\end{proof}

The proofs of above Lemmata are postponed to Appendix \ref{Appendix-outer}.

\medskip

\begin{remark}\label{remark6.1}
Let us point out the reason why we use the $\|\cdot\|_{*}$-norm of $\psi$ \eqref{def-norm*} only involving $\nu$ but not $\nu_2$ appearing in Lemma \ref{lemma-rhs2}. Lemma \ref{lemma-rhs2} is needed to deal with the right hand side of outer problem with cut-off $1-\eta_R$ in front. For convenience, when we carry out the inner--outer gluing procedure to bound right hand sides in the chosen topology, we will adjust $\nu_2$ such that the control of $\psi$ in Lemma \ref{lemma-rhs2} is better than that of Lemma \ref{lemma-rhs1}. This will result in a constraint for the parameters
$
\nu_2>\nu-1+\frac{\alpha}{2}.
$
In fact, the above constraint will be satisfied by the choices of parameters in Section \ref{subsec-choices}.
\end{remark}

\medskip


\section{Linear theory for the inner problem}\label{sec-linearinner}

\medskip

In this section, we develop a linear theory concerning the estimates for the associated linear problem of the inner problem under certain topology.

In order to solve the inner problem \eqref{inner}, we consider the associated linear problem
\begin{equation}\label{eqn-li}
\la^{2} \phi_t=\Delta_y \phi+3U^{2}(y) \phi+h(y,t)~\mbox{ in }~\mathcal D_{2R}\times (0,T).
\end{equation}
Recall that the linearized operator $L_0=\Delta +3U^{2}$
has only one positive eigenvalue $\mu_0$ such that
$$L_0(Z_0)=\mu_0 Z_0,~Z_0\in L^{\infty}(\R^4),$$
where the corresponding eigenfunction $Z_0$ is radially symmetric with the asymptotic behavior
\begin{equation}\label{decay-Z_0}
Z_0(y)\sim |y|^{-3/2}e^{-\sqrt{\mu_0}|y|}~\mbox{ as }~|y|\to+\infty.
\end{equation}
Multiplying equation \eqref{eqn-li} by $Z_0$ and integrating over $\R^4$, we obtain that
$$\la^2(t)\dot p(t)-\mu_0 p(t)=q(t),$$
where
$$p(t)=\int_{\R^4} \phi(y,t)Z_0(y) dy~\mbox{ and }~q(t)=\int_{\R^4} h(y,t)Z_0(y) dy.$$
Then we have
$$p(t)=e^{\int_0^t \mu_0\la^{-2}(r) dr}\left(p(0)+\int_0^t q(\eta)\la^{-2}(\eta)e^{-\int_0^{\eta} \mu_0\la^{-2}(r)dr}d\eta\right).$$
In order to get a decaying solution, the initial condition
$$p(0)=-\int_0^T q(\eta)\la^{-2}(\eta)e^{-\int_0^{\eta} \mu_0\la^{-2}(r)dr}d\eta$$
is required. The above formal argument suggests that a linear constraint should be imposed on the initial value $\phi(y,0)$. Therefore, we consider the associated linear Cauchy problem of the inner problem \eqref{inner}
\begin{equation}\label{linear-inner'''''}
\begin{cases}
  \la^2\phi_{t}=\Delta_y\phi + 3U^{2}(y)\phi + h(y,t), & \mbox{ in } \mathcal D_{2R}\times (0,T) \\
  \phi(y,0)=e_0 Z_0(y), & \mbox{ in } \mathcal D_{2R(0)}
\end{cases}
\end{equation}
where $R=R(t)=(T-t)^{-1/2}|\log(T-t)|^{\theta}$ with $\theta\in \left(\frac{n-1}{n-2},\frac{n}{n-2}\right)$. Note that with the above choice of $R(t)$, the inner problem is inside the self-similar region since $$\la_*(t)R(t)\sim\sqrt{T-t}|\log(T-t)|^{\theta-\frac{n}{n-2}}\ll\sqrt{T-t}$$ as $t\nearrow T$. On the other hand, the parabolic operator $-\la^2\partial_t+L_0$ is certainly not invertible since all the time independent elements in the 5 dimensional kernel of $L_0$ (see \eqref{kernels}) also belong to the kernel of $-\la^2\partial_t+L_0$. In order to construct solution to \eqref{linear-inner'''''} with suitable space-time decay, some orthogonality conditions are expected to hold. So we consider the projected problem
\begin{equation}\label{linear-inner}
\begin{cases}
  \la^2\phi_{t}=\Delta_y\phi + 3U^{2}(y)\phi + h(y,t)+\tilde c^0(t)Z_5(y)+\sum\limits_{\ell=1}^4 c^{\ell}(t) Z_{\ell}(y), & \mbox{ in } \mathcal D_{2R}\times (0,T), \\
  \phi(y,0)=e_0 Z_0(y), & \mbox{ in } \mathcal D_{2R(0)}.
\end{cases}
\end{equation}

Our aim is to find suitable solution to problem \eqref{linear-inner} with space-time decay of the following type
\begin{equation}\label{def-normnua}
\|\phi\|_{\rm{in},\nu,a}:=\sup_{\substack{y\in \mathcal D_{2R} \\ t\in(0,T)}} \la_*^{-\nu}(t)(1+|y|^a)\left[|\phi(y,t)|+(1+|y|)|\nabla \phi(y,t)|\right],
\end{equation}
and the norm of the right hand side of problem \eqref{linear-inner} is given by
$$\|h\|_{\nu,a}:=\sup_{\substack{y\in \mathcal D_{2R} \\ t\in(0,T)}} \la_*^{-\nu}(t)(1+|y|^a)|h(y,t)|.
$$ 
The construction of such solution is carried out by decomposing the equation into different spherical harmonic modes. Let an orthonormal basis $\{\Theta_i\}_{i=0}^{\infty}$ be made up of spherical harmonics in $L^2(\mathbb{S}^{3})$, i.e.
$$\Delta_{\mathbb{S}^{3}}\Theta_i+\tilde\mu_i \Theta_i=0 ~\mbox{ in } ~\mathbb{S}^{3}$$
with
$$0=\tilde\mu_0<\tilde\mu_1=\cdots=\tilde\mu_4=3<\tilde\mu_5\leq\cdots.$$
More precisely, $\Theta_0(y)=c_0,~\Theta_i(y)=c_1y_i,~i=1,\cdots,4$ for two constants $c_0$, $c_1$ and $\tilde\mu_i$ takes the general form $i(2+i)$ with multiplicity $\frac{(3+i)!}{6i!}$ for $i\geq 0$.

For $h\in L^2(\mathcal D_{2R})$, we decompose it into
\begin{equation*}
h(y,t)=\sum\limits_{j=0}^{\infty} h_j(r,t)\Theta_j(y/r),\quad r=|y|,\quad h_j(r,t)=\int_{\mathbb{S}^{3}} h(r\theta,t)\Theta_j(\theta)d\theta
\end{equation*}
and write $h=h^0+h^1+h^{\perp}$ with
$$h^0=h_0(r,t),\quad h^1=\sum\limits_{j=1}^4 h_j(r,t)\Theta_j,\quad h^{\perp}=\sum\limits_{j=5}^{\infty} h_j(r,t)\Theta_j.$$
Also, we decompose $\phi=\phi^0+\phi^1+\phi^{\perp}$ in a similar form. Then finding a solution to problem \eqref{linear-inner} is equivalent to finding the pairs $(\phi^0,h^0),~(\phi^1,h^1),~(\phi^{\perp},h^{\perp})$ in each mode.

The main result of this section is stated as follows.
\begin{prop}\label{lineartheory}
Let constants $a,\nu,\nu_1,\sigma\in(0,1)$ and $a_1\in(1,2)$. For $T>0$ sufficiently small and any $h(y,t)$ satisfying $\|h^0\|_{\nu,2+a}<+\infty$, $\|h^1\|_{\nu_1,2+a_1}<+\infty$, $\|h^{\perp}\|_{\nu,2+a}<+\infty$, there exists a solution $(\phi,\tilde c^0,c^{\ell},e_0)$ solving \eqref{linear-inner} and   $(\phi,\tilde c^0,c^{\ell},e_0)=(\phi[h],\tilde c^0[h^0],c^{\ell}[h^1],e_0[h])$ defines a linear operator of $h$ that satisfies the estimates:
$$ \begin{aligned}
|\phi(y,t)|+(1+|y|)|\nabla \phi(y,t)|\lesssim&~ \Bigg[\frac{\la_*^{\nu}(t)R^{\sigma(4-a)}\log R}{1+|y|^4}\|h^0\|_{\nu,2+a}+\frac{\la_*^{\nu_1}(t)}{1+|y|^{a_1}}\|h^1\|_{\nu_1,2+a_1}\\
&~+\frac{\la_*^{\nu}(t)}{1+|y|^a}\|h^{\perp}\|_{\nu,2+a}\Bigg]
\end{aligned}
$$ 
for $|y|\leq 2R^{\sigma}$, and
$$ 
\begin{aligned}
|\phi(y,t)|+(1+|y|)|\nabla \phi(y,t)|\lesssim&~ \Bigg[\frac{\la_*^{\nu}(t)\log R}{1+|y|^a}\|h^0\|_{\nu,2+a}+\frac{\la_*^{\nu_1}(t)}{1+|y|^{a_1}}\|h^1\|_{\nu_1,2+a_1}+\frac{\la_*^{\nu}(t)}{1+|y|^a}\|h^{\perp}\|_{\nu,2+a}\Bigg]
\end{aligned}
$$ 
for $2R^{\sigma}\leq |y| \leq 2R$.
$$ 
\tilde c^0(t)=-\frac{\int_{\mathcal D_{2R_*}}h^0 Z_5 dy}{\int_{\mathcal D_{2R_*}}|Z_5|^2 dy}-\mathcal O[h^0],~~c^{\ell}(t)=-\frac{\int_{\mathcal D_{2R}}h^1 Z_{\ell} dy}{\int_{\mathcal D_{2R}}|Z_{\ell}|^2 dy}~\mbox{ for }~\ell=1,\cdots,4,
$$ 
where $R_*=R^{\sigma}$ and $\mathcal O[h^0]$ is a linear operator of $h^0$ satisfying
$$|\mathcal O[h^0]|\lesssim \la_*^{\nu} R_*^{a'-a}\log R\|h^0\|_{\nu,2+a}$$
for $a'\in(0,a)$. Moreover,
\begin{equation*}
|e_0[h]|\lesssim \la_*^{\nu}\left(\|h^0\|_{\nu,2+a}+\|h^1\|_{\nu_1,2+a_1}+\|h^{\perp}\|_{\nu,2+a}\right).
\end{equation*}
\end{prop}

We devote the rest of this section to proving Proposition \ref{lineartheory}. Our strategy is to construct $\phi=\phi^0+\phi^1+\phi^{\perp}$ mode by mode.

\medskip

\noindent {\bf 1. Construction at mode $0$.}

\medskip

We construct $\phi^0$ solving the linearized problem at mode $0$
\begin{equation}\label{eqnmode0'''}
\begin{cases}
  \la^2\phi^0_{t}=\Delta_y\phi^0 + 3U^{2}(y)\phi^0 + h^0(y,t)+\tilde c^0(t)Z_5(y), & \mbox{ in } \mathcal D_{2R}\times (0,T), \\
  \phi^0(y,0)=e_0 Z_0(y), & \mbox{ in } \mathcal D_{2R(0)}.
\end{cases}
\end{equation}

The main result for mode $0$ is the following

\begin{prop}\label{propmode0}
Let $\nu,a,\sigma\in(0,1)$. Suppose $\|h^0\|_{\nu,2+a}<+\infty$. Then there exists a solution $(\phi^0,\tilde c^0,e_0)$ to problem \eqref{eqnmode0'''}, which depends on $h^0$ linearly such that
\begin{equation*}
|\phi^0(y,t)|+(1+|y|)|\nabla \phi^0(y,t)|\lesssim \la_*^{\nu}\log R\|h^0\|_{\nu,2+a}\begin{cases}
\frac{R^{\sigma(4-a)}}{1+|y|^4},~&\mbox{ for }~|y|\leq 2R^{\sigma},\\
\frac{1}{1+|y|^a},~&\mbox{ for }~2R^{\sigma}\leq|y|\leq 2R,\\
\end{cases}
\end{equation*}
\begin{equation*}
\tilde c^0[h^0](t)=-\frac{\int_{\mathcal D_{2R_*}}h^0 Z_5 dy}{\int_{\mathcal D_{2R_*}}|Z_5|^2 dy}-\mathcal O[h^0],
\end{equation*}
where $\mathcal O[h^0]$ is a linear operator of $h^0$ satisfying
$$|\mathcal O[h^0]|\lesssim \la_*^{\nu} R_*^{a'-a}\log R\|h^0\|_{\nu,2+a}$$
for $a'\in(0,a).$ Futhermore, it holds that
$$|e_0[h^0]|\lesssim \la_*^{\nu}\|h^0\|_{\nu,2+a}.$$
\end{prop}

\begin{remark}
If we define
\begin{equation}\label{def-normm0}
\|\phi^0\|_{0,\sigma,\nu,a}:=\sup_{(y,t)\in\mathcal D_{2R}\times(0,T)} \frac{1+|y|^4}{\la_*^{\nu}(t)R^{\sigma(4-a)}(t)\log R}\left[|\phi^0(y,t)|+(1+|y|)|\nabla \phi^0(y,t)|\right],
\end{equation}
then Proposition \ref{propmode0} implies that
\begin{equation*}
\|\phi^0\|_{0,\sigma,\nu,a}\lesssim \|h^0\|_{\nu,2+a}.
\end{equation*}
\end{remark}

\medskip

The strategy to prove Proposition \ref{propmode0} is a new inner--outer gluing scheme. We shall decompose $\phi^0$ into inner and outer profiles to get more refined estimates. Before we prove Proposition \ref{propmode0}, we first state a result for the following problem
\begin{equation}\label{modelpropmode0''}
\begin{cases}
\la^2\phi_t=\Delta_y \phi+3U^2(y)\phi+h(y,t)+\tilde c^0(t)Z_5-c(t)Z_0,~&\mbox{ in } \mathcal D_{2R}\times (0,T),\\
\phi(y,0)=0,~&\mbox{ in } \mathcal D_{2R(0)}.
\end{cases}
\end{equation}
\begin{prop}\label{modelpropmode0}
Let $\nu,a\in(0,1)$. Then for sufficiently large $R$ and any $h$ satisfying $\|h\|_{\nu,2+a}<+\infty$, there exists a solution $(\phi,\tilde c^0,c)$ to \eqref{modelpropmode0''} which is linear in $h$ such that
\begin{equation}\label{mode0est444}
|\phi(y,t)|+(1+|y|)|\nabla \phi(y,t)|\lesssim \la_*^{\nu}\frac{R^{4-a}\log R}{1+|y|^4}\|h\|_{\nu,2+a},
\end{equation}
\begin{equation*}
\tilde c^0(t)=-\frac{\int_{\mathcal D_{2R}}h Z_5 dy}{\int_{\mathcal D_{2R}}|Z_5|^2 dy},
\end{equation*}
and
\begin{equation}\label{est-ccc}
\left|c(t)-\int_{\mathcal D_{2R}} h Z_0\right| \lesssim \la_*^{\nu}(t)\left[\left\|h-Z_0\int_{\mathcal D_{2R}}h Z_0\right\|_{\nu,2+a}+e^{-\vartheta R}\|h\|_{\nu,2+a}\right].
\end{equation}
\end{prop}
The proof of Proposition \ref{modelpropmode0} can be carried out similar to that of \cite[Section 7]{Green16JEMS} (see also \cite[Section 5.2]{ni4d}). Proposition \ref{modelpropmode0} will be needed to describe the inner profile of $\phi^0$ when the inner--outer gluing scheme is carried out.

\begin{proof}[Proof of Proposition \ref{propmode0}]
Suppose
$$\phi^0(y,t)=\phi^0_1+e(t)Z_0(y)$$
with $\phi_1^0$ solving problem
\begin{equation}\label{linear-inner''}
\begin{cases}
\la^2\phi_t=\Delta_y \phi+3U^2(y)\phi+h^0(y,t)+\tilde c^0(t)Z_5-c(t)Z_0,~&\mbox{ in } \mathcal D_{2R}\times (0,T),\\
\phi(y,0)=0,~&\mbox{ in } \mathcal D_{2R(0)}.
\end{cases}
\end{equation}
For $e\in C^1((0,T))$, we get
$$\la^2\phi^0_t=\Delta_y \phi^0+3U^2\phi^0+h^0(y,t)+\tilde c^0(t)Z_5+[\la^2\dot e(t)-\mu_0 e(t)-c(t)]Z_0(y),$$
from which we see that a natural choice of bounded solution $e(t)$ to
$$\la^2\dot e(t)-\mu_0 e(t)-c(t)=0,~t\in(0,T)$$
is
\begin{equation}\label{choiceeee}
e(t)=-\int_t^T \exp \left(-\int_t^{\eta}\frac{\mu_0}{\la^2(s)}ds\right)\frac{c(\eta)}{\la^2(\eta)}d\eta.
\end{equation}
Therefore, $\phi^0$ solves problem \eqref{eqnmode0'''} with the initial condition $\phi^0(y,0)=e(0)Z_0(y)$. It is clear from \eqref{choiceeee} and \eqref{est-ccc} that
$$|e_0|\lesssim \la_*^{\nu}\|h^0\|_{\nu,2+a}.$$
So, to solve \eqref{eqnmode0'''}, we only need to consider \eqref{linear-inner''}.

Now we carry out an inner--outer gluing scheme for the mode $0$. Consider
\begin{equation}\label{iomode0}
\begin{cases}
\la^2\phi_t=\Delta_y \phi+3U^2(y)\phi+h^0(y,t)+\tilde c^0(t)Z_5-c(t)Z_0,~&\mbox{ in } \mathcal D_{2R}\times (0,T),\\
\phi(y,0)=0,~&\mbox{ in } \mathcal D_{2R(0)},\\
\phi=0,~&\mbox{ on } \partial \mathcal D_{2R}\times (0,T).\\
\end{cases}
\end{equation}
We shall construct $\phi^0$ solving \eqref{iomode0} of the form
$$\phi^0=\phi^0_{out}+\eta_{R_*}\phi^0_{in},$$
where
$$\eta_{R_*}:=\eta\left(\frac{|y|}{R_*}\right)$$
with $\eta$ defined in \eqref{def-cutoff} and
$$R_*=R^{\sigma}~\mbox{ for }~\sigma\in(0,1).$$
A solution $\phi^0$ to \eqref{iomode0} is found if the pair $(\phi^0_{out},\phi^0_{in})$ solves the system
\begin{equation}\label{outer-mode0}
\left\{
\begin{aligned}
&\la^2\partial_t \phi^0_{out}=\Delta_y \phi^0_{out}+3(1-\eta_{R_*})U^2(y)\phi^0_{out}+\mathtt C[\phi^0_{in}]+(1-\eta_{R_*})h^0,~\mbox{ in } \mathcal D_{2R}\times (0,T)\\
&\phi_{out}^0(y,0)=0,~\mbox{ in } \mathcal D_{2R(0)}\\
&\phi_{out}^0=0,~\mbox{ on } \partial \mathcal D_{2R}\times (0,T)\\
\end{aligned}
\right.
\end{equation}
$$ 
\left\{
\begin{aligned}
&\la^2\partial_t \phi^0_{in}=\Delta_y \phi_{in}^0+3U^2(y)\phi^0_{in}+3U^2(y)\phi^0_{out}+h^0+\tilde c^0Z_5-cZ_0,~\mbox{ in } \mathcal D_{2R_*}\times (0,T)\\
&\phi_{in}^0(y,0)=0,~\mbox{ in } \mathcal D_{2R_*(0)}\\
\end{aligned}
\right.
$$ 
where
$$\mathtt C[\phi^0_{in}]:=\phi_{in}^0(\Delta \eta_{R_*}-\la^2 \partial_t \eta_{R_*}) +2\nabla \eta_{R_*}\cdot \nabla \phi_{in}^0.$$
We first consider the outer part \eqref{outer-mode0}. For the model problem
\begin{equation*}
\begin{cases}
\la^2 \partial_t \psi=\Delta \psi+h^0~&\mbox{ in } \mathcal D_{2R}\times (0,T)\\
\psi(y,0)=0,~&\mbox{ in } \mathcal D_{2R(0)}\\
\psi=0,~&\mbox{ on } \partial \mathcal D_{2R}\times (0,T)\\
\end{cases}
\end{equation*}
we have
\begin{equation}\label{mode0est111}
\|\psi\|_{\nu,a}\lesssim\|h^0\|_{\nu,2+a}
\end{equation}
by the parabolic comparison. Then we apply the above estimate to the following problem
\begin{equation}\label{mode0eqn111}
\begin{cases}
\la^2 \partial_t \psi=\Delta \psi+3(1-\eta_{R_*})U^2\psi+h^0~&\mbox{ in } \mathcal D_{2R}\times (0,T)\\
\psi(y,0)=0,~&\mbox{ in } \mathcal D_{2R(0)}\\
\psi=0,~&\mbox{ on } \partial \mathcal D_{2R}\times (0,T)\\
\end{cases}
\end{equation}
and we claim that the solution $\psi$ to \eqref{mode0eqn111} satisfies
$$\|\psi\|_{\nu,a}\lesssim\|h^0\|_{\nu,2+a}.$$
Indeed, by \eqref{mode0est111}, we only need to estimate
\begin{equation*}
\begin{aligned}
3(1-\eta_{R_*})U^2\psi \lesssim&~ (1-\eta_{R_*})\frac{\la_*^{\nu}}{1+|y|^{4+a}}\|\psi\|_{\nu,a}
\lesssim R_*^{-2}\frac{\la_*^{\nu}}{1+|y|^{2+a}}\|\psi\|_{\nu,a}
\end{aligned}
\end{equation*}
and we conclude that
\begin{equation}\label{mode0est222}
\|3(1-\eta_{R_*})U^2\psi\|_{\nu,2+a}\lesssim R_*^{-2}\|\psi\|_{\nu,a}.
\end{equation}
So from \eqref{mode0est111} and \eqref{mode0est222}, we obtain
\begin{equation*}
\begin{aligned}
\|\psi\|_{\nu,a}\lesssim&~\|3(1-\eta_{R_*})U^2\psi+h^0\|_{\nu,2+a} 
\lesssim R_*^{-2}\|\psi\|_{\nu,a}+\|h^0\|_{\nu,2+a}
\end{aligned}
\end{equation*}
and for $R_*$ sufficiently large, it follows that
\begin{equation}\label{mode0est333}
\|\psi\|_{\nu,a}\lesssim \|h^0\|_{\nu,2+a}
\end{equation}
as desired.

We look for a solution $\phi_{out}^0$ to problem \eqref{outer-mode0}. By \eqref{mode0est333}, we get
\begin{equation}\label{mode0est666}
\|\phi^0_{out}\|_{\nu,a'}\lesssim \|\mathtt C[\phi_{in}^0]\|_{\nu,2+a'}+\|(1-\eta_{R_*})h^0\|_{\nu,2+a'},
\end{equation}
where $a'\in(0,a)$. Here $\phi^0_{out}$ defines a linear operator of $\phi^0_{in}$ and $h^0$. We write it as $\phi^0_{out}[\phi^0_{in},h^0]$.
 Now we need to find $\phi^0_{in}$ solving the inner part
\begin{equation}\label{inner-mode0''}
\left\{
\begin{aligned}
&\la^2\partial_t \phi^0_{in}=\Delta_y \phi_{in}^0+3U^2(y)\phi^0_{in}+3U^2(y)\phi^0_{out}[\phi^0_{in},h^0]+h^0+\tilde c^0Z_5-cZ_0,~\mbox{ in } \mathcal D_{2R_*}\times (0,T),\\
&\phi_{in}^0(y,0)=0,~\mbox{ in } \mathcal D_{2R_*(0)}.\\
\end{aligned}
\right.
\end{equation}
To solve the inner part \eqref{inner-mode0''}, we consider the fixed point problem
$$\phi^0_{in}=\mathcal T\left[3U^2(y)\phi^0_{out}[\phi^0_{in},h^0]+h^0\right]$$
in the function space equipped with the norm
\begin{equation*}
\|\phi^0_{in}\|_{0,*}:=\sup_{(y,t)\in \mathcal D_{2R_*}\times(0,T)} \la_*^{-\nu}(t) R_*^{a-4}(\log R)^{-1}(1+|y|^4)\left[|\phi^0_{in}|+(1+|y|)|\nabla\phi^0_{in}|\right].
\end{equation*}
We apply Proposition \ref{modelpropmode0} in the inner regime $\mathcal D_{2R_*}\times (0,T)$, then \eqref{mode0est444} gives
\begin{equation}\label{mode0est999}
\|\mathcal T[g]\|_{0,*}\lesssim \|g\|_{\nu,2+a}.
\end{equation}
We claim that
\begin{equation}\label{mode0est555}
\|\mathtt C[\phi^0_{in}]\|_{\nu,2+a'}\lesssim R_*^{a'-a}\log R\|\phi^0_{in}\|_{0,*}.
\end{equation}
Indeed, we evaluate
\begin{equation*}
\begin{aligned}
|\mathtt C[\phi^0_{in}]|=&~\left|\phi_{in}^0(\Delta \eta_{R_*}-\la^2 \partial_t \eta_{R_*}) +2\nabla \eta_{R_*}\cdot \nabla \phi_{in}^0\right|\\
\lesssim&~R_*^{-2}|\eta''|\la_*^{\nu}\frac{R_*^{4-a}\log R}{1+|y|^4} \|\phi^0_{in}\|_{0,*}
\lesssim \frac{\la_*^{\nu}\log R}{1+|y|^{2+a'}}R_*^{a'-a}\|\phi^0_{in}\|_{0,*}\\
\end{aligned}
\end{equation*}
which proves \eqref{mode0est555}. From \eqref{mode0est666} and \eqref{mode0est555}, we then get that
\begin{equation}\label{mode0est777}
\begin{aligned}
\|\phi^0_{out}\|_{\nu,a'}\lesssim&~ R_*^{a'-a}\log R\|\phi^0_{in}\|_{0,*}+\|(1-\eta_{R_*})h^0\|_{\nu,2+a'}\\
\lesssim&~ R_*^{a'-a}\log R\|\phi^0_{in}\|_{0,*}+ R_*^{a'-a}\|h^0\|_{\nu,2+a}.\\
\end{aligned}
\end{equation}
Next, we compute
\begin{equation*}
\begin{aligned}
|3U^2(y)\phi^0_{out}|\lesssim&~ \frac{\la_*^{\nu}}{1+|y|^{4+a'}}\|\phi^0_{out}\|_{\nu,a'}
\lesssim  \frac{\la_*^{\nu}}{1+|y|^{2+a}}\|\phi^0_{out}\|_{\nu,a'}.
\end{aligned}
\end{equation*}
So we get
\begin{equation}\label{mode0est888}
\|3U^2(y)\phi^0_{out}\|_{\nu,2+a}\lesssim \|\phi^0_{out}\|_{\nu,a'}.
\end{equation}
By \eqref{mode0est777} and \eqref{mode0est888}, we obtain
\begin{equation}\label{mode0est101010}
\|3U^2(y)\phi^0_{out}\|_{\nu,2+a}\lesssim R_*^{a'-a}\log R\|\phi^0_{in}\|_{0,*}+ R_*^{a'-a}\|h^0\|_{\nu,2+a}.
\end{equation}
Therefore, we conclude from \eqref{mode0est999} that
$$\left\|\mathcal T\left[3U^2(y)\phi^0_{out}[\phi^0_{in},h^0]+h^0\right]\right\|_{0,*}\lesssim R_*^{a'-a}\log R\|\phi^0_{in}\|_{0,*}+\|h^0\|_{\nu,2+a},$$
which shows that the operator
$$\phi_{in}^0\mapsto \mathcal T\left[3U^2(y)\phi^0_{out}[\phi^0_{in},h^0]+h^0\right]$$
is a contraction if $R_*$ is sufficiently large. A unique fixed point $\phi^0_{in}$ thus exists and
\begin{equation}\label{est-mode0in}
\|\phi^0_{in}\|_{0,*}\lesssim \|h^0\|_{\nu,2+a}.
\end{equation}
Replacing $a'$ by $a$ in the computations of \eqref{mode0est666} and \eqref{mode0est555}, we obtain
\begin{equation}\label{est-mode0out}
\|\phi^0_{out}\|_{\nu,a}\lesssim \log R\|h^0\|_{\nu,2+a}.
\end{equation}
Recalling $\phi^0=\phi^0_{out}+\eta_{R_*}\phi^0_{in}$ and combining \eqref{est-mode0in} and \eqref{est-mode0out}, we conclude
\begin{equation*}
|\phi^0(y,t)|+(1+|y|)|\nabla \phi^0(y,t)|\lesssim \la_*^{\nu}\log R\|h^0\|_{\nu,2+a}\begin{cases}
\frac{R^{\sigma(4-a)}}{1+|y|^4},~&\mbox{ for }~|y|\leq 2R^{\sigma},\\
\frac{1}{1+|y|^a},~&\mbox{ for }~2R^{\sigma}\leq|y|\leq 2R.\\
\end{cases}
\end{equation*}

Finally, we prove the estimate of $\tilde c^0$. By Proposition \ref{modelpropmode0}, we get
$$\tilde c^0(t)=-\frac{\int_{\mathcal D_{2R_*}}\big(h^0+3U^2(y)\phi^0_{out}[\phi^0_{in},h^0]\big) Z_5 dy}{\int_{\mathcal D_{2R_*}}|Z_5|^2 dy}.$$
Notice that $3U^2(y)\phi^0_{out}[\phi^0_{in},h^0]$ is linear in $h^0$. By \eqref{mode0est101010} and \eqref{est-mode0in}, we conclude that
\begin{equation*}
\begin{aligned}
\left|\int_{\mathcal D_{2R_*}}3U^2(y)\phi^0_{out}[\phi^0_{in},h^0] Z_5 dy\right|\lesssim&~ \la_*^{\nu}\left(R_*^{a'-a}\log R\|\phi^0_{in}\|_{0,*}+ R_*^{a'-a}\|h^0\|_{\nu,2+a}\right)\\
\lesssim&~\la_*^{\nu} R_*^{a'-a}\log R\|h^0\|_{\nu,2+a}.
\end{aligned}
\end{equation*}
The proof is complete.
\end{proof}

\medskip

\noindent {\bf 2. Construction at modes $1$ to $4$.}

\medskip

As we can see in mode $0$, the estimates are somewhat deteriorated inside the inner regime, and this will result in difficulties when solving the inner problem. One can observe that, for modes $1$ to $4$, the kernel function for the corresponding linearized operator has faster decay than mode $0$ which suggests that the estimates at modes $1$ to $4$ should be better than mode $0$'s. Inspired by the argument in \cite[Section 7]{17HMF}, we shall carry out the construction for modes $1$ to $4$ by means of the blow-up argument.

We perform the change of variable
$$\tau=\tau_{\la}(t)=\tau_0+\int_0^t \frac{ds}{\la_*^2(s)}$$
so that
$$\tau\sim \tau_0+\frac{|\log(T-t)|^{\frac{n}{n-2}}}{\la_* |\log T|^{\frac{2}{n-2}}}.$$
We choose the constant $\nu_1'>0$ so that
$$\tau^{-\nu_1'}\sim \la_*^{\nu_1}$$
for $\nu_1\in (0,1)$.

The main proposition for modes $1$ to $4$ is the following.
\begin{prop}\label{propmode1}
Assume $a_1\in (1,2),\nu_1\in(0,1),\|h^1\|_{\nu_1,2+a_1}<+\infty,$ and
$$\int_{\mathcal D_{2R}} h^1(y,\tau)Z_i(y)dy=0,~\mbox{for all}~\tau\in(\tau_0,+\infty),i=1,\cdots,4.$$
For sufficiently large $R$, there exists a pair $(\phi^1,e_0)$ solving
\begin{equation*}
\begin{cases}
\partial_{\tau} \phi^1=\Delta \phi^1+3U^{2}\phi^1+h^1(y,\tau),~&y\in \mathcal D_{2R}\times (\tau_0,\infty)\\
\phi^1(y,\tau_0)=e_0 Z_0(y),~&y\in \mathcal D_{2R}
\end{cases}
\end{equation*}
and $(\phi^1,e_0)=(\phi^1[h^1],e_0[h^1])$ defines a linear operator of $h^1$ that satisfies
\begin{equation*}
\|\phi^1\|_{{\rm in},\nu_1,a_1}\lesssim \|h^1\|_{\nu_1,2+a_1}
\end{equation*}
and
\begin{equation*}
|e_0[h^1]|\lesssim \tau^{-\nu_1'} \|h^1\|_{\nu_1,2+a_1}.
\end{equation*}
\end{prop}

In order to prove Proposition \ref{propmode1}, we consider the following Cauchy problem
\begin{equation}\label{blowupmode1''}
\begin{cases}
\partial_{\tau} \phi^1=\Delta \phi^1+3U^{2}\phi^1+h^1(y,\tau)-c(\tau)Z_0(y),&~y\in \mathbb{R}^4,~\tau\geq\tau_0\\
\phi^1(y,\tau_0)=0,&~y\in \mathbb{R}^4
\end{cases}
\end{equation}
with $h^1$ supported in $\mathcal D_{2R}\times (\tau_0,+\infty)$ and $\|h^1\|_{\nu_1,2+a_1}<+\infty$ in the $(y,\tau)$ variable, where $\nu_1\in(0,1)$ and $a_1\in(1,2)$. For notational convenience, we denote $\phi^1$ by $\phi$ and $h^1$ by $h$ in the following lemma.
\begin{lemma}\label{lemmamode1}
Assume $a_1\in (1,2),\nu_1\in(0,1),\|h\|_{\nu_1,2+a_1}<+\infty,$ and
$$\int_{\mathbb{R}^4} h(y,\tau)Z_i(y)dy=0,~\mbox{for all}~\tau\in(\tau_0,+\infty),~i=1,\cdots,4.$$
For $\tau_1$ sufficiently large, the solution $\phi(y,\tau)$ to
\begin{equation}\label{blowupmode1}
\left\{
\begin{aligned}
&\partial_{\tau} \phi=\Delta \phi+3U^{2}\phi+h(y,\tau)-c(\tau)Z_0(y),~~~y\in \mathbb{R}^4,~\tau\geq\tau_0\\
&\int_{\mathbb{R}^4} \phi(y,\tau) Z_0(y)dy=0,~~~\mbox{for all}~\tau\in(\tau_0,+\infty),\\
&\phi(y,\tau_0)=0,~~~y\in \mathbb{R}^4
\end{aligned}
\right.
\end{equation}
satisfies
\begin{equation}\label{est-blowupmode1}
\|\phi(y,\tau)\|_{a_1,\tau_1}\lesssim\|h\|_{2+a_1,\tau_1}.
\end{equation}
Further,
\begin{equation*}
|c(\tau)|\lesssim \tau^{-\nu_1'} R^{a_1} \|h\|_{2+a_1,\tau_1}~~\mbox{for}~\tau\in[\tau_0,\tau_1),
\end{equation*}
where $\|h\|_{b,\tau_1}:=\sup_{\tau\in [\tau_0,\tau_1)}\tau^{\nu_1'}\sup\limits_{y\in \mathbb{R}^4}(1+|y|^b)|h(y,\tau)|$.
\end{lemma}
\begin{proof}
Note that problem \eqref{blowupmode1} is equivalent to problem \eqref{blowupmode1''} for
$$c(\tau)=\frac{\int_{\mathbb R^4}h(y,\tau)Z_0(y)dy}{\int_{\mathbb R^4}|Z_0(y)|^2dy}.$$
By the time decay of $h$ and spatial decay of $Z_0$ (see \eqref{decay-Z_0}), we have
\begin{equation}\label{2}
|c(\tau)|\lesssim \tau^{-\nu_1'} R^{a_1} \|h\|_{2+a_1,\tau_1}.
\end{equation}
Now we prove \eqref{est-blowupmode1} by blow-up argument.

By standard parabolic theory, for any $R'>0$, there exists a constant $K$ depending on $R'$ and $\tau_1$ such that
$$|\phi(y,\tau)|\leq K,~\mbox{in}~B_{R'}\times[\tau_0,\tau_1].$$
It is easy to check that $\bar{\phi}=\frac{C}{1+|y|^{a_1}}$ is a super-solution to the original equation \eqref{blowupmode1}. Thus, $\|\phi\|_{a_1,\tau_1}<+\infty$. We claim that
$$\int_{\mathcal D_{2R}} \phi Z_i=0~\mbox{for all}~\tau\in[\tau_0,\tau_1],~i=1,\cdots,4.$$
Indeed, we multiply \eqref{blowupmode1} by $Z_i\eta_{R'},$ where $\eta_{R'}:=\eta(\frac{|y|}{R'})$ and $\eta$ is the standard cut-off function defined in \eqref{def-cutoff}. Then we have
$$\int_{\mathbb{R}^4}\phi(\cdot, \tau)\cdot Z_i\eta_{R'} = \int_{\tau_0}^\tau ds\int_{\mathbb{R}^4}(\phi(\cdot, s)\cdot L_0[\eta_{R'} Z_i] + h Z_i\eta_{R'} - c(s)Z_0Z_i\eta_{R'}).$$
Further computation gives
\begin{equation*}
\begin{aligned}
&\quad\int_{\mathbb{R}^4}[\phi(\cdot, s)\cdot L_0[\eta_{R'} Z_i] + h Z_i\eta_{R'} - c(s)Z_0Z_i\eta_{R'}].\\
&=\int_{\mathbb{R}^4}\phi(\cdot, s)[\eta_{R'}L_0[Z_i]+Z_i \Delta\eta_{R'}+2\nabla\eta_{R'}\cdot\nabla Z_i] + h Z_i\eta_{R'} - c(s)Z_0Z_i\eta_{R'}=O((R')^{-\epsilon})
\end{aligned}
\end{equation*}
for some $\epsilon>0$.
By taking $R'\rightarrow +\infty$, we get the desired result.

Now we want to prove
$$\|\phi\|_{a_1,\tau_1}\lesssim\|h\|_{2+a_1,\tau_1}.$$
We prove by contradiction. Suppose that there exist sequences $\tau_1^k\to +\infty$ and $\phi_k$, $h_k$, $c_k$ satisfying
\begin{equation*}
\left\{
\begin{aligned}
&\partial_\tau\phi_k = \Delta\phi_k + 3U^{2}(y)\phi_k + h_k - c_k(\tau)Z_0(y),~ y\in \mathbb{R}^4,~ \tau\geq \tau_0,\\
&\int_{\mathbb{R}^4}\phi_k(y, \tau)\cdot Z_i(y)dy = 0\text{ for all }\tau\in [\tau_0,\tau_1^k), ~i= 0, 1,\cdots, 4,\\
&\phi_k(y,\tau_0) = 0, ~y\in \mathbb{R}^4,
\end{aligned}
\right.
\end{equation*}
and
$$ 
\|\phi_k\|_{a_1,\tau_1^k}=1,\quad \|h_k\|_{2+a_1,\tau_1^k}\to 0.
$$ 
By \eqref{2}, we know $\sup_{\tau\in (\tau_0, \tau_1^k)}\tau^{\nu_1'} c_k(\tau)\to 0$.
We claim that
\begin{equation}\label{e5:37}
\sup_{\tau_0 < \tau < \tau_1^k}\tau^{\nu_1'}|\phi_k(y,\tau)|\to 0
\end{equation}
uniformly on compact subsets of $\mathbb{R}^4$. We prove \eqref{e5:37} by contradiction.

\noindent {\bf Case 1.} For some $|y_k|\leq M$ and $\tau_0 < \tau_2^k < \tau_1^k$, if
\begin{equation*}
(\tau_2^k)^{\nu_1'}|\phi_k(y_k,\tau_2^k)|\geq \frac{1}{2},
\end{equation*}
then we know that $\tau_2^k\to +\infty$. Define
\begin{equation*}
\tilde{\phi}_k(y,\tau) = (\tau_2^k)^{\nu_1'}\phi_k(y,\tau_2^k + \tau).
\end{equation*}
Then
\begin{equation*}
\partial_\tau\tilde{\phi}_k = L_0[\tilde{\phi}_k] + \tilde{h}_k - \tilde{c}_k(\tau)Z_0(y)\text{ in }\mathbb{R}^4\times (\tau_0-\tau_2^k,0].
\end{equation*}
Due to the spatial decay of $h$ and $c$, we know $\tilde{h}_k\to 0$, $\tilde{c}_k\to 0$. By comparison, we get
\begin{equation*}
|\tilde{\phi}_k(y,\tau)|\leq \frac{1}{1+|y|^{a_1}}\text{ in }\mathbb{R}^4\times (\tau_0-\tau_2^k,0].
\end{equation*}
Hence, up to a subsequence, $\tilde{\phi}_k\to\tilde{\phi}$ uniformly on compact subsets with $\tilde{\phi}\neq 0$ and
\begin{equation}\label{eqn-case1}
\left\{
\begin{aligned}
&\partial_\tau\tilde{\phi} =\Delta\tilde{\phi} + 3U^{2}(y)\tilde{\phi}~\text{ in }~\mathbb{R}^4\times (-\infty, 0],\\
&\int_{\mathbb{R}^4}\tilde{\phi}(y, \tau)\cdot Z_j(y)dy = 0~\text{ for all }~\tau\in (-\infty, 0], ~ j= 0, 1,\cdots, 4,\\
&|\tilde{\phi}(y,\tau)|\leq \frac{1}{1+|y|^{a_1}}~\text{ in }~\mathbb{R}^4\times (-\infty, 0],\\
&\tilde{\phi}(y,\tau_0) = 0, ~y\in\mathbb{R}^4.
\end{aligned}
\right.
\end{equation}
Note that the orthogonality conditions above are well-defined if $a_1>1$. We now claim that $\tilde{\phi} = 0$. Indeed, by parabolic regularity theory, $\tilde{\phi}(y,\tau)$ is smooth. By scaling argument, we get
\begin{equation*}
\frac{1}{1+|y|}|\nabla\tilde{\phi}| + |\tilde{\phi}_\tau| + |\Delta\tilde{\phi}|\lesssim \frac{1}{1+|y|^{2+a_1}}.
\end{equation*}
Differentiating \eqref{eqn-case1} with respect to $\tau$, we get $\partial_\tau\tilde{\phi}_\tau =\Delta\tilde{\phi}_\tau + 3U^{2}(y)\tilde{\phi}_\tau$ and
\begin{equation*}
\frac{1}{1+|y|}|\nabla\tilde{\phi}_\tau| + |\tilde{\phi}_{\tau\tau}| + |\Delta\tilde{\phi}_\tau|\lesssim \frac{1}{1+|y|^{4+a_1}}.
\end{equation*}
Differentiating \eqref{eqn-case1} with respect to $\tau$ and integrating, we get
\begin{equation*}
\frac{1}{2}\partial_\tau\int_{\mathbb{R}^4}|\tilde{\phi}_\tau|^2 + B(\tilde{\phi}_\tau, \tilde{\phi}_\tau) = 0,
\end{equation*}
where
\begin{equation*}
B(\tilde{\phi}, \tilde{\phi}) = \int_{\mathbb{R}^4}|\nabla\tilde{\phi}|^2 - 3U^{2}(y)|\tilde{\phi}|^2dy.
\end{equation*}
Since $\int_{\mathbb{R}^4}\tilde{\phi}(y, \tau)\cdot Z_j(y)dy = 0$ for all $\tau\in (-\infty, 0]$, $j= 0, 1,\cdots, 4 $, $B(\tilde{\phi}, \tilde{\phi})\geq 0$. Also, we have
\begin{equation*}
\int_{\mathbb{R}^4}|\tilde{\phi}_\tau|^2 = -\frac{1}{2}\partial_\tau B(\tilde{\phi}, \tilde{\phi}).
\end{equation*}
From above, we get
\begin{equation*}
\partial_\tau\int_{\mathbb{R}^4}|\tilde{\phi}_\tau|^2 \leq 0,\quad \int_{-\infty}^0d\tau\int_{\mathbb{R}^4}|\tilde{\phi}_\tau|^2 < +\infty.
\end{equation*}
Hence $\tilde{\phi}_\tau = 0$. So $\tilde{\phi}$ is independent of $\tau$ and $L_0[\tilde{\phi}] = 0$. Since $\tilde{\phi}$ is bounded, by the non-degeneracy of $L_0$, $\tilde{\phi}$ is a linear combination of $Z_j$, $j = 1,\cdots, 4$. From orthogonality conditions $\int_{\mathbb{R}^4}\tilde{\phi}\cdot Z_j = 0$, $j = 1,\cdots, 4$, we obtain $\tilde{\phi} = 0$, a contradiction. Thus,
$$\sup_{\tau_0 < \tau < \tau_1^k}\tau^{\nu_1'}|\phi_k(y,\tau)|\to 0.$$

\noindent {\bf Case 2.}
Suppose there exists $y_k$ with $|y_k|\to +\infty$ such that
\begin{equation*}
(\tau_2^k)^{\nu_1'}(1+|y_k|^{a_1})|\phi_k(y_k, \tau_2^k)|\geq \frac{1}{2}.
\end{equation*}
Let
\begin{equation*}
\tilde{\phi}_k(z, \tau):=(\tau_2^k)^{\nu_1'}|y_k|^{a_1}\phi_k(y_k+|y_k|z,|y_k|^{2}\tau + \tau_2^k).
\end{equation*}
Then
\begin{equation*}
\partial_\tau \tilde{\phi}_k = \Delta\tilde{\phi}_k + a_k\tilde{\phi}_k + \tilde{h}_k(z,\tau),\quad a_k=3U^2(y_k+|y_k|z),
\end{equation*}
where
\begin{equation*}
\tilde{h}_k(z,\tau) = (\tau_2^k)^{\nu_1'}|y_k|^{2+a_1}h_k(y_k+|y_k|z,|y_k|^{2}\tau + \tau_2^k)-(\tau_2^k)^{\nu_1'}|y_k|^{2+a_1}c(|y_k|^{2}\tau+\tau_2^k)Z_0(y_k+|y_k|z).
\end{equation*}
By the definition of $h_k$,
\begin{equation*}
|\tilde{h}_k(z,\tau)| \lesssim o(1)\frac{((\tau_2^k)^{-1}|y_k|^{2}\tau + 1)^{-\nu_1'}}{|\hat{y}_k+z|^{2+a_1}}
\end{equation*}
with
$
\hat{y}_k = \frac{y_k}{|y_k|}\to -{\hat e}
$
and $|\hat e|= 1$. Thus $\tilde{h}_k(z,\tau)\to 0$ uniformly on compact subsets of $\mathbb{R}^4\setminus\{\hat e\}\times (-\infty, 0]$ and $a_k$ has the same property. Moreover, $|\tilde{\phi}_k(0, \tau_0)|\geq \frac{1}{2}$ and
\begin{equation*}
|\tilde{\phi}_k(z,\tau)| \lesssim \frac{((\tau_2^k)^{-1}|y_k|^{2}\tau + 1)^{-\nu_1'}}{|\hat{y}_k+z|^{a_1}}.
\end{equation*}
Hence we may assume $\tilde{\phi}_k\to \tilde{\phi}\neq 0$ uniformly on compact subsets of $\mathbb{R}^4\setminus\{{\hat e}\}\times (-\infty,0]$ with $\tilde{\phi}$ satisfying
\begin{equation}\label{e5:39}
\tilde{\phi}_\tau = \Delta\tilde{\phi}\quad\text{in }\mathbb{R}^4\setminus\{{\hat e}\}\times (-\infty,0]
\end{equation}
and
\begin{equation}\label{e5:40}
|\tilde{\phi}(z,\tau)|\leq |z-{e}|^{-a_1}\quad\text{in }\mathbb{R}^4\setminus\{{\hat e}\}\times (-\infty,0].
\end{equation}
{\bf Claim}: functions $\tilde{\phi}$ satisfying (\ref{e5:39}) and (\ref{e5:40}) are 0.

Without loss of generality, we assume ${\hat e} = 0$. Then
$$ 
\left\{
\begin{aligned}
&\tilde{\phi}_\tau = \Delta\tilde{\phi},\quad\text{in }\mathbb{R}^4\setminus\{0\}\times (-\infty,0],\\
&|\tilde{\phi}(z,\tau)|\leq |z|^{-a_1},\quad\text{in }\mathbb{R}^4\setminus\{0\}\times (-\infty,0].
\end{aligned}
\right.
$$ 

We consider the function $\bar u(\rho,\tau) = (\rho^{2} + c\tau)^{-\frac{a_1}{2}} + \epsilon \rho^{-2}$ for some constant $c > 0$.
Direct computations give us
$$\bar u_{\tau}-\Delta \bar u=a_1(\rho^2+c\tau)^{-\frac{a_1}{2}-2}\left[(2-a_1-\frac{c}{2})\rho^2+(4c-\frac{c^2}{2})\tau\right].$$
Then we know that if $a_1<2$, we can always find $c>0$ such that $\bar u(\rho,\tau+M)$ is a super-solution, where $M$ is a large constant. Thus, $|\tilde{\phi}|\leq 2 \bar u(\rho,\tau+M)$. By letting $M\to \infty$ and the arbitrariness of $\epsilon$, we get $\tilde{\phi} = 0$, a contradiction. The proof is complete.
\end{proof}

\begin{proof}[Proof of Proposition \ref{propmode1}]
From Lemma \ref{lemmamode1}, for any $\tau_1 > \tau_0$ with $\tau_0$ fixed sufficiently large, we have
\begin{equation*}
|\phi^1(y,\tau)|\lesssim\tau^{-\nu_1'}(1+|y|)^{-a_1}\|h^1\|_{2+a_1, \tau_1}\text{ for all }\tau\in (\tau_0, \tau_1), \,\,y\in \mathbb{R}^4
\end{equation*}
and
\begin{equation*}
|c(\tau)|\leq \tau^{-\nu_1'}R^{a_1}\|h^1\|_{2+a_1,\tau_1}\text{ for all }\tau\in (\tau_0,\tau_1).
\end{equation*}
By assumption, $\|h^1\|_{\nu_1,2+a_1} < +\infty$ and $\|h^1\|_{2+a_1, \tau_1}\leq \|h^1\|_{\nu_1,2+a_1}$ for an arbitrary $\tau_1$. It then follows that
\begin{equation*}
|\phi^1(y,\tau)|\lesssim\tau^{-\nu_1'}(1+|y|)^{-a_1}\|h^1\|_{\nu_1,2+a_1}\text{ for all }\tau\in (\tau_0, \tau_1),\,\, y\in \mathbb{R}^4
\end{equation*}
and
\begin{equation*}
|c(\tau)|\leq \tau^{-\nu_1'}R^{a_1}\|h^1\|_{\nu_1,2+a_1}\text{ for all }\tau\in (\tau_0, \tau_1).
\end{equation*}
By the arbitrariness of $\tau_1$, we have
\begin{equation*}
|\phi^1(y,\tau)|\lesssim\tau^{-\nu_1'}(1+|y|)^{-a_1}\|h^1\|_{\nu_1,2+a_1}\text{ for all }\tau\in (\tau_0, +\infty),\,\, y\in \mathbb{R}^4
\end{equation*}
and
\begin{equation*}
|c(\tau)|\leq \tau^{-\nu_1'}R^{a_1}\|h\|_{\nu_1,2+a_1}\text{ for all }\tau\in (\tau_0, +\infty).
\end{equation*}
The gradient estimates follows from the scaling argument and the standard parabolic theory. The proof is complete.
\end{proof}

\medskip

\noindent {\bf 3. Construction at higher modes $j\geq 5$.}

\medskip

For higher modes $j\geq 5$, we recall that
$$h^{\perp}=\sum\limits_{j=5}^{\infty} h_j(r,t)\Theta_j,~\phi^{\perp}[h^{\perp}]=\sum\limits_{j=5}^{\infty} \phi_j(r,t)\Theta_j$$
and let $\phi^{\perp}[h^{\perp}]$ solve the following problem
\begin{equation*}
\begin{cases}
\la^2\phi_t=\Delta_y \phi +3U^2(y)\phi+h^{\perp},~&\mbox{ in } \mathcal D_{2R}\times(0,T),\\
\phi=0,~&\mbox{ on } \partial \mathcal D_{2R}\times (0,T),\\
\phi(\cdot,0)=0,~&\mbox{ in } \mathcal D_{2R}.\\
\end{cases}
\end{equation*}
Similarly, it follows from \cite[Section 7]{Green16JEMS} that
\begin{equation}\label{highermodes}
|\phi^{\perp}(y,t)|+(1+|y|)|\nabla\phi^{\perp}(y,t)|\lesssim \la_*^{\nu} \frac{1}{1+|y|^a}\|h^{\perp}\|_{\nu,2+a}.
\end{equation}

\begin{proof}[Proof of Proposition \ref{lineartheory}]
Recall that
$$\phi[h]=\phi^0[h^0]+\phi^1[h^1]+\phi^{\perp}[h^{\perp}].$$
The validity of Proposition \ref{lineartheory} is concluded from Proposition \ref{propmode0}, Proposition \ref{propmode1} and \eqref{highermodes}. The proof is complete.
\end{proof}

\medskip


\section{Solving the inner--outer gluing system}\label{sec-innerouter}

\medskip

In this section, we shall solve the inner--outer gluing system by the linear theories developed in Section \ref{sec-linearouter} and Section \ref{sec-linearinner}, and the Schauder fixed point theorem. Our goal is to find a solution $(\phi^0,\phi^1,\phi^{\perp},\psi,\la,\xi)$ to the inner--outer gluing system in Section \ref{sec-inneroutergluingscheme} so that the desired blow-up solution is constructed. We shall solve the inner--outer gluing system in the function space $\mathcal X$ defined in \eqref{def-mX}. First, we make some assumptions about the parameter functions.

We write
$$\la_*(t)=\frac{2}{n-2}\frac{|\log T|^{\frac{2}{n-2}}(T-t)}{|\log(T-t)|^{\frac{n}{n-2}}}$$
and assume that for some numbers $c_1,c_2>0$,
$$c_1|\dot\la_*(t)|\leq |\dot\la(t)|\leq c_2 |\dot\la_*(t)|~\mbox{ for all }~t\in(0,T).$$

In Section \ref{subsec-outer} and Section \ref{subsec-inner}, for given $\|\phi^0\|_{0,\sigma,\nu,a}$, $\|\phi^1\|_{{\rm in},\nu_1,a_1}$, $\|\phi^{\perp}\|_{{\rm in},\nu,a}$, $\|\psi\|_*$, $\|Z^*\|_{\infty}$, $\|\la\|_F$, $\|\xi\|_G$ bounded,
we shall first estimate right hand sides $\mathcal G(\phi,\psi,\la,\xi)$ and $\mathcal H(\phi,\psi,\la,\xi)$ in the inner and outer problems. Here the above norms are defined in \eqref{def-normm0}, \eqref{def-normnua}, \eqref{def-norm*}, \eqref{def-normla} and \eqref{def-normxi}.

\medskip

\subsection{The outer problem: estimates of \texorpdfstring{$\mathcal G$}{G}}\label{subsec-outer}

\medskip

Recall from \eqref{outer} that the outer problem
\begin{equation*}
\psi_t =\Delta_{(r,z)} \psi+\frac{n-4}{r}\partial_r \psi  + \mathcal G(\phi,\psi,\la,\xi)~~\mbox{ in }\mathcal D\times (0,T)
\end{equation*}
where
$\mathcal G(\phi,\psi,\la,\xi)$ is defined in \eqref{def-mG}.

In order to apply the linear theory Proposition \ref{outer-apriori}, we estimate all the terms in $\mathcal G(\phi,\psi,\la,\xi)$. Define
$$\mathcal G(\phi,\psi,\la,\xi)=g_1+g_2+g_3$$
with
$$
\begin{aligned}
&g_1:=3\la^{-2}(1-\eta_R)U^2(y)(\psi+Z^*+\Psi_0+\Psi_1),\\
&g_2:=\la^{-3}\left[(\Delta_y \eta_R) \phi+2\nabla_y\eta_R\cdot\nabla_y \phi-\la^2\phi\partial_t\eta_R \right],\\
&g_3:=(1-\eta_R)\mathcal K_1[\la,\xi]+\mathcal S_{\rm out}[\la,\xi]-\mathcal S_{\rm out}[\la_*,\xi_*]+(1-\eta_R)\mathcal N(\mathtt P+\Psi_0+\Psi_1).\\
\end{aligned}
$$
To estimate $g_1$, we need to estimate the corrections $\Psi_0$ and $\Psi_1$ defined in \eqref{def-Psi0} and \eqref{def-Psi1}. This is done in Appendix \ref{App-est-Psi_0}.

\medskip

\noindent {\bf Estimate of $g_1$.}

\medskip

Since we shall solve $\psi$ in the function space $X_{\psi}$ defined in \eqref{def-fcnspaces}, we get
\begin{align*}
g_1=&~3\la^{-2}(1-\eta_R)U^2(y)(\psi+Z^*+\Psi_0+\Psi_1)\\
\lesssim&~\frac{R^{-2}(t)\la_*^{\nu}(0)R^{2-\alpha}(0)|\log T|}{|(r,z)-\xi(t)|^2}\chi_{\{|(r,z)-\xi(t)|\geq\la_*R\}}\|\psi\|_*\\
&~+\frac{R^{-2}(t)}{|(r,z)-\xi(t)|^2}\chi_{\{|(r,z)-\xi(t)|\geq\la_*R\}}\|Z^*\|_{\infty}+\frac{R^{-2}(t)|\log(T-t)|}{|(r,z)-\xi(t)|^2}\chi_{\{|(r,z)-\xi(t)|\geq\la_*R\}}\|\la\|_{\infty}.
\end{align*}
So by the choice of the weight $\varrho_2$ as in \eqref{def-weights}, we have
\begin{equation}\label{g1**}
\|g_1\|_{**}\lesssim T^{\epsilon_0}(\|\psi\|_*+\|Z^*\|_{\infty}+\|\la\|_{\infty}+1)
\end{equation}
provided
$$ 
\nu+\frac{\alpha}{2}-\nu_2>0,\quad
1-\nu_2>0.
$$ 
Here $\epsilon_0$ is a small positive number, and we have used 
$$
R(t)=(T-t)^{-1/2}|\log(T-t)|^{\theta},\quad \theta\in \left(\frac{n-1}{n-2},~\frac{n}{n-2}\right).
$$

\medskip

\noindent {\bf Estimate of $g_2$.}

\medskip

Thanks to the cut-off, $g_2$ is supported in
$$\left\{(r,z,t)\in\mathcal D\times(0,T):~\la_*R\leq|(r,z)-\xi(t)|\leq 2\la_*R,~t\in(0,T)\right\},$$
and we have
\begin{equation*}
\begin{aligned}
&\quad g_2= \la^{-3}\left[(\Delta_y \eta_R) \phi+2\nabla_y\eta_R\cdot\nabla_y \phi-\la^2(\partial_t\eta_R)\phi\right]+\frac{(n-4)\la^{-1}}{r}\,\phi\partial_{r}\eta_R\\
&\lesssim \la_*^{\nu-3}R^{-2-a}\log R\left(1+r-\sqrt{2(n-4)(T-t)}+\dot\la_* \la_* R^2\right)\chi_{\{|(r,z)-\xi(t)|\sim\la_*R\}}\\
&\qquad \times\left(\|\phi^0\|_{0,\sigma,\nu,a}+\|\phi^1\|_{{\rm in},\nu_1,a_1}+\|\phi^{\perp}\|_{{\rm in},\nu,a}\right)\\
&\lesssim (R^{\alpha-a}\log R+\la_*^{1/2}R^{1+\alpha-a} \log R)\varrho_1\left(\|\phi^0\|_{0,\sigma,\nu,a}+\|\phi^1\|_{{\rm in},\nu_1,a_1}+\|\phi^{\perp}\|_{{\rm in},\nu,a}\right).
\end{aligned}
\end{equation*}
So it follows that
$$ 
\left\|g_2\right\|_{**}\lesssim T^{\epsilon_0}\left(\|\phi^0\|_{0,\sigma,\nu,a}+\|\phi^1\|_{{\rm in},\nu_1,a_1}+\|\phi^{\perp}\|_{{\rm in},\nu,a}\right)
$$ 
provided
$$ 
0<\alpha<a<1.\\
$$ 
Here $\epsilon_0$ is a small positive number.

\medskip

\noindent {\bf Estimate of $g_3$.}

\medskip

We want to estimate $g_3=(1-\eta_R)\mathcal K_1[\la,\xi]+\mathcal S_{\rm out}[\la,\xi]-\mathcal S_{\rm out}[\la_*,\xi_*]+(1-\eta_R)\mathcal N(\mathtt P+\Psi_0+\Psi_1)$. Recall the definition of $\mathcal K_1[\la,\xi]$ in \eqref{def-mK1}. Due to the cut-off $(1-\eta_R)\eta_*$, 
\begin{equation*}
\begin{aligned}
|(1-\eta_R)\mathcal K_1[\la,\xi]|
\lesssim \la^{-2}R^{-4}|\dot\la| \lesssim |\log(T-t)|^{\frac{n}{n-2}-4\theta} \lesssim |\log T|^{-\epsilon_0}\|\la\|_{\infty} \varrho_3,
\end{aligned}
\end{equation*}
then one has
$$ 
\|(1-\eta_R)\mathcal K_1[\la,\xi]\|_{**}\lesssim |\log T|^{-\epsilon_0}
$$ 
for some $\epsilon_0>0$.

Finally, we compute the nonlinear terms
\begin{align*}
(1-\eta_R)\mathcal N(\mathtt P+\Psi_0+\Psi_1) 
\lesssim&~\frac{\la^{-1}}{1+|y|^2}\left[(\la^{-1}\eta_R\phi)^2+\Psi_0^2+\psi^2+(Z^*)^2+\Psi_1^2\right](1-\eta_R)\\
\lesssim&~\la_*^{\nu}R^{2\sigma(4-a)+\alpha-8}(\log R)^2\|\phi^0\|^2_{0,\sigma,\nu,a}\varrho_1+\la_*^{2\nu_1-\nu}R^{\alpha-2a_1}\|\phi^1\|^2_{{\rm in},\nu_1,a_1}\varrho_1\\
&~+\la_*^{\nu}R^{\alpha-2a}\|\phi^{\perp}\|^2_{{\rm in},\nu,a}\varrho_1+(|\log(T-t)|^2|\dot\la_*|^2+|\log T|^{-\frac{4}{n-2}})\la_*^{1-\nu_2}\varrho_2\\
&~+\la_*^{1-\nu_2}(t)\la_*^{2\nu}(0)R^{4-2\alpha}(0)|\log T|^2\varrho_2\|\psi\|_*^2+\la^{1-\nu_2}\varrho_2\|Z^*\|_{\infty}^2\\
&~+\la_*^{1-\nu_2}|\log T|^{-2\epsilon_0}\varrho_2,
\end{align*}
where we have used \eqref{est-Psi_0} and \eqref{est-Psi_1}. Therefore, we obtain that for $\epsilon_0>0$
$$
\begin{aligned}
&~\|(1-\eta_R)\mathcal N(\mathtt P+\Psi_0+\Psi_1)\|_{**}\\
\lesssim &~T^{\epsilon_0}\left(\|\phi^0\|_{0,\sigma,\nu,a}^2+\|\phi^1\|^2_{{\rm in},\nu_1,a_1}+\|\phi^{\perp}\|_{{\rm in},\nu,a}^2+\|\psi\|_*^2+\|Z^*\|_{\infty}^2+\|\la\|_{\infty}+1\right)
\end{aligned}
$$
provided
\begin{equation}\label{g3**3cond}
\begin{aligned}
&~\nu-\frac12\left(2\sigma(4-a)+\alpha-8\right)>0,\quad
2\nu_1-\nu-\frac12(\alpha-2a_1)>0,\\
&~\nu-\frac12(\alpha-2a)>0,\quad
\nu_2<1,\quad
2\nu-\nu_2+\alpha-1>0.
\end{aligned}
\end{equation}

Collecting \eqref{g1**}--\eqref{g3**3cond},  we conclude that for a fixed number $\epsilon_0>0$
\begin{equation}\label{outer-contraction}
\|\mathcal G\|_{**}\lesssim |\log T|^{-\epsilon_0}\left(\|\psi\|_*+\|Z^*\|_{\infty}+\|\phi^0\|_{0,\sigma,\nu,a}+\|\phi^1\|_{{\rm in},\nu_1,a_1}+\|\phi^{\perp}\|_{{\rm in},\nu,a}+\|\la\|_{\infty}+\|\xi\|_{G}+1\right)
\end{equation}
with the parameters $a,a_1,\alpha,\nu,\nu_1,\nu_2,\sigma$ chosen in the following range
$$ 
\begin{aligned}
&~\nu+\frac{\alpha}{2}-\nu_2>0,\quad
0<\alpha<a<1,\quad\nu-\frac12\left(2\sigma(4-a)+\alpha-8\right)>0,\\
&~2\nu_1-\nu-\frac12(\alpha-2a_1)>0,\quad
\nu-\frac12(\alpha-2a)>0,\quad
\nu_2<1,\quad
2\nu-\nu_2+\alpha-1>0.
\end{aligned}
$$

\medskip

\subsection{The inner problems: estimates of \texorpdfstring{$\mathcal H^0$}{H0}, \texorpdfstring{$\mathcal H^1$}{H1} and \texorpdfstring{$\mathcal H^{\perp}$}{Hperp}}\label{subsec-inner}

\medskip

Recall from \eqref{inner} that the inner problem is the following
\begin{equation*}
\begin{aligned}
\la^2 \phi_t = \Delta_y \phi +3U^2(y)\phi+\mathcal H(\phi,\psi,\la,\xi)~~\mbox{ in }\mathcal D_{2R}\times (0,T)
\end{aligned}
\end{equation*}
with $\mathcal H(\phi,\psi,\la,\xi)$ defined in \eqref{def-mH}. Since the inner--outer gluing relies on delicate analysis of the space-time decay of solutions, we further decompose the inner problem \eqref{inner} into three different spherical harmonic modes
\begin{equation*}
\begin{cases}
\la^2 \phi^0_t = \Delta_y \phi^0 +3U^2(y)\phi^0+\mathcal H^0(\phi,\psi,\la,\xi)~~&\mbox{ in }\mathcal D_{2R}\times (0,T)\\
\phi^0(\cdot,0)=0~~&\mbox{ in }\mathcal D_{2R}\\
\end{cases}
\end{equation*}

\begin{equation*}
\begin{cases}
\la^2 \phi^1_t = \Delta_y \phi^1 +3U^2(y)\phi^1+\mathcal H^1(\phi,\psi,\la,\xi)~~&\mbox{ in }\mathcal D_{2R}\times (0,T)\\
\phi^1(\cdot,0)=0~~&\mbox{ in }\mathcal D_{2R}\\
\end{cases}
\end{equation*}

\begin{equation*}
\begin{cases}
\la^2 \phi^{\perp}_t = \Delta_y \phi^{\perp} +3U^2(y)\phi^{\perp}+\mathcal H^{\perp}(\phi,\psi,\la,\xi)~~&\mbox{ in }\mathcal D_{2R}\times (0,T)\\
\phi^{\perp}(\cdot,0)=0~~&\mbox{ in }\mathcal D_{2R}\\
\end{cases}
\end{equation*}
with
\begin{equation}\label{def-mHm0}
\mathcal H^0(\phi,\psi,\la,\xi)=\int_{\mathbb{S}^3} \mathcal H(\phi,\psi,\la,\xi) \Theta_0(\theta) d\theta,
\end{equation}
$$ 
\mathcal H^1(\phi,\psi,\la,\xi)=\sum\limits_{j=1}^4\left(\int_{\mathbb{S}^3} \mathcal H(\phi,\psi,\la,\xi) \Theta_j(\theta) d\theta\right) \Theta_j,
$$ 
and
\begin{equation}\label{def-mHmh}
\mathcal H^{\perp}(\phi,\psi,\la,\xi)=\sum\limits_{j\geq 5}\left(\int_{\mathbb{S}^3} \mathcal H(\phi,\psi,\la,\xi) \Theta_j(\theta) d\theta\right) \Theta_j,
\end{equation}
where $\Theta_j~(j=0,1,\cdots)$ are spherical harmonics. From the linear theory in Section \ref{sec-linearinner}, we know that for $\mathcal H=\mathcal H^0+\mathcal H^1+\mathcal H^{\perp}$ satisfying
$$\|\mathcal H^0\|_{\nu,2+a},~\|\mathcal H^1\|_{\nu_1,2+a_1},~\|\mathcal H^{\perp}\|_{\nu,2+a}<+\infty,$$
there exists a solution $(\phi^0,\phi^1,\phi^{\perp},c^0,c^{\ell})$ ($\ell=1,\cdots,4$) solving the projected inner problems
\begin{equation}\label{eqn-mode000}
\begin{cases}
\la^2 \phi^0_t = \Delta_y \phi^0 +3U^2(y)\phi^0+\mathcal H^0(\phi,\psi,\la,\xi)+c^0 Z_5~~&\mbox{ in }\mathcal D_{2R}\times (0,T)\\
\phi^0(\cdot,0)=0~~&\mbox{ in }\mathcal D_{2R}\\
\end{cases}
\end{equation}

$$ 
\begin{cases}
\la^2 \phi^1_t = \Delta_y \phi^1 +3U^2(y)\phi^1+\mathcal H^1(\phi,\psi,\la,\xi)+\sum\limits_{\ell=1}^4 c^{\ell}Z_{\ell}~~&\mbox{ in }\mathcal D_{2R}\times (0,T)\\
\phi^1(\cdot,0)=0~~&\mbox{ in }\mathcal D_{2R}\\
\end{cases}
$$ 

\begin{equation}\label{eqn-modehhh}
\begin{cases}
\la^2 \phi^{\perp}_t = \Delta_y \phi^{\perp} +3U^2(y)\phi^{\perp}+\mathcal H^{\perp}(\phi,\psi,\la,\xi)~~&\mbox{ in }\mathcal D_{2R}\times (0,T)\\
\phi^{\perp}(\cdot,0)=0~~&\mbox{ in }\mathcal D_{2R}\\
\end{cases}
\end{equation}
and the inner solution $\phi[\mathcal H]=\phi^0[\mathcal H^0]+\phi^1[\mathcal H^1]+\phi^{\perp}[\mathcal H^{\perp}]$ with proper space-time decay can be found ensuring the inner--outer gluing to be carried out. First, we choose all the parameters such that
$$\|\mathcal H^0\|_{\nu,2+a},~\|\mathcal H^1\|_{\nu_1,2+a_1},~\|\mathcal H^{\perp}\|_{\nu,2+a}<+\infty.$$
To this end, we first give some estimates for $\mathcal H$.

\noindent $\bullet$  
 By \eqref{est-Psi_0}, we have
\begin{equation}\label{est-mH1}
\begin{aligned}
&\quad \left|3\la U^{2}(y)[\Psi_0(\la y+\xi,t)+\psi(\la y+\xi,t)+Z^*(\la y+\xi,t)]\right|\\
&\lesssim \frac{\la_*(t)}{1+|y|^4}\left[|\dot\la_*|\left(\log{\la_*}+\log(1+|y|)\right)+|\log T|^{-\frac{2}{n-2}}+\la_*^{\nu}(0)R^{2-\alpha}(0)|\log T|\|\psi\|_{*}+\|Z^*\|_{\infty}\right].
\end{aligned}
\end{equation}

\medskip

\noindent $\bullet$
We also have
\begin{align*}
&\quad\left|\la\left[\dot\la (\nabla_y\phi\cdot y+\phi)+ \nabla_y \phi\cdot\dot\xi\right]+\frac{(n-4)\la^2}{r}\phi\partial_r \eta_R\right|\\
&\lesssim \left(\la_*|\dot\la_*|+\frac{\la_*}{R\sqrt{T-t}}\right)\Bigg(\frac{\la_*^{\nu}R^{\sigma(4-a)}\log R}{1+|y|^{4}}\|\phi^0\|_{0,\sigma,\nu,a}+\frac{\la_*^{\nu_1}}{1+|y|^{a_1}}\|\phi^1\|_{{\rm in},\nu_1,a_1}\\
&~\qquad\qquad\qquad+\frac{\la_*^{\nu}}{1+|y|^{a}}\|\phi^{\perp}\|_{{\rm in},\nu,a}\Bigg) + \la_*\dot\xi \frac{\la_*^{\nu_1}}{1+|y|^{1+a_1}}\|\phi^1\|_{{\rm in},\nu_1,a_1}.
\end{align*}

\medskip

\noindent $\bullet$ Using \eqref{est-Psi_0} and \eqref{est-Psi_1}, we evaluate
\begin{align}\notag
&\quad \left|\la^3\mathcal N(\mathtt P+\Psi_0+\Psi_1)+\la^3 \mathcal K[\la,\xi]+3\la U^{2}(y)\Psi_1\right|\\ \notag
&\lesssim \frac{\la_*^{2\nu}R^{2\sigma(4-a)}(\log R)^2}{1+|y|^{10}}\|\phi^0\|^2_{0,\sigma,\nu,a}+\frac{\la_*^{2\nu_1}}{1+|y|^{2+2a_1}}\|\phi^1\|^2_{{\rm in},\nu_1,a_1}+\frac{\la_*^{2\nu}}{1+|y|^{2+2a}}\|\phi^{\perp}\|^2_{{\rm in},\nu,a}\\ \notag
&\quad+\frac{\la_*^2(t)\la_*^{2\nu}(0)R^{4-2\alpha}(0)|\log T|^2}{1+|y|^2}\|\psi\|_*^2+\frac{\la_*^2}{1+|y|^2}\Big(|\dot\la_*|^2|\log(T-t)|^2+|\log T|^{-\frac{4}{n-2}}\Big)\\ \notag
&\quad+\frac{\la_*^2}{1+|y|^2}\|Z^*\|^2_{\infty}+\frac{\la_*^{2}|\log T|^{-2\epsilon_0}}{1+|y|^2}+\frac{\la_*\dot\la_*}{(1+|y|^2)^2}+\frac{\la_*|\dot\xi|}{1+|y|^3}\\ \label{est-mH3}
&\quad+\frac{\la_*}{\sqrt{T-t}}\frac{y_1}{1+|y|^4}+\frac{\la_*|\log T|^{-\epsilon_0}}{1+|y|^4}.
\end{align}

\medskip

\noindent $\bullet$ In the spherical coordinates, the projection of
$$\la^{-2}(t)\nabla U(y)\cdot\dot\xi(t)+\frac{n-4}{\la(t)y_1+\xi_r(t)}\la^{-2}(t)\partial_{y_1}U(y)$$
on mode $0$ is given by
\begin{equation}\label{proj-mode0}
\begin{aligned}
&\quad\int_{\mathbb S^3}\frac{n-4}{\la(t)y_1+\xi_r(t)}\la^{-2}(t)\partial_{y_1}U(y) \Theta_0 d\theta\\
&=C\frac{\la^{-2}}{(1+|y|^2)^2}\int_0^{\pi} \frac{|y|\cos\theta\sin^2\theta}{\la|y|\cos\theta+\xi_r} d\theta=-C\frac{\la^{-2}}{(1+|y|^2)^2}\frac{\la|y|^2\pi}{2(\xi_r+\sqrt{\xi_r^2-(\la|y|)^2})^2},
\end{aligned}
\end{equation}
where $C$ is a constant. Note that since our choice of $\xi_r$ is $$\xi_r\sim\sqrt{2(n-4)(T-t)},$$
the above projection on mode $0$ behaves exactly like the first error $\mathcal E_0$ defined in \eqref{def-mE0}, and direct computations show that the sum of  these two terms does not vanish. So we can deal with \eqref{proj-mode0} by slightly modifying the first correction $\Psi_0$. Here we omit the details.

\medskip

\noindent $\bullet$ Similarly, the projection of
$$\la^{-2}(t)\nabla U(y)\cdot\dot\xi(t)+\frac{n-4}{\la(t)y_1+\xi_r(t)}\la^{-2}(t)\partial_{y_1}U(y)$$
on mode $1$ can be computed as
\begin{align*}
&\quad\int_{\mathbb S^3}\la^{-2}(t)\partial_{y_1}U(y)\left(\frac{n-4}{\la(t)y_1+\xi_r(t)}+\dot\xi_r\right) \Theta_1(\theta) d\theta\\
&=-\int_{\mathbb S^3}\la^{-2}(t)\partial_{y_1}U(y)\frac{(n-4)\la(t) y_1}{\xi_r(t)[\la(t)y_1+\xi_r(t)]} \Theta_1(\theta) d\theta\\
&=C'\frac{\la^{-1}(t)|y|^2}{\xi_r(t)(1+|y|^2)^2}\int_0^{\pi} \frac{\cos^3\theta\sin^2\theta}{\la|y|\cos\theta+\xi_r} d\theta\\
&=C'\frac{\la^{-1}(t)|y|^2}{\xi_r(t)(1+|y|^2)^2}\frac{\left((\la|y|)^4+4(\la|y|)^2\xi_r^2-8\xi_r^4+8\xi_r^3\sqrt{\xi_r^2-(\la|y|)^2}\right)\pi}{8(\la|y|)^5},
\end{align*}
where $C'$ is a constant and we have used that $\dot\xi_r\sim-\frac{n-4}{\xi_r}$. Note that in $\mathcal D_{2R}$, namely $|y|\leq 2R$, we have $\la|y|\ll\xi_r$ for $T$ sufficiently small. Therefore, by directly expanding the above expression, we obtain
\begin{equation}\label{proj-mode1}
\int_{\mathbb S^3}\la^{-2}(t)\partial_{y_1}U(y)\left(\frac{n-4}{\la(t)y_1+\xi_r(t)}+\dot\xi_r\right) \Theta_1(\theta) d\theta\lesssim \frac{1}{\xi_r^3(1+|y|)}.
\end{equation}

Then we estimate $\mathcal H$ in three different modes.

\medskip

\noindent {\bf Estimate of $\mathcal H^0$.}

\medskip

By \eqref{est-mH1}--\eqref{proj-mode0}, we obtain
\begin{align*}
\|\mathcal H^0\|_{\nu,2+a}\lesssim &~ \la_*^{1-\nu}(|\dot\la_*||\log(T-t)|+|\log T|^{-\frac2{n-2}})+\la_*^{1-\nu}(t)\la_*^{\nu}(0)R^{2-\alpha}(0)|\log T|\|\psi\|_*\\
&~+\la_*^{1-\nu}\|Z^*\|_{\infty}+\la_*^{\frac12}R^{\sigma(4-a)}\log R\|\phi^0\|_{0,\sigma,\nu,a}+\la_*|\dot\la_*|R^{\sigma(4-a)}\log R\|\phi^0\|_{0,\sigma,\nu,a}\\
&~+\la_*^{\nu}R^{2\sigma(4-a)}(\log R)^2\|\phi^0\|^2_{0,\sigma,\nu,a}+\la_*^{2-\nu}(t)R^a(t)\la_*^{2\nu}(0)R^{4-2\alpha}(0)|\log T|^2\|\psi\|^2_*\\
&~+\la_*^{2-\nu}R^a(|\dot\la_*|^2 |\log(T-t)|^2+|\log T|^{-\frac4{n-2}})+\la_*^{2-\nu}R^a\|Z^*\|_{\infty}^2+\la_*^{2-\nu}R^a|\log T|^{-2\epsilon_0}\\
&~+\la_*^{1-\nu}|\dot\la_*|+\la_*^{1-\nu}|\log T|^{-\epsilon_0},
\end{align*}
from which we conclude that
\begin{equation}\label{est-mHm0}
\|\mathcal H^0\|_{\nu,2+a} \lesssim T^{\epsilon_0}\left(\|\phi^0\|_{0,\sigma,\nu,a}+\|\psi\|_*+\|Z^*\|_{\infty}+\|\la\|_{\infty}+\|\xi\|_G+1\right)
\end{equation}
provided
$$ \begin{aligned}
&~\nu<1,\quad
\alpha>0,\quad
\frac12-\frac12\sigma(4-a)>0,\quad \nu-\sigma(4-a)>0,\\
&~2-\nu-\frac12 a>0,\quad
\nu+\alpha-\frac{a}{2}>0.
\end{aligned}
$$ 

\medskip

\noindent {\bf Estimate of $\mathcal H^1$.}

\medskip

From \eqref{est-mH1}--\eqref{est-mH3} and \eqref{proj-mode1}, we have
\begin{align*}
\|\mathcal H^1\|_{\nu_1,2+a_1}\lesssim &~ \la_*^{1-\nu_1}(|\log(T-t)||\dot\la_*|+|\log T|^{-\frac{2}{n-2}})+\la_*^{1-\nu_1}(t)\la_*^{\nu}(0)R^{2-\alpha}(0)|\log T|\|\psi\|_*\\
&~+\la_*^{1-\nu_1}\|Z^*\|_{\infty}+\la_*|\dot\la_*| R^2\|\phi^1\|_{{\rm in},\nu_1,a_1}+\frac{\la_*}{\sqrt{T-t}}R\|\phi^1\|_{{\rm in},\nu_1,a_1}\\
&~+\frac{\la_*^{2-\nu_1}}{\sqrt{T-t}}|\dot\la_*|R^{1+a_1}+\la_*^{\nu_1}\|\phi^1\|_{{\rm in},\nu_1,a_1}^2+\la_*^{2-\nu_1}R^{a_1}(|\dot\la_*|^2 |\log(T-t)|^2+|\log T|^{-\frac4{n-2}})\\
&~+\la_*^{2-\nu_1}(t)R^{a_1}(t)\la_*^{2\nu}(0)R^{4-2\alpha}(0)|\log T|^2\|\psi\|_*^2\\
&~+\la_*^{2-\nu_1} R^{a_1}\|Z^*\|_{\infty}^2+\la_*^{2-\nu_1}R^{a_1}|\log T|^{-2\epsilon_0}+\la_*^{1-\nu_1}|\log T|^{-\epsilon_0}.\\
\end{align*}
Therefore, we obtain that for some $\epsilon_0>0$
\begin{equation}\label{est-mHm1}
\|\mathcal H^1\|_{\nu_1,2+a_1}\lesssim T^{\epsilon_0}\left(\|\phi^1\|_{{\rm in},\nu_1,a_1}+\|\psi\|_*+\|Z^*\|_{\infty}+\|\la\|_{\infty}+\|\xi\|_G+1\right)
\end{equation}
provided
$$ 
\begin{aligned}
&~\nu_1<1,\quad
\nu-\nu_1+\frac{\alpha}{2}>0,\quad \theta<\frac{n}{n-2},\quad
\frac32-\nu_1-\frac12(1+a_1)>0,\\
&~\nu_1>0,\quad 2\nu-\nu_1+\alpha-\frac{a_1}{2}>0.
\end{aligned}
$$ 

\medskip

\noindent {\bf Estimate of $\mathcal H^{\perp}$.}

\medskip

Using \eqref{est-mH1}--\eqref{est-mH3}, we get
\begin{align*}
\|\mathcal H^{\perp}\|_{\nu,2+a}\lesssim &~  \la_*^{1-\nu}|\dot\la_*||\log(T-t)|+\la_*^{1-\nu}(t)\la_*^{\nu}(0)R^{2-\alpha}(0)|\log T|\|\psi\|_*\\
&~+\la_*^{1-\nu}\|Z^*\|_{\infty}+\la_* R^2|\dot\la_*|\|\phi^{\perp}\|_{{\rm in},\nu,a}+\la_*^{\nu}\|\phi^{\perp}\|^2_{{\rm in},\nu,a}+\la_*^{2-\nu}R^a(|\dot\la_*|^2 |\log(T-t)|^2+|\log T|^{-\frac4{n-2}})\\
&~+\la_*^{2-\nu}(t)R^a(t)\la_*^{2\nu}(0)R^{4-2\alpha}(0)|\log T|^2\|\psi\|^2_*\\
&~+\la_*^{2-\nu}R^a\|Z^*\|_{\infty}^2+\la_*^{2-\nu}R^a|\log T|^{-2\epsilon_0}+\la_*^{1-\nu}|\log T|^{-\epsilon_0}.
\end{align*}
Thus, one has
\begin{equation}\label{est-mHmp}
\|\mathcal H^{\perp}\|_{\nu,2+a}\lesssim T^{\epsilon_0}\left(\|\phi^{\perp}\|_{{\rm in},\nu,a}+\|\psi\|_*+\|Z^*\|_{\infty}+\|\la\|_{\infty}+\|\xi\|_G+1\right)
\end{equation}
provided
\begin{equation}\label{cond-mHmp}
\begin{aligned}
0<\nu<1,\quad
\alpha>0,\quad
\theta<\frac{n}{n-2},\quad
2-\nu-\frac12a>0,\quad
\nu+\alpha-\frac{a}{2}>0.
\end{aligned}
\end{equation}

Collecting \eqref{est-mHm0}--\eqref{cond-mHmp}, we conclude that for some $\epsilon_0>0$
$$ 
\begin{aligned}
&~\|\mathcal H^0\|_{\nu,2+a}+\|\mathcal H^1\|_{\nu_1,2+a_1}+\|\mathcal H^{\perp}\|_{\nu,2+a}\\
\lesssim&~ T^{\epsilon_0}\bigg(\|\phi^0\|_{0,\sigma,\nu,a}+\|\phi^1\|_{{\rm in},\nu_1,a_1}+\|\phi^{\perp}\|_{{\rm in},\nu,a}+\|\psi\|_*+\|\la\|_{\infty}+\|\xi\|_{G}+\|Z^*\|_{\infty}+1\bigg)
\end{aligned}
$$ 
provided the parameters $\theta,~a,~a_1,~\alpha,~\nu,~\nu_1,~\sigma$ satisfy
\begin{align*}
&~\nu<1,\quad
\alpha>0,\quad
\frac12-\frac12\sigma(4-a)>0,\quad
\nu-\sigma(4-a)>0,\quad 2-\nu-\frac12 a>0,\\
&~\nu+\alpha-\frac{a}{2}>0,\quad
0<\nu_1<1,\quad
\nu-\nu_1+\frac{\alpha}{2}>0,\quad
\theta<\frac{n}{n-2},\\
&~
\frac32-\nu_1-\frac12(1+a_1)>0,\quad 2\nu-\nu_1+\alpha-\frac{a_1}{2}>0.
\end{align*}

\medskip

\subsection{The parameter problems}

\medskip

From \eqref{eqn-mode000}--\eqref{eqn-modehhh}, we need to adjust the parameter functions $\la(t)$, $\xi(t)$ such that
$$c^0[\la,\xi,\Psi^*]=0,~~c^{\ell}[\la,\xi,\Psi^*]=0,~~\ell=1,\cdots,4,$$
where
\begin{equation}\label{def-cmode0}
c^0[\la,\xi,\Psi^*]=-\frac{\int_{\mathcal D_{2R_*}}\mathcal H^0 Z_5 dy}{\int_{\mathcal D_{2R_*}}|Z_5|^2 dy}-\mathcal O[\mathcal H^0],
\end{equation}
\begin{equation}\label{def-cmode1}
c^{\ell}[\la,\xi,\Psi^*]=-\frac{\int_{\mathcal D_{2R}}\mathcal H^1 Z_{\ell} dy}{\int_{\mathcal D_{2R}}|Z_{\ell}|^2 dy}~\mbox{ for }~\ell=1,\cdots,4.
\end{equation}
It turns out that we can easily achieve at the translation mode \eqref{def-cmode1}, but the scaling mode \eqref{def-cmode0} is more complicated.

\medskip

\subsubsection{The reduced problem of \texorpdfstring{$\xi(t)$}{xi(t)}}

\medskip

We first consider the reduced equation for $\xi(t)=(\xi_r(t),\xi_z(t)).$ Notice that \eqref{def-cmode1} is equivalent to
\begin{equation*}
\int_{\mathcal D_{2R}} \mathcal H^1(\phi,\psi,\la,\xi)(y,t) Z_i(y) dy=0,~\mbox{ for all }t\in (0,T),~i=1,\cdots,4.
\end{equation*}
Recall that
$$\xi_r(t)=\sqrt{2(n-4)(T-t)}+\xi_{r,1}(t),~~\xi_z(t)=z_0+\xi_{z,1}(t)$$
and write $\Psi^*=\psi+Z^*.$ Then for $i=1,\cdots,4,$
$$\int_{\mathcal D_{2R}} \mathcal H^1(\phi,\psi,\la,\xi)(y,t) Z_i(y) dy=0$$
yield that
\begin{equation}\label{eqn-xi}
\left\{
\begin{aligned}
&\dot\xi_r+\frac{n-4}{\xi_r}=b_r[\la,\xi,\phi,\Psi^*]\\
&\dot\xi_{z_j}=b_{z_j}[\la,\xi,\phi,\Psi^*]\\
\end{aligned}
\right.
\end{equation}
where
\begin{equation*}
b_r[\la,\xi,\phi,\Psi^*]=\int_{\mathcal D_{2R}} \mathcal H_r[\la,\xi,\phi,\Psi^*](y,t)Z_1(y) dy
\end{equation*}
\begin{equation*}
b_{z_j}[\la,\xi,\phi,\Psi^*]=\int_{\mathcal D_{2R}} \mathcal H_{z_j}[\la,\xi,\phi,\Psi^*](y,t)Z_j(y) dy~\mbox{ for }~ j=2,3,4
\end{equation*}
with
\begin{equation*}
\mathcal H_r[\la,\xi,\phi,\Psi^*](y,t)=\left[\int_{\mathbb S^3} \left(\mathcal H-\la U_{y_1}\left(\dot\xi_r+\frac{n-4}{\la y_1+\xi_r}\right)\right)\Theta_1(\theta)d\theta\right]\Theta_1
\end{equation*}
and
\begin{equation*}
\mathcal H_{z_j}[\la,\xi,\phi,\Psi^*]=\sum\limits_{j=2}^4\left[\int_{\mathbb S^3} \left(\mathcal H-\la U_{y_j}\dot\xi_{z_j}\right)\Theta_j(\theta)d\theta\right]\Theta_j.
\end{equation*}
Here $\Theta_j$ $(j=1,\cdots,4)$ are the eigenfunctions corresponding to the second eigenvalue of $-\Delta_{\mathbb S^3}.$ Now we want to evaluate the sizes of $b_r[\la,\xi,\phi,\Psi^*]$ and $b_{z_j}[\la,\xi,\phi,\Psi^*]$. By direct computations, we get
\begin{align}\notag
|b_r[\la,\xi,\phi,\Psi^*]|\lesssim&~\left(\la_*|\dot\la_*|(|\log(T-t)|+|\log T|^{-\frac2{n-2}})+\la_*\|Z^*\|_{\infty}\right)(1+O(R^{-3}))\\ \notag
&~+\la_*(t)\la_*^{\nu}(0)R^{2-\alpha}(0)|\log T|\|\psi\|_*(1+O(R^{-3}))\\ \notag
&~+\la_*^{\frac12+\nu_1}\|\xi\|_G\|\phi^1\|_{{\rm in},\nu_1,a_1}(1+O(R^{-a_1}))+\la_*^{1+\nu_1}|\dot\la_*|\|\phi^1\|_{{\rm in},\nu_1,a_1}(1+O(R^{1-a_1}))\\ \notag
&~+\la_*^{\frac32}|\dot\la_*|R(1+O(R^{-1}))+\la_*^{2\nu_1}\|\phi^1\|^2_{{\rm in},\nu_1,a_1}(1+O(R^{-2a_1-1}))\\ \notag
&~+\la_*^{2}(t)\la_*^{2\nu}(0)R^{4-2\alpha}(0)|\log T|^2\|\psi\|_*^2(1+O(R^{-1}))\\ \notag
&~+\la_*^2\|Z^*\|^2_{\infty}(1+O(R^{-1}))+\la_*^2(|\dot\la_*|^2|\log(T-t)|^2+|\log T|^{-\frac4{n-2}})(1+O(R^{-1}))\\ \notag
&~+\la_*^{2}|\log T|^{-2\epsilon_0} (1+O(R^{-1}))+\la_*|\dot\la_*|(1+O(R^{-3}))\\ \notag
&~+\la_*^2|\dot\la_*|(1+O(R^{-1}))+\la_*^3|\dot\la_*|R+\la_*^{\frac52}|\dot\la_*|R^2\|\xi\|_G\\ \label{est-br}
&~+\la_*^{\frac32}|\dot\la_*|R\|\xi\|_G(1+O(R^{-1}))+\la_*|\log T|^{-\epsilon_0}(1+O(R^{-3}))
\end{align}
and
\begin{align} \notag
|b_{z_j}[\la,\xi,\phi,\Psi^*]|\lesssim&~\left(\la_*(|\dot\la_*||\log(T-t)|+|\log T|^{-\frac2{n-2}})+\la_*\|Z^*\|_{\infty}\right)(1+O(R^{-3}))\\ \notag
&~+\la_*(t)\la_*^{\nu}(0)R^{2-\alpha}(0)|\log T|\|\psi\|_*(1+O(R^{-3}))\\ \notag
&~+\la_*^{1+\nu_1}|\dot\la_*|\|\phi^1\|_{{\rm in},\nu_1,a_1}(1+O(R^{1-a_1}))+\la_*^{1+\nu_1}|\dot\xi_{z_j}|\|\phi^1\|_{{\rm in},\nu_1,a_1}(1+O(R^{-a_1}))\\ \notag
&~+\la_*^{\frac32}|\dot\la_*|R(1+O(R^{-1}))+\la_*^{2\nu_1}\|\phi^1\|^2_{{\rm in},\nu_1,a_1}(1+O(R^{-2a_1-1}))\\ \notag
&~+\la_*^{2}(t)\la_*^{2\nu}(0)R^{4-2\alpha}(0)|\log T|^2\|\psi\|_*^2(1+O(R^{-1}))\\ \notag
&~+\la_*^2\|Z^*\|^2_{\infty}(1+O(R^{-1}))+\la_*^2(|\dot\la_*|^2|\log(T-t)|^2+|\log T|^{-\frac{4}{n-2}})(1+O(R^{-1}))\\ \notag
&~+\la_*^{2}|\log T|^{-2\epsilon_0}  (1+O(R^{-1}))+\la_*|\dot\la_*|(1+O(R^{-3}))\\ \notag
&~+\la_*^2|\dot\la_*|(1+O(R^{-1}))+\la_*^3|\dot\la_*|R+\la_*^{3+\upsilon}|\dot\la_*|R^2\|\xi\|_G\\ \label{est-bz}
&~+\la_*^{2+\upsilon}|\dot\la_*|R\|\xi\|_G(1+O(R^{-1}))+\la_*|\log T|^{-\epsilon_0}(1+O(R^{-3})).
\end{align}
Since $\xi_r(t)=\sqrt{2(n-4)(T-t)}+\xi_{r,1}(t),$
problem \eqref{eqn-xi} becomes
\begin{equation}\label{eqn-xi''}
\left\{
\begin{aligned}
&\dot\xi_{r,1}-\frac{(n-4)\xi_{r,1}}{\sqrt{2(n-4)(T-t)}\left(\sqrt{2(n-4)(T-t)}+\xi_{r,1}\right)}=b_r[\la,\xi,\phi,\Psi^*],\\
&\dot\xi_{z_j}=b_{z_j}[\la,\xi,\phi,\Psi^*].\\
\end{aligned}
\right.
\end{equation}
Then we analyze the reduced problem \eqref{eqn-xi''}, which defines operators $\Xi_r$ and $\Xi_{z_j}$ $(j=2,3,4)$ that return the solutions $\xi_{r,1}$ and $\xi_{z_j}$ respectively. Here we write
\begin{equation}\label{def-XiXi}
\Xi=(\Xi_r,\Xi_{z_2},\Xi_{z_3},\Xi_{z_4}).
\end{equation}
We shall solve $(\xi_{r,1},\xi_{z,1})$ under the norm
\begin{equation*}
\|\xi\|_G=\sup_{t\in(0,T)} \left[(T-t)^{-\frac12-\upsilon}|\xi_{r,1}(t)|+M_1(T-t)^{\frac12-\upsilon}|\dot\xi_{r,1}(t)|+|\xi_{z_j}(t)|+(T-t)^{-\upsilon}|\dot\xi_{z_j}(t)|\right]
\end{equation*}
for $\upsilon>0$ and $0<M_1<1$. From \eqref{eqn-xi''}, we have
\begin{equation*}
|\xi_{r,1}(t)|\leq \left(\frac{(T-t)^{\frac12+\upsilon}\|\xi\|_G}{2(T-t)}+\left\|b_r[\la,\xi,\phi,\Psi^*]\right\|_{L^{\infty}(0,T)}\right)(T-t)
\end{equation*}
and
\begin{equation*}
|\xi_{z_j}(t)|\leq |z_0|+\left\|b_{z_j}[\la,\xi,\phi,\Psi^*]\right\|_{L^{\infty}(0,T)}(T-t).
\end{equation*}
Therefore, we obtain
\begin{equation}\label{est-Xir''}
\|\Xi_r\|_G\leq \frac{1+M_1}{2}\|\xi\|_G+(1+M_1)(T-t)^{\frac12-\upsilon}\left\|b_r[\la,\xi,\phi,\Psi^*]\right\|_{L^{\infty}(0,T)}
\end{equation}
and
\begin{equation}\label{est-Xiz''}
\|\Xi_{z_j}\|_G\leq |z_0|+(T-t)^{-\upsilon}\left\|b_{z_j}[\la,\xi,\phi,\Psi^*]\right\|_{L^{\infty}(0,T)}.
\end{equation}
By \eqref{est-br}, \eqref{est-bz}, \eqref{est-Xir''} and \eqref{est-Xiz''}, we conclude that for some constant $C>0$
\begin{align}\notag
\|\Xi_r\|_G\leq&~\left[\frac{1+M_1}{2}+C(1+M_1)(T-t)^{\frac12-\upsilon}\left(\la_*^{\frac12+\nu_1}+\la_*^{\frac32}|\dot\la_*|R\right)\right]\|\xi\|_G\\ \notag
&~+C(1+M_1)(T-t)^{\frac12-\upsilon}\bigg[\la_*(t)\la_*^{\nu}(0)R^{2-\alpha}\|\psi\|_*+\la_*\|Z^*\|_{\infty}\\ \label{est-Xir}
&~+(\la_*^{\frac12+\nu_1}+\la_*^{2\nu_1})\|\phi^1\|_{{\rm in},\nu_1,a_1}+\la_*\|\la\|_{\infty}+\la_*\bigg]
\end{align}
and
\begin{equation}\label{est-Xiz}
\begin{aligned}
\|\Xi_{z_j}\|_G\leq&~ |z_0|+C(T-t)^{-\upsilon}\Bigg[\la_*(t)\la_*^{\nu}(0)R^{2-\alpha}\|\psi\|_*+\la_*\|Z^*\|_{\infty}\\
&~+(\la_*^{1+\nu_1}+\la_*^{2\nu_1})\|\phi^1\|_{{\rm in},\nu_1,a_1}+\la_*\|\la\|_{\infty}+\la_*+\left(\la_*^{1+\nu_1+\upsilon}+\la_*^{2+\upsilon}|\dot\la_*|R\right)\|\xi\|_G\Bigg].\\
\end{aligned}
\end{equation}

\medskip

\subsubsection{The reduced problem of \texorpdfstring{$\lambda(t)$}{lambda(t)}}\label{subsec-la}

\medskip

Since the reduced problem of $\la$ is essentially the same as that of \cite{17HMF}, we shall follow the strategy and logic in \cite[Section 8]{17HMF}.

From direct computations, we see that \eqref{def-cmode0}
gives a non-local integro-differential equation
\begin{equation}\label{indi-la}
\begin{aligned}
\int_{B_{4R}} 3 U^2(y_1,y') \Psi_0\big(\la y_1+c_n\sqrt{T-t},\la y',t\big)  Z_5(y_1,y')dy_1 dy'+{\bf c_0} \dot\la=a[\la,\xi,\Psi^*](t)+\mathtt a_r[\la,\xi,\phi,\Psi^*](t),
\end{aligned}
\end{equation}
where
\begin{equation}\label{def-aaa}
{\bf c_0}=2\alpha_0\int_{\R^4}\frac{Z_5(y)}{(1+|y|^2)^2}dy,\quad a[\la,\xi,\Psi^*]=-\int_{\mathcal D_{2R_*}} 3U^2(y)\left(\Psi_0+\Psi^*\right)Z_5(y) dy,
\end{equation}
and the remainder term $\mathtt a_r[\la,\xi,\phi,\Psi^*](t)$ turns out to be smaller order and has the following bound
\begin{align*}
\left|\mathtt a_r[\la,\xi,\phi,\Psi^*](t)\right|\lesssim&~\left[\la_*^{\nu}R^{\sigma(4-a)}(1+|\dot\la_*|)+\la_*^{\nu-\frac12}R^{\sigma(4-a)}\|\xi\|_G\right]\log R\|\phi^0\|_{0,\sigma,\nu,a}(1+O(R_*^{-2}))\\
&~+\la_*^{\frac12} R |\dot\la_*||\log(T-t)|+\la_*^{2\nu-1}R^{2\sigma(4-a)}(\log R)^2\|\phi^0\|_{0,\sigma,\nu,a}^2(1+O(R_*^{-8}))\\
&~+\la_*(t)\la_*^{2\nu}(0)R^{4-2\alpha}(0)|\log T|^2|\log(T-t)|\|\psi\|_*^2+\la_*|\dot\la_*|^2|\log(T-t)|^3\\
&~+\la_*|\log(T-t)|\|Z^*\|_{\infty}^2+\la_*|\log T|^{-2\epsilon_0}|\log(T-t)|+|\log T|^{-\epsilon_0},
\end{align*}
where we have used \eqref{est-Psi_1}. Here for technical reasons $\la(t)$ is assumed to be defined for negative $t$.
We first introduce the following norms
\begin{equation*}
\|f\|_{\Theta,l}:=\sup_{t\in[0,T]}\frac{|\log(T-t)|^l}{(T-t)^{\Theta}}|f(t)|,
\end{equation*}
where $f\in C([-T,T];\R)$ with $f(T)=0$, and $\Theta\in(0,1)$, $l\in\R$.
\begin{equation*}
[g]_{\gamma,m,l}:=\sup_{I_T} \frac{|\log(T-t)|^l}{(T-t)^m (t-s)^{\gamma}}|g(t)-g(s)|,
\end{equation*}
where $I_T=\left\{0\leq s\leq t\leq T: t-s\leq\frac{1}{10}(T-t)\right\},$ $g\in C([-T,T];\R)$ with $g(T)=0$ and $0<\gamma<1$, $m>0$, $l\in\R$. Also, we define
\begin{equation}\label{def-mB0}
\mathcal B_0[\la](t):=\int_{B_{4R}} 3 U^2(y_1,y') \Psi_0\big(\la y_1+c_n\sqrt{T-t},\la y',t\big)  Z_5(y_1,y')dy_1 dy'+{\bf c_0} \dot\la
\end{equation}
and write
\begin{equation}\label{def-c^0}
c^0[\mathcal H]=\frac{\mathcal B_0[\la]-(a[\la,\xi,\Psi^*]+\mathtt a_r[\la,\xi,\phi,\Psi^*])}{\int_{\mathcal D_{2R_*}}|Z_5(y)|^2dy}.
\end{equation}

A key proposition concerning the solvability of $\la$ is stated as follows.

\begin{prop}[\cite{17HMF}]\label{keyprop-la}
Let $\omega,\Theta\in(0,\frac12)$, $\gamma\in(0,1)$, $m\leq \Theta-\gamma$ and $l\in\R$. If $a(t)$ satisfies $a(T)<0$ with $1/C\leq a(T)\leq C$ for some constant $C>1$, and
\begin{equation}\label{assump-a}
T^{\Theta}|\log T|^{1+c-l}\|a(\cdot)-a(T)\|_{\Theta,l-1}+[a]_{\gamma,m,l-1}\leq C_1
\end{equation}
for some $c>0$, then there exist two operators $\mathcal P$ and $\mathcal R_0$ such that $\la=\mathcal P[a]:[-T,T]\rightarrow \R$ satisfies
\begin{equation}\label{eqn-mB0}
\mathcal B_0[\la](t)=a(t)+\mathcal R_0[a](t)
\end{equation}
with
\begin{equation*}
\begin{aligned}
|\mathcal R_0[a](t)|\lesssim \left(T^{\frac12+c}+T^{\Theta}\frac{\log|\log T|}{|\log T|}\|a(\cdot)-a(T)\|_{\Theta,l-1}+[a]_{\gamma,m,l-1}\right)\frac{(T-t)^{m+(1+\omega)\gamma}}{|\log(T-t)|^l}.
\end{aligned}
\end{equation*}
\end{prop}
The proof of Proposition \ref{keyprop-la} is in \cite{17HMF}.  The idea of the proof is to observe that
$$\mathcal B_0[\la]\approx C\left(\int_{-T}^{t-(T-t)} \frac{\dot\la(s)}{t-s}\,ds
    +c_n^*\int_{t-(T-t)}^{t-\la^2(t)} \frac{\dot\la(s)}{t-s}\,ds\right),$$
where $C>0$,
and we decompose
$$\mathcal B_0[\la]=\mathcal B_0^*[\la]+\mathcal S_{\omega}[\dot\la]+\mathcal R_{\omega}[\dot\la],$$
where
\begin{equation*}
\mathcal B_0^*[\la]:=\mathcal B_0[\la]-C\left(\int_{-T}^{t-(T-t)} \frac{\dot\la(s)}{t-s}\,ds
    +c_n^*\int_{t-(T-t)}^{t-\la^2(t)} \frac{\dot\la(s)}{t-s}\,ds\right),
\end{equation*}
\begin{equation*}
\mathcal S_{\omega}[\dot\la]:=\dot\la\left[(1+\omega)\log(T-t)-2\log\la_*(t)\right]+\int_{-T}^{t-(T-t)^{1+\omega}}\frac{\dot\la(s)}{t-s}ds
\end{equation*}
and
\begin{equation*}
\mathcal R_{\omega}[\dot\la]:=-\int_{t-(T-t)^{1+\omega}}^{t-\la_*^2(t)}\frac{\dot\la(t)-\dot\la(s)}{t-s} ds.
\end{equation*}
Here $\omega>0$ is a fixed number. We solve a modified equation where we drop $\mathcal R_{\omega}[\dot\la]$ in \eqref{eqn-mB0}, and thus the remainder $\mathcal R_0$ is essentially $\mathcal R_{\omega}[\dot\la]$ and $\mathtt a_r[\la,\xi,\phi,\Psi^*]$.

In another aspect, we modify problem \eqref{indi-la} replacing $a[\la,\xi,\Psi^*]$ by its main term. To this end, we define
$$a[\la,\xi,\Psi^*]=a^0[\la,\xi,\Psi^*]+a^1[\la,\xi,\Psi^*]+a^{\perp}[\la,\xi,\Psi^*]$$
with
$$a^0[\la,\xi,\Psi^*]=-\int_{\mathcal D_{2R_*}} \tilde L^0[\Psi] Z_5(y) dy,\quad a^1[\la,\xi,\Psi^*]=-\int_{\mathcal D_{2R_*}} \tilde L^1[\Psi] Z_5(y) dy,$$
and
$$a^{\perp}[\la,\xi,\Psi^*]=-\int_{\mathcal D_{2R_*}} \tilde L^{\perp}[\Psi] Z_5(y) dy,$$
where $\tilde L[\Psi]:=3U^2(\Psi_0+\Psi^*)$, $\tilde L^0[\Psi]$ is the projection of $\tilde L[\Psi]$ on mode $0$, $\tilde L^1[\Psi]$ is the projection of $\tilde L[\Psi]$ on modes $1$ to $4$, and $\tilde L^{\perp}[\Psi]$ is the projection of $\tilde L[\Psi]$ on higher modes $j\geq 5$.

We define
\begin{equation}\label{def-c_*^0}
\begin{aligned}
c_*^0[\la,\xi,\Psi^*](t):=&~\frac{\mathcal R_0\left[a^0[\la,\xi,\Psi^*]\right](t)+a^1[\la,\xi,\Psi^*](t)+a^{\perp}[\la,\xi,\Psi^*](t)}{\int_{\mathcal D_{2R_*}}|Z_5(y)|^2dy}\\
&~-\left(c^0[\mathcal H[\la,\xi,\Psi^*]]-\tilde c^0[\mathcal H^0[\la,\xi,\Psi^*]]\right),
\end{aligned}
\end{equation}
where $\mathcal R_0$ is the operator given in Proposition \ref{keyprop-la}, $c^0$ is defined in \eqref{def-c^0}, and $\tilde c^0$ is the operator given in Proposition \ref{propmode0}. The reason for choosing such $c_*^0$ is the following. By Proposition \ref{keyprop-la}, the equation we solve is
$$\mathcal B_0[\la](t)=a^0[\la,\xi,\Psi^*](t)+\mathcal R_0[a^0[\la,\xi,\Psi^*]]$$
which is equivalent to
$$c^0[\mathcal H]=\frac{\mathcal R_0\left[a^0[\la,\xi,\Psi^*]\right](t)+a^1[\la,\xi,\Psi^*](t)+a^{\perp}[\la,\xi,\Psi^*](t)}{\int_{\mathcal D_{2R_*}}|Z_5(y)|^2dy}.$$
We shall consider the following reduced equation
$$\tilde c^0[\mathcal H^0]=c_*^0[\la,\xi,\Psi^*],$$
from which we get \eqref{def-c_*^0}.

By \eqref{def-aaa}, Proposition \ref{keyprop-la} and Proposition \ref{outer-apriori}, it is natural to choose
\begin{equation*}
\Theta=\nu-1+\frac{\alpha}{2}
\end{equation*}
and
\begin{equation*}
m=\nu-2-\gamma+\frac12(2+\alpha).
\end{equation*}
In order for $\|a(\cdot)-a(T)\|_{\Theta,l-1}$ and $[a]_{\gamma,m,l-1}$ to be finite, we require
$$l<\max\{1+2\Theta,1+2m\}$$
and it then follows that
$$\|a(\cdot)-a(T)\|_{\Theta,l-1}\lesssim |\log T|^{l-\Theta-1},\quad [a]_{\gamma,m,l-1}\lesssim |\log T|^{l-m-1}.$$
Another assumption $m\leq\Theta-\gamma$ in Proposition \ref{keyprop-la} is also valid.
Finally, in order to make the remainder $\mathcal R_0[a]$ small, we impose
\begin{equation*}
m+(1+\omega)\gamma>\Theta,
\end{equation*}
which implies that
\begin{equation*}
\omega\gamma>0.
\end{equation*}

\medskip


\subsection{Inner--outer gluing system}

By the discussions in Section \ref{subsec-la}, we transform the inner--outer problems \eqref{inner}, \eqref{outer} into the problems of finding solutions $(\psi,\phi^0,\phi^1,\phi^{\perp},\la,\xi)$ solving the following {\em inner--outer gluing system}
\begin{equation}\label{eqn-outer}
\begin{cases}
\psi_t=\Delta_{(r,z)} \psi+\frac{n-4}{r}\partial_r \psi +\mathcal G(\phi^0+\phi^1+\phi^{\perp},\psi+Z^*,\la,\xi),~&\mbox{ in }\mathcal D\times (0,T)\\
\psi=-\Psi^0,~&\mbox{ on }(\partial \mathcal D\backslash \{r=0\})\times (0,T)\\
\psi_r=0,~&\mbox{ on } (\mathcal D\cap\{r=0\})\times (0,T)\\
\psi(r,z,0)=0,~&\mbox{ in }\mathcal D\\
\end{cases}
\end{equation}
\begin{equation}\label{eqn-projm0}
\begin{cases}
\la^2 \phi^0_t = \Delta_y \phi^0 +3U^2(y)\phi^0+\mathcal H^0(\phi,\psi,\la,\xi)+\tilde c^0[\mathcal H^0] Z_5~~&\mbox{ in }\mathcal D_{2R}\times (0,T)\\
\phi^0(\cdot,0)=0~~&\mbox{ in }\mathcal D_{2R}\\
\end{cases}
\end{equation}

\begin{equation}\label{eqn-projm1}
\begin{cases}
\la^2 \phi^1_t = \Delta_y \phi^1 +3U^2(y)\phi^1+\mathcal H^1(\phi,\psi,\la,\xi)+\sum\limits_{\ell=1}^4 c^{\ell}[\mathcal H^1]Z_{\ell}~~&\mbox{ in }\mathcal D_{2R}\times (0,T)\\
\phi^1(\cdot,0)=0~~&\mbox{ in }\mathcal D_{2R}\\
\end{cases}
\end{equation}

\begin{equation}\label{eqn-projmh}
\begin{cases}
\la^2 \phi^{\perp}_t = \Delta_y \phi^{\perp} +3U^2(y)\phi^{\perp}+\mathcal H^{\perp}(\phi,\psi,\la,\xi)+c^0_*[\la,\xi,\Psi^*]Z_5~~&\mbox{ in }\mathcal D_{2R}\times (0,T)\\
\phi^{\perp}(\cdot,0)=0~~&\mbox{ in }\mathcal D_{2R}\\
\end{cases}
\end{equation} 
$$
\begin{aligned}
c^0[\mathcal H](t)-\tilde c^0[\la,\xi,\Psi^*](t)=0~\mbox{ for all }~t\in(0,T)\label{redu-la}\\
c^1[\mathcal H](t)=0~\mbox{ for all }~t\in(0,T)\label{redu-xi}
\nonumber \end{aligned} $$
where $\mathcal G$ is defined in \eqref{def-mG}, $\mathcal H^0$, $\mathcal H^1$, $\mathcal H^{\perp}$ are the projections on different modes defined in \eqref{def-mHm0}--\eqref{def-mHmh}.

It is direct to see that if $(\psi,\phi^0,\phi^1,\phi^{\perp},\la,\xi)$ satisfies the system \eqref{eqn-outer}--\eqref{redu-xi}, then
$$\Psi^*=\psi+Z^*,\qquad\phi=\phi^0+\phi^1+\phi^{\perp}$$
solve the inner--outer problems \eqref{inner}, \eqref{outer}, and thus the desired blow-up solution is found.

\subsection{The fixed point formulation}

The inner--outer gluing system \eqref{eqn-outer}--\eqref{redu-xi} can be formulated as a fixed point problem for operators we shall describe below.

We first define the following function spaces
\begin{align}\notag
&X_{\phi^0}:=\left\{\phi^0\in L^{\infty}(\mathcal D_{2R}\times (0,T)):~\nabla_y \phi^0\in L^{\infty}(\mathcal D_{2R}\times (0,T)),~\|\phi^0\|_{0,\sigma,\nu,a}<+\infty\right\},\\ \notag
&X_{\phi^1}:=\left\{\phi^1\in L^{\infty}(\mathcal D_{2R}\times (0,T)):~\nabla_y \phi^1\in L^{\infty}(\mathcal D_{2R}\times (0,T)),~\|\phi^1\|_{{\rm in},\nu_1,a_1}<+\infty\right\},\\ \notag
&X_{\phi^{\perp}}:=\left\{\phi^{\perp}\in L^{\infty}(\mathcal D_{2R}\times (0,T)):~\nabla_y \phi^{\perp}\in L^{\infty}(\mathcal D_{2R}\times (0,T)),~\|\phi^{\perp}\|_{{\rm in},\nu,a}<+\infty\right\},\\ \notag
&X_{\psi}:=\big\{\psi\in L^{\infty}(\mathcal D\times (0,T)):~\|\psi\|_{*}<+\infty,~\mbox{ $\psi$ is Lipschitz continuous with respect}\\ \label{def-fcnspaces}
&\qquad\qquad\mbox{ to $(r,z)$ in $\mathcal D\times (0,T)$}\big\}.
\end{align}

In order to introduce the space for the parameter function $\la(t)$, we recall from \eqref{def-mB0} that the integral operator $\mathcal B_0$ takes the following approximate form
$$\mathcal B_0[\la]=\int_{-T}^{t-\la_*^2(t)}\frac{\dot\la(s)}{t-s} ds+O(\|\dot\la\|_{\infty}).$$
Proposition \ref{keyprop-la} provides an approximate inverse operator $\mathcal P$ of the integral operator $\mathcal B_0$ such that for $a(t)$ satisfying \eqref{assump-a}, $\la:=\mathcal P[a]$ satisfies
$$\mathcal B_0[\la]=a+\mathcal R_0[a]~\mbox{ in }~[-T,T],$$
where $\mathcal R_0[a]$ is a small remainder. Also, the proof  in \cite{17HMF} gives the following decomposition
\begin{equation}\label{def-mP1}
\mathcal P[a]=\la_{0,\kappa}+\mathcal P_1[a]
\end{equation}
with
\begin{equation*}
\la_{0,\kappa}:=\kappa|\log T|\int_t^T \frac{1}{|\log(T-s)|^2}ds,~t\leq T
\end{equation*}
$\kappa=\kappa[a]\in\R$, and the function $\la_1=\mathcal P_1[a]$ satisfies
\begin{equation}\label{est-la1la1}
\|\la_1\|_{*,3-\iota}\lesssim |\log T|^{1-\iota}\log^2(|\log T|)
\end{equation}
for $0<\iota<1,$ where the $\|\cdot\|_{*,3-\iota}$-norm is defined as follows
\begin{equation*}
\|f\|_{*,k}:=\sup_{t\in[-T,T]}|\log(T-t)|^k |\dot f(t)|.
\end{equation*}
So we define
$$X_{\la}:=\{\la_1\in C^1([-T,T]):\la_1(T)=0,\|\la_1\|_{*,3-\iota}<\infty\}.$$
Here by $(\kappa,\la_1)$, we represent $\la$ in the form
$$\la=\la_{0,\kappa}+\la_1,$$
and from \cite{17HMF}, one can write the norm
\begin{equation}\label{def-normla}
\|\la\|_F=|\kappa|+\|\la_1\|_{*,3-\iota}.
\end{equation}

Recall that $\xi(t)=(\xi_r(t),\xi_z(t))$ with $\xi_r(t)=\sqrt{2(n-4)(T-t)}+\xi_{r,1}(t)$, $\xi_z(t)=z_0+\xi_{z,1}(t)$ and write $\xi(t)=\xi_*(t)+\xi_1(t)$. We define the following space for $(\xi_{r,1},\xi_{z,1})$
\begin{equation*}
X_{\xi}=\left\{\xi\in C^1((0,T);\R^4),~\dot\xi(T)=0,~\|\xi\|_{G}<+\infty\right\}
\end{equation*}
with
\begin{equation}\label{def-normxi}
\|\xi\|_G:=\sup_{t\in(0,T)} \left[(T-t)^{-\frac12-\upsilon}|\xi_{r,1}(t)|+M_1(T-t)^{\frac12-\upsilon}|\dot\xi_{r,1}(t)|+|\xi_{z_j}(t)|+(T-t)^{-\upsilon}|\dot\xi_{z_j}(t)|\right]
\end{equation}
for some $0<\upsilon<1$.

Define
\begin{equation}\label{def-mX}
\mathcal X=X_{\phi^0}\times X_{\phi^1}\times X_{\phi^{\perp}}\times X_{\psi}\times\R\times X_{\la}\times X_{\xi}.
\end{equation}
We shall solve the inner--outer gluing system in a closed ball $\mathcal B$ in which $(\phi^0,\phi^1,\phi^{\perp},\psi,\kappa,\la_1,\xi_1)\in\mathcal X$ satisfies
\begin{equation}\label{def-closedball}
\left\{
\begin{aligned}
&\|\phi^0\|_{0,\sigma,\nu,a}+\|\phi^1\|_{{\rm in},\nu_1,a_1}+\|\phi^{\perp}\|_{{\rm in},\nu,a}\leq 1,\qquad \|\psi\|_*\leq 1,\\
&|\kappa-\kappa_0|\leq |\log T|^{-1/2},\quad \|\la_1\|_{*,3-\iota}\leq C|\log T|^{1-\iota}\log^2(|\log T|),\quad\|\xi\|_G\leq 1
\end{aligned}
\right.
\end{equation}
for some large and fixed constant $C$, where $\kappa_0=Z^*_0(0)$.

The inner--outer gluing system \eqref{eqn-outer}--\eqref{redu-xi} can be formulated as a fixed point problem, where we define an operator $\mathcal F$ which returns the solution from $\mathcal B$ to $\mathcal X$
\begin{equation*}
\begin{aligned}
\mathcal F: \mathcal B\subset \mathcal X~\rightarrow& ~\mathcal X\\
v~\mapsto&~ \mathcal F(v)=(\mathcal F_{\phi^0}(v),\mathcal F_{\phi^1}(v),\mathcal F_{\phi^{\perp}}(v),\mathcal F_{\psi}(v),\mathcal F_{\kappa}(v),\mathcal F_{\la_1}(v),\mathcal F_{\xi}(v))
\end{aligned}
\end{equation*}
with
\begin{equation}\label{def-operators}
\begin{aligned}
\mathcal F_{\phi^0}(\phi^0,\phi^1,\phi^{\perp},\psi,\kappa,\la_1,\xi_1)=&~\mathcal T_0(\mathcal H^0[\la,\xi,\Psi^*]),\\
\mathcal F_{\phi^1}(\phi^0,\phi^1,\phi^{\perp},\psi,\kappa,\la_1,\xi_1)=&~\mathcal T_1(\mathcal H^1[\la,\xi,\Psi^*]),\\
\mathcal F_{\phi^{\perp}}(\phi^0,\phi^1,\phi^{\perp},\psi,\kappa,\la_1,\xi_1)=&~\mathcal T_{\perp}\left(\mathcal H^{\perp}[\la,\xi,\Psi^*]+c^0_*[\la,\xi,\Psi^*]Z_5\right),\\
\mathcal F_{\psi}(\phi^0,\phi^1,\phi^{\perp},\psi,\kappa,\la_1,\xi_1)=&~\mathcal T_{\psi}\left(\mathcal G(\phi^0+\phi^1+\phi^{\perp},\Psi^*,\la,\xi)\right),\\
\mathcal F_{\kappa}(\phi^0,\phi^1,\phi^{\perp},\psi,\kappa,\la_1,\xi_1)=&~\kappa\left[a^0[\la,\xi,\Psi^*]\right],\\
\mathcal F_{\la_1}(\phi^0,\phi^1,\phi^{\perp},\psi,\kappa,\la_1,\xi_1)=&~\mathcal P_1\left[a^0[\la,\xi,\Psi^*]\right],\\
\mathcal F_{\xi}(\phi^0,\phi^1,\phi^{\perp},\psi,\kappa,\la_1,\xi_1)=&~\Xi(\phi^0,\phi^1,\phi^{\perp},\psi,\la,\xi).\\
\end{aligned}
\end{equation}
Here $\mathcal T_0$, $\mathcal T_1$ and $\mathcal T_{\perp}$ are the operators given from Proposition \ref{lineartheory} which solve different modes of the inner problems \eqref{eqn-projm0}-\eqref{eqn-projm1}-\eqref{eqn-projmh}. The operator $\mathcal T_{\psi}$ defined by Proposition \ref{outer-apriori} deals with the outer problem \eqref{eqn-outer}. Operators $\kappa[a]$, $\mathcal P_1$ and $\Xi$ handle the equations for $\la$ and $\xi$ which are defined in Proposition \ref{keyprop-la}, \eqref{def-mP1} and \eqref{def-XiXi}.

\medskip


\medskip

\subsection{Choice of constants}\label{subsec-choices}

We now collect the restrictions on the parameters
\[
\theta,\quad \alpha,\quad a,\quad a_1,\quad
\nu,\quad \nu_1,\quad \nu_2,\quad \sigma
\]
needed to close the fixed point argument. We choose
\[
\theta\in\left(\frac{n-1}{n-2},\frac n{n-2}\right),
\qquad
0<\alpha<a<1,
\qquad
0<\sigma\ll1.
\]
The outer estimates require
\[
\nu+\frac{\alpha}{2}-\nu_2>0,
\qquad
\nu_2<1,
\qquad
2\nu-\nu_2+\alpha-1>0,
\qquad
\nu-1+\frac{\alpha}{2}>0,
\]
and, as explained in Remark \ref{remark6.1},
\[
\nu_2>\nu-1+\frac{\alpha}{2}.
\]
Thus we first take
\[
1-\frac{\alpha}{2}<\nu<1,
\]
and then choose
\[
\nu-1+\frac{\alpha}{2}
<
\nu_2
<
\min\left\{
1,\,
\nu+\frac{\alpha}{2},\,
2\nu+\alpha-1
\right\}.
\]
This interval is nonempty precisely because
\(\nu>1-\alpha/2\).

The inner estimates require, in addition,
\[
0<\nu_1<1,
\qquad
\nu-\nu_1+\frac{\alpha}{2}>0,
\qquad
\frac32-\nu_1-\frac12(1+a_1)>0,
\qquad
2\nu-\nu_1+\alpha-\frac{a_1}{2}>0,
\]
together with
\[
\frac12-\frac12\sigma(4-a)>0,
\qquad
\nu-\sigma(4-a)>0.
\]
We choose \(a_1>1\) sufficiently close to \(1\), and then take
\[
0<\nu_1<
\min\left\{
1,\,
\nu+\frac{\alpha}{2},\,
1-\frac{a_1}{2},\,
2\nu+\alpha-\frac{a_1}{2}
\right\}.
\]
The interval is nonempty for such a choice of \(a_1\). Finally,
\(\sigma>0\) is chosen sufficiently small so that the last two
inequalities hold.

With these choices, all the restrictions from the outer estimates, the
inner estimates, and the reduced equation for \(\lambda(t)\) are
satisfied. 

\medskip


\subsection{Proof of Theorem \texorpdfstring{\ref{thm}}{1.1}}

\medskip

Consider the operator
\begin{equation}\label{def-mF}
\mathcal F=(\mathcal F_{\phi^0},\mathcal F_{\phi^1},\mathcal F_{\phi^{\perp}},\mathcal F_{\psi},\mathcal F_{\kappa},\mathcal F_{\la_1},\mathcal F_{\xi})
\end{equation}
given in \eqref{def-operators}. To prove Theorem \ref{thm}, our strategy is to show that the operator $\mathcal F$ has a fixed point in $\mathcal B$ by the Schauder fixed point theorem. Here the closed ball $\mathcal B$ is defined in \eqref{def-closedball}. By collecting the estimates \eqref{outer-contraction}, \eqref{est-mHm0}, \eqref{est-mHm1}, \eqref{est-mHmp}, \eqref{est-Xir}, \eqref{est-Xiz}, \eqref{est-la1la1}, and using Proposition \ref{outer-apriori}, Proposition \ref{lineartheory}, Proposition \ref{keyprop-la}, we conclude that for $(\phi^0,\phi^1,\phi^{\perp},\psi,\kappa,\la_1,\xi_1)\in\mathcal B$
\begin{equation}\label{contraction-mF}
\left\{
\begin{aligned}
&\|\mathcal F_{\phi^0}(\phi^0,\phi^1,\phi^{\perp},\psi,\kappa,\la_1,\xi_1)\|_{0,\sigma,\nu,a}\leq CT^{\epsilon}\\
&\|\mathcal F_{\phi^1}(\phi^0,\phi^1,\phi^{\perp},\psi,\kappa,\la_1,\xi_1)\|_{{\rm in},\nu_1,a_1}\leq CT^{\epsilon}\\
&\|\mathcal F_{\phi^{\perp}}(\phi^0,\phi^1,\phi^{\perp},\psi,\kappa,\la_1,\xi_1)\|_{{\rm in},\nu,a}\leq CT^{\epsilon}\\
&\|\mathcal F_{\psi}(\phi^0,\phi^1,\phi^{\perp},\psi,\kappa,\la_1,\xi_1)\|_*\leq CT^{\epsilon}\\
&\left|\mathcal F_{\kappa}(\phi^0,\phi^1,\phi^{\perp},\psi,\kappa,\la_1,\xi_1)-\kappa_0\right|\leq C|\log T|^{-1}\\
&\|\mathcal F_{\la_1}(\phi^0,\phi^1,\phi^{\perp},\psi,\kappa,\la_1,\xi_1)\|_{*,3-\iota}\leq C|\log T|^{1-\iota}\log^2(|\log T|)\\
&\|\mathcal F_{\xi}(\phi^0,\phi^1,\phi^{\perp},\psi,\kappa,\la_1,\xi_1)\|_G\leq CT^{\epsilon}\\
\end{aligned}
\right.
\end{equation}
where $C>0$ is a constant independent of $T$, and $\epsilon>0$ is a small fixed number. On the other hand, compactness of the operator $\mathcal F$ defined in \eqref{def-mF} can be proved by suitable variants of \eqref{contraction-mF}. Indeed, if we vary the parameters $\theta,\alpha,a,a_1,\nu,\nu_1,\nu_2,\sigma$ slightly such that all the restrictions in Section \ref{subsec-choices} are satisfied, then we get \eqref{contraction-mF} with the norms in the left hand side defined by the new parameters, while the closed ball $\mathcal B$ remains the same. To be more specific, for fixed $\sigma',\nu',a'$ which are close to $\sigma,\nu,a$, one can show that if $(\phi^0,\phi^1,\phi^{\perp},\psi,\kappa,\la_1,\xi_1)\in\mathcal B$, then
$$\|\mathcal F_{\phi^0}(\phi^0,\phi^1,\phi^{\perp},\psi,\kappa,\la_1,\xi_1)\|_{0,\sigma',\nu',a'}\leq CT^{\epsilon'}.$$
Moreover, one can show that for $\nu'>\nu$ and $\nu'-\frac{\sigma'}{2}(4-a')>\nu-\frac{\sigma}{2}(4-a)$, one has a compact embedding in the sense that if a sequence $\{\phi^0_n\}$ is bounded in the $\|\cdot\|_{0,\sigma',\nu',a'}$-norm, then there exists a subsequence which converges in the $\|\cdot\|_{0,\sigma,\nu,a}$-norm. Thus, the compactness follows directly from a standard diagonal argument by Arzel\`a--Ascoli's theorem. Arguing in a similar manner, one can prove the compactness of the rest operators. Therefore, the existence of the desired solution follows from the Schauder fixed point theorem. The proof is complete.\qed

\medskip


\medskip

\appendix
\section{Proofs of technical Lemmas}\label{Appendix-outer}

\medskip

\begin{proof}[Proof of Lemma \ref{lemma-rhs1}]
The proof is achieved by considering the following Cauchy problem in $\R^n$
\begin{equation}\label{outer-linear''''}
\begin{cases}
\partial_t\psi_0=\Delta \psi_0+f,~&\mbox{ in }\R^n\times(0,T)\\
\psi_0(x,0)=0,~&\mbox{ in }\R^n\\
\end{cases}
\end{equation}
If we decompose the solution to \eqref{outer-linear''} in the form
$$\psi=\psi_0+\psi_1,$$
then $\psi_1$ solves the homogeneous heat equation in $\Omega\times(0,T)$ with boundary condition $-\psi_1$. By standard parabolic estimates, it suffices to establish the estimates \eqref{outer-rhs1}--\eqref{outerholder-rhs1} for $\psi_0$. In the sequel, we denote $\psi$ by the solution to \eqref{outer-linear''''} given by Duhamel's formula
\begin{align*}
\psi(x,t)=&~\int_0^t \int_{\R^n} \frac{e^{-\frac{|x-w|^2}{4(t-s)}}}{(4\pi(t-s))^{n/2}} f(w,s) dw ds\\
\lesssim&~ \int_0^t \la_*^{\nu-3}(s) R^{-2-\alpha}(s) \int_{\left|(r,z)-\xi(s)\right|\leq 2\la_*(s)R(s)} \frac{e^{-\frac{|x-w|^2}{4(t-s)}}}{(4\pi(t-s))^{n/2}} dw ds,
\end{align*}
where $w=(w_1,\cdots,w_n)$, $r=\sqrt{w_1^2+\cdots+w_{n-3}^2}$ and  $z=(w_{n-2},w_{n-1},w_n)$. We decompose
\begin{align*}
&\quad \int_0^t \la_*^{\nu-3}(s) R^{-2-\alpha}(s) \int_{\left|(r,z)-\xi(s)\right|\leq 2\la_*(s)R(s)} \frac{e^{-\frac{|x-w|^2}{4(t-s)}}}{(4\pi(t-s))^{n/2}} dw ds\\
&=\left(\int_0^{t-(T-t)}+\int_{t-(T-t)}^{t-\la^{\delta_1}_*(t)}+\int_{t-\la^{\delta_1}_*(t)}^t\right) \la_*^{\nu-3}(s) R^{-2-\alpha}(s) \int_{\left|(r,z)-\xi(s)\right|\leq 2\la_*(s)R(s)} \frac{e^{-\frac{|x-w|^2}{4(t-s)}}}{(4\pi(t-s))^{n/2}} dw ds\\
&:=I_{11}+I_{12}+I_{13}
\end{align*}
for some $\delta_1\geq 1$ to be found. Here we recall that
$\xi(t)\sim(\sqrt{2(n-4)(T-t)},z_0),$
with $z_0=(0,0,0)$. Directly integrating, we obtain
\begin{align*}
I_{11}=&~ \int_0^{t-(T-t)}\la_*^{\nu-3}(s) R^{-2-\alpha}(s) \int_{\left|(r,z)-\xi(s)\right|\leq 2\la_*(s)R(s)} \frac{e^{-\frac{|x-w|^2}{4(t-s)}}}{(4\pi(t-s))^{n/2}} dw ds\\
\lesssim&~\int_0^{t-(T-t)}\la_*^{\nu-3}(s) R^{-2-\alpha}(s) \int_{\left|(\tilde r,\tilde z)-\frac{\xi(s)}{\sqrt{t-s}}\right|\leq \frac{2\la_*(s)R(s)}{\sqrt{t-s}}} e^{-\frac{|\tilde x-\tilde w|^2}{4}} d\tilde w ds\\
\lesssim&~ \int_0^{t-(T-t)} \la_*^{\nu-3}(s) R^{-2-\alpha}(s)\int_{\left|(\tilde r,\tilde z)-\frac{\xi(s)}{\sqrt{t-s}}\right|\leq \frac{2\la_*(s)R(s)}{\sqrt{t-s}}} (\tilde r)^{n-4} d\tilde r d\tilde z ds\\
\lesssim&~ \int_0^{t-(T-t)} \frac{\la_*^{\nu-3}(s) R^{-2-\alpha}(s)}{(T-s)^{n/2}}\left(\la_*(s)R(s)\right)^4 (T-s)^{\frac{n-4}{2}} ~ds
\lesssim \la_*^{\nu}(0)R^{2-\alpha}(0),
\end{align*}
where $\tilde x=x(t-s)^{-1/2}$, $\tilde w=w(t-s)^{-1/2}$, $\tilde r=r(t-s)^{-1/2}$, $\tilde z=z(t-s)^{-1/2}$, and we have used the fact that
$\sqrt{T-s}\gg \la_*(s)R(s)$
for $T\ll 1$ since $\la_*(s)R(s)\sim\sqrt{T-s}|\log(T-s)|^{\theta-2}$ with $\frac{n-1}{n-2}<\theta<\frac{n}{n-2}$. Then similarly we compute

\begin{equation}\label{est-I12}
\begin{aligned}
I_{12}\lesssim&~ \int_{t-(T-t)}^{t-\la^{\delta_1}_*(t)} \la_*^{\nu-3}(s) R^{-2-\alpha}(s)\int_{\left|(\tilde r,\tilde z)-\frac{\xi(s)}{\sqrt{t-s}}\right|\leq \frac{2\la_*(s)R(s)}{\sqrt{t-s}}} (\tilde r)^{n-4} d\tilde r d\tilde z ds\\
\lesssim&~\int_{t-(T-t)}^{t-\la^{\delta_1}_*(t)} \la_*^{\nu-3}(s) R^{-2-\alpha}(s)\frac{\la_*(s)R(s)(\sqrt{T-s})^{n-4}}{(t-s)^{(n-3)/2}}\frac{\left(\la_*(s)R(s)\right)^3}{(t-s)^{3/2}} ds\\
\lesssim&~ \la_*^{\nu+\frac{(n-2)(1-\delta_1)}{2}}(t)R^{2-\alpha}(t)|\log (T-t)|
\end{aligned}
\end{equation}
and
\begin{equation}\label{est-I13}
\begin{aligned}
I_{13}\lesssim&~ \int_{t-\la^{\delta_1}_*(t)}^t \la_*^{\nu-3}(s) R^{-2-\alpha}(s) ds
\lesssim \la_*^{\nu-3+\delta_1}(t) R^{-2-\alpha}(t)|\log(T-t)|.
\end{aligned}
\end{equation}
We can choose $\delta_1=1$. Therefore, we get
$$I_{11}+I_{12}+I_{13}\lesssim \la_*^{\nu}(0) R^{2-\alpha}(0)|\log T|$$
as desired.

Similarly, to prove \eqref{outerT-rhs1}, we decompose
$$|\psi(x,t)-\psi(x,T)|\leq I_{21}+I_{22}+I_{23}$$
with
$$I_{21}=\int_0^{t-(T-t)} \int_{\R^n}|G(x-w,t-s)-G(x-w,T-s)||f(w,s)|dwds,$$
$$I_{22}=\int_{t-(T-t)}^t \int_{\R^n}|G(x-w,t-s)-G(x-w,T-s)||f(w,s)|dwds,$$
$$I_{23}=\int_t^T \int_{\R^n}|G(x-w,T-s)||f(w,s)|dwds,$$
where $G(x,t)$ is the heat kernel
\begin{equation}\label{def-heatkernel}
G(x,t)=\frac{e^{-\frac{|x|^2}{4t}}}{(4\pi t)^{n/2}}.
\end{equation}
For the first integral $I_{21}$, we have
$$I_{21}\leq (T-t)\int_0^1\int_0^{t-(T-t)}\int_{|(w_r,w_z)-\xi(s)|\leq 2\la_*(s)R(s)} |\partial_t G(x-w,t_v-s)|\la_*^{\nu-3}(s)R^{-2-\alpha}(s) dwdsdv,$$
where $(w_r,w_z)=\left(\sqrt{w_1^2+\cdots+w_{n-3}^2},w_{n-2},w_{n-1},w_n\right)$ and $t_v=vT+(1-v)t$. Changing variables
$$w_v=(w_{r,v},w_{z,v})=(w_r (t_v-s)^{-1/2},w_z (t_v-s)^{-1/2}),$$
we evaluate
\begin{equation*}
\begin{aligned}
&\quad \int_{|(w_r,w_z)-\xi(s)|\leq 2\la_*(s)R(s)} |\partial_t G(x-w,t_v-s)| dw\\
&\lesssim \int_{|(w_r,w_z)-\xi(s)|\leq 2\la_*(s)R(s)} e^{-\frac{|x-w|^2}{4(t_v-s)}}\left(\frac{|x-w|^2}{(t_v-s)^{\frac{n+4}{2}}}+\frac{1}{(t_v-s)^{\frac{n+2}{2}}}\right)dw\\
&=\int_{\left|(w_{r,v},w_{z,v})-\frac{\xi(s)}{\sqrt{t_v-s}}\right|\leq \frac{2 \la_*(s)R(s)}{\sqrt{t_v-s}}} e^{-\frac{|x_v-w_v|^2}{4}}\left(1+|x_v-w_v|^2\right)\frac{1}{t_v-s}dw_v
\end{aligned}
\end{equation*}
and thus
\begin{align*}
&\quad\int_0^{t-(T-t)}\int_{|(w_r,w_z)-\xi(s)|\leq 2\la_*(s)R(s)} |\partial_t G(x-w,t_v-s)|\la_*^{\nu-3}(s)R^{-2-\alpha}(s) dwds\\
&\lesssim \int_0^{t-(T-t)} \la_*^{\nu-3}(s)R^{-2-\alpha}(s) \frac{(T-s)^{\frac{n-4}{2}}(\la_*(s)R(s))^4}{(t_v-s)^{\frac{n+2}{2}}}ds\\
&\lesssim \int_0^{t-(T-t)} \la_*^{\nu+1}(s)R^{2-\alpha}(s) (T-s)^{-3} ds\lesssim \la_*^{\nu-1}(t)R^{2-\alpha}(t),
\end{align*}
from which we conclude that
\begin{equation}\label{est-I21}
I_{21}\lesssim \la_*^{\nu}(t)R^{2-\alpha}(t)|\log (T-t)|.
\end{equation}
For $I_{22}$, we have
\begin{equation*}
\begin{aligned}
I_{22}\leq&~ \int_{t-(T-t)}^t\int_{|(w_r,w_z)-\xi(s)|\leq 2\la_*(s)R(s)} |G(x-w,t-s)|\la_*^{\nu-3}(s)R^{-2-\alpha}(s)dwds\\
&~+ \int_{t-(T-t)}^t\int_{|(w_r,w_z)-\xi(s)|\leq 2\la_*(s)R(s)} |G(x-w,T-s)|\la_*^{\nu-3}(s)R^{-2-\alpha}(s)dwds.
\end{aligned}
\end{equation*}
The first integral above can be estimated as
\begin{equation*}
\begin{aligned}
&\quad \int_{t-(T-t)}^t \int_{|(w_r,w_z)-\xi(s)|\leq 2\la_*(s)R(s)} |G(x-w,t-s)|\la_*^{\nu-3}(s)R^{-2-\alpha}(s)dwds\\
&=\left(\int_{t-(T-t)}^{t-\la_*^{\delta_1}(t)}+\int_{t-\la_*^{\delta_1}(t)}^t\right) |G(x-w,t-s)|\la_*^{\nu-3}(s)R^{-2-\alpha}(s)dwds.\\
\end{aligned}
\end{equation*}
Notice that we already estimate the above integral in \eqref{est-I12} and \eqref{est-I13}. So with the choice $\delta_1=1$, one has
\begin{equation*}
\begin{aligned}
&\quad \int_{t-(T-t)}^t \int_{|(w_r,w_z)-\xi(s)|\leq 2\la_*(s)R(s)} |G(x-w,t-s)|\la_*^{\nu-3}(s)R^{-2-\alpha}(s)dwds \lesssim  \la_*^{\nu}(t) R^{2-\alpha}(t)|\log(T-t)|.
\end{aligned}
\end{equation*}
Similarly, it holds that
\begin{equation*}
\begin{aligned}
&\quad \int_{t-(T-t)}^t \int_{|(w_r,w_z)-\xi(s)|\leq 2\la_*(s)R(s)} |G(x-w,T-s)|\la_*^{\nu-3}(s)R^{-2-\alpha}(s)dwds \lesssim  \la_*^{\nu}(t) R^{2-\alpha}(t)|\log(T-t)|.
\end{aligned}
\end{equation*}
Therefore, we obtain
\begin{equation}\label{est-I22}
I_{22}\lesssim \la_*^{\nu}(t) R^{2-\alpha}(t)|\log(T-t)|.
\end{equation}
For $I_{23}$, using
$\tilde x=x(T-s)^{-1/2}$, $\tilde w=w(T-s)^{-1/2}$, and $(\tilde w_r,\tilde w_z)=(w_r (T-s)^{-1/2},w_z (T-s)^{-1/2}),$
one has
\begin{equation}\label{est-I23}
\begin{aligned}
I_{23}\lesssim&~ \int_t^T \int_{\left|(w_r,w_z)-\xi(s)\right|\leq 2\la_*(s)R(s)}\frac{e^{-\frac{|x-w|^2}{4(T-s)}}}{(T-s)^{n/2}} \la_*^{\nu-3}(s)R^{-2-\alpha}(s)dwds\\
\lesssim&~ \int_t^T \int_{\left|(\tilde w_r, \tilde w_z)-\frac{\xi(s)}{\sqrt{T-s}}\right|\leq \frac{2\la_*(s)R(s)}{\sqrt{T-s}}} e^{-\frac{|\tilde x-\tilde w|^2}{4}} \la_*^{\nu-3}(s)R^{-2-\alpha}(s) d\tilde w ds\\
\lesssim&~ \int_t^T  \la_*^{\nu-3}(s)R^{-2-\alpha}(s)\frac{(T-s)^{\frac{n-4}{2}}(\la_*(s)R(s))^4}{(T-s)^{n/2}} ds\\
\lesssim&~ \int_t^T \frac{|\log T|^{\nu+1}(T-s)^{\nu-2+\frac{\alpha}{2}}}{|\log (T-s)|^{2\nu+2-(2-\alpha)\theta}} ds
\lesssim \la_*^{\nu}(t)R^{2-\alpha}(t)
\end{aligned}
\end{equation}
provided $\nu-1+\frac{\alpha}2>0$.
Collecting \eqref{est-I21}, \eqref{est-I22} and \eqref{est-I23}, we conclude the validity of \eqref{outerT-rhs1}.

Then we prove the gradient estimate \eqref{outergradient-rhs1}. By the heat kernel, we get
\begin{equation*}
\begin{aligned}
|\nabla \psi(x,t)|\lesssim&~ \int_0^t \frac{\la_*^{\nu-3}(s)R^{-2-\alpha}(s)}{(t-s)^{\frac{n+2}{2}}}\int_{\left|(w_r,w_z)-\xi(s)\right|\leq2 \la_*(s)R(s)} e^{-\frac{|x-w|^2}{4(t-s)}}|x-w| dwds\\
\lesssim&~ \int_0^t \frac{\la_*^{\nu-3}(s)R^{-2-\alpha}(s)}{(t-s)^{1/2}}\int_{\left|(\tilde w_r,\tilde w_z)-\frac{\xi(s)}{\sqrt{t-s}}\right|\leq \frac{2 \la_*(s)R(s)}{\sqrt{t-s}}}  e^{-\frac{|\tilde w|^2}{4}}(1+|\tilde w|) d\tilde wds,\\
\end{aligned}
\end{equation*}
where $\tilde x=x(t-s)^{-1/2}$,
$(w_r,w_z)=\left(\sqrt{w_1^2+\cdots+w_{n-3}^2},w_{n-2},w_{n-1},w_n\right),$
and
$$(\tilde w_r,\tilde w_z)=\left(\sqrt{\tilde w_1^2+\cdots+\tilde w_{n-3}^2},\tilde w_{n-2},\tilde w_{n-1},\tilde w_n\right).$$
First, we compute
\begin{equation}\label{grad-rhs11}
\begin{aligned}
&\quad\int_0^{t-(T-t)} \frac{\la_*^{\nu-3}(s)R^{-2-\alpha}(s)}{(t-s)^{1/2}}\int_{\left|(\tilde w_r,\tilde w_z)-\frac{\xi(s)}{\sqrt{t-s}}\right|\leq \frac{2 \la_*(s)R(s)}{\sqrt{t-s}}} e^{-\frac{|\tilde w|^2}{4}}(1+|\tilde w|) d\tilde wds\\
&\lesssim \int_0^{t-(T-t)} \frac{\la_*^{\nu-3}(s)R^{-2-\alpha}(s)}{(t-s)^{1/2}}\frac{\la_*(s)R(s)(\sqrt{T-s})^{n-4}}{(t-s)^{(n-3)/2}}\frac{\left(\la_*(s)R(s)\right)^3}{(t-s)^{3/2}} ds\\
&\lesssim \int_0^{t-(T-t)} \frac{\la_*^{\nu+1}(s)R^{2-\alpha}(s)(T-s)^{\frac{n-4}{2}}}{(t-s)^{\frac{n+1}{2}}} ds\\
&\lesssim \int_0^{t-(T-t)} \la_*^{\nu+1}(s)R^{2-\alpha}(s)(T-s)^{-\frac{5}{2}} ds\lesssim  \la_*^{\nu-\frac12}(0)R^{2-\alpha}(0)|\log T|.
\end{aligned}
\end{equation}
Then we compute
\begin{align*}
&\quad\int_{t-(T-t)}^{t-\la_*^{\delta_2}(t)} \frac{\la_*^{\nu-3}(s)R^{-2-\alpha}(s)}{(t-s)^{1/2}}\int_{\left|(\tilde w_r,\tilde w_z)-\frac{\xi(s)}{\sqrt{t-s}}\right|\leq \frac{2 \la_*(s)R(s)}{\sqrt{t-s}}} e^{-\frac{|\tilde w|^2}{4}}(1+|\tilde w|) d\tilde wds\\
&\lesssim \int_{t-(T-t)}^{t-\la_*^{\delta_2}(t)} \frac{\la_*^{\nu+1}(s)R^{2-\alpha}(s)(T-s)^{\frac{n-4}{2}}}{(t-s)^{\frac{n+1}{2}}} ds \lesssim \la_*^{\nu+\frac{n-2}{2}+\frac{(1-n)\delta_2}{2}}(t)R^{2-\alpha}(t) |\log(T-t)|,
\end{align*}
where $\delta_2\geq 1$ is a constant to be determined. On the other hand, we have
\begin{equation}\label{grad-rhs13}
\begin{aligned}
&\quad\int_{t-\la_*^{\delta_2}(t)}^t \frac{\la_*^{\nu-3}(s)R^{-2-\alpha}(s)}{(t-s)^{\frac{n+2}{2}}}\int_{\left|(w_r,w_z)-\xi(s)\right|\leq2 \la_*(s)R(s)} e^{-\frac{|x-w|^2}{4(t-s)}}|x-w| dwds\\
&\lesssim \int_{t-\la_*^{\delta_2}(t)}^t \frac{\la_*^{\nu-3}(s)R^{-2-\alpha}(s)}{(t-s)^{1/2}} ds \lesssim \la_*^{\nu-3+\frac{\delta_2}{2}}(t)R^{-2-\alpha}(t).
\end{aligned}
\end{equation}
By choosing $\delta_2=1$ and combining \eqref{grad-rhs11}--\eqref{grad-rhs13}, we prove the validity of the gradient estimate \eqref{outergradient-rhs1}. The proof of \eqref{outergradientT-rhs1} is similar to that of \eqref{outerT-rhs1}. We omit the details.

To prove the H\"older estimate \eqref{outerholder-rhs1}, we decompose
$$|\psi(x,t_2)-\psi(x,t_1)|\leq J_{11}+J_{12}+J_{13}$$
with
$$J_{11}=\int_0^{t_1-(t_2-t_1)}\int_{\R^n} |G(x-w,t_2-s)-G(x-w,t_1-s)|f(w,s)dwds,$$
$$J_{12}=\int_{t_1-(t_2-t_1)}^{t_1}\int_{\R^n} |G(x-w,t_2-s)-G(x-w,t_1-s)|f(w,s)dwds,$$
and
$$J_{13}=\int_{t_1}^{t_2}\int_{\R^n} G(x-w,t_2-s)f(w,s)dwds,$$
where $G(x,t)$ is the heat kernel \eqref{def-heatkernel}. Here we assume that $0<t_1<t_2<T$ with $t_2<2t_1$.
For $J_{11}$, by letting $t_v=vt_2+(1-v)t_1$, we have
\begin{align*}
J_{11}\leq &~(t_2-t_1)\int_0^1 \int_0^{t_1-(t_2-t_1)}\int_{\R^n} |\partial_t G(x-w,t_v-s)|f(w,s)dwdsdv\\
\lesssim &~ (t_2-t_1)\int_0^1 \int_0^{t_1-(t_2-t_1)}\int_{\left|(w_r,w_z)-\xi(s)\right|\leq2 \la_*(s)R(s)} e^{-\frac{|x-w|^2}{4(t_v-s)}}\bigg(\frac{|x-w|^2}{(t_v-s)^{\frac{n+4}{2}}}\\
&\qquad\qquad\qquad\qquad\qquad\qquad+\frac{1}{(t_v-s)^{\frac{n+2}{2}}}\bigg) \la_*^{\nu-3}(s)R^{-2-\alpha}(s) dwdsdv,\\
\end{align*}
where $(w_r,w_z)=\left(\sqrt{w_1^2+\cdots+w_{n-3}^2},w_{n-2},w_{n-1},w_n\right).$ Taking
$x_v=x(t_v-s)^{-1/2}$, $w_v=w(t_v-s)^{-1/2}$, and $(w_{r,v},w_{z,v})=(w_r (t_v-s)^{-1/2},w_z (t_v-s)^{-1/2}),$ we get
\begin{equation*}
\begin{aligned}
&~\int_{\left|(w_r,w_z)-\xi(s)\right|\leq2 \la_*(s)R(s)} e^{-\frac{|x-w|^2}{4(t_v-s)}}\left(\frac{|x-w|^2}{(t_v-s)^{\frac{n+4}{2}}}+\frac{1}{(t_v-s)^{\frac{n+2}{2}}}\right)\la_*^{\nu-3}(s)R^{-2-\alpha}(s) dw\\
=&~\int_{\left|(w_{r,v},w_{z,v})-\frac{\xi(s)}{\sqrt{t_v-s}}\right|\leq \frac{2 \la_*(s)R(s)}{\sqrt{t_v-s}}} e^{-\frac{|x_v-w_v|^2}{4}}\left(1+|x_v-w_v|^2\right)\frac{\la_*^{\nu-3}(s)R^{-2-\alpha}(s)}{t_v-s}dw_v.\\
\end{aligned}
\end{equation*}
Observing that for any $\mu\in(0,1)$, we have
$$\int_{\left|(w_{r,v},w_{z,v})-\frac{\xi(s)}{\sqrt{t_v-s}}\right|\leq \frac{2 \la_*(s)R(s)}{\sqrt{t_v-s}}} e^{-\frac{|x_v-w_v|^2}{4}}\left(1+|x_v-w_v|^2\right) dw_v \lesssim \left(\frac{\sqrt{T-s}}{\sqrt{t_v-s}}\right)^{\mu},$$
where we have used the facts that $\xi(s)\sim(\sqrt{2(n-4)(T-s)},0,0,0)$ and $\sqrt{T-s}\gg \la_*(s)R(s)$ for $T\ll 1$. Thus, one has
\begin{equation*}
\begin{aligned}
J_{11}\lesssim (t_2-t_1) \int_0^{t_1-(t_2-t_1)} \frac{\la_*^{\nu-3}(s)R^{-2-\alpha}(s)(T-s)^{\frac{\mu}{2}}}{(t_2-s)^{1+\frac{\mu}{2}}}ds.
\end{aligned}
\end{equation*}
Recalling that $R(t)=(T-t)^{-1/2}|\log(T-t)|^{\theta}$ for $\theta\in\left(\frac{n-1}{n-2},\frac{n}{n-2}\right)$, we have the following two cases:

\noindent $\bullet$ If $\nu-3+\frac12(2+\alpha)+\frac{\mu}{2}< 0$, then we have
\begin{align*}
&\quad \int_0^{t_1-(t_2-t_1)} \frac{\la_*^{\nu-3}(s)R^{-2-\alpha}(s)(T-s)^{\frac{\mu}{2}}}{(t_2-s)^{1+\frac{\mu}{2}}}ds\\
&\lesssim \la_*^{\nu-3+\frac{\mu}{2}}(t_1)R^{-2-\alpha}(t_1)|\log(T-t_1)|\int_0^{t_1-(t_2-t_1)}\frac{1}{(t_2-s)^{1+\frac{\mu}{2}}}ds\\
&\lesssim \la_*^{\nu-3+\frac{\mu}{2}}(t_1)R^{-2-\alpha}(t_1)|\log(T-t_1)|(t_2-t_1)^{-\mu/2}.
\end{align*}

\noindent $\bullet$ If $\nu-3+\frac12(2+\alpha)+\frac{\mu}{2}\geq 0$, then we decompose
\begin{align*}
&\quad\int_0^{t_1-(t_2-t_1)} \frac{\la_*^{\nu-3}(s)R^{-2-\alpha}(s)(T-s)^{\frac{\mu}{2}}}{(t_2-s)^{1+\frac{\mu}{2}}}ds\\
&= \left(\int_0^{t_1-(T-t_1)}+\int_{t_1-(T-t_1)}^{t_1-(t_2-t_1)}\right) \frac{\la_*^{\nu-3}(s)R^{-2-\alpha}(s)(T-s)^{\frac{\mu}{2}}}{(t_2-s)^{1+\frac{\mu}{2}}}ds.
\end{align*}
Assuming
$\nu-3+\frac12(2+\alpha)<0,$
we obtain that
\begin{equation*}
\begin{aligned}
 &~\int_0^{t_1-(T-t_1)} \frac{\la_*^{\nu-3}(s)R^{-2-\alpha}(s)(T-s)^{\frac{\mu}{2}}}{(t_2-s)^{1+\frac{\mu}{2}}}ds
\lesssim \int_0^{t_1-(T-t_1)} \frac{\la_*^{\nu-3}(s)R^{-2-\alpha}(s)(T-s)^{\frac{\mu}{2}}}{(T-s)^{1+\frac{\mu}{2}}}ds\\
=&~\int_0^{t_1-(T-t_1)}\frac{|\log T|^{\nu-3}(T-s)^{\nu-4+\frac12(2+\alpha)}}{|\log(T-s)|^{2(\nu-3)+\theta(2+\alpha)}} ds
\lesssim\la_*^{\nu+\frac{\mu}{2}-3}(t_2)R^{-2-\alpha}(t_2)(t_2-t_1)^{-\mu/2}
\end{aligned}
\end{equation*}
and similarly
\begin{equation*}
\begin{aligned}
\int_{t_1-(T-t_1)}^{t_1-(t_2-t_1)} \frac{\la_*^{\nu-3}(s)R^{-2-\alpha}(s)(T-s)^{\frac{\mu}{2}}}{(t_2-s)^{1+\frac{\mu}{2}}}ds\lesssim \la_*^{\nu+\frac{\mu}{2}-3}(t_2)R^{-2-\alpha}(t_2)(t_2-t_1)^{-\mu/2}.
\end{aligned}
\end{equation*}
In both cases, we have
$$J_{11}\lesssim \la_*^{\nu+\frac{\mu}{2}-3}(t_2)R^{-2-\alpha}(t_2)(t_2-t_1)^{1-\mu/2}.$$

For $J_{12}$, we evaluate
\begin{align*}
&~\int_{t_1-(t_2-t_1)}^{t_1}\int_{\R^n} |G(x-w,t_1-s)|f(w,s)dwds\\
\lesssim&~ \int_{t_1-(t_2-t_1)}^{t_1} \la_*^{\nu-3}(s)R^{-2-\alpha}(s)\int_{\frac{\left|(w_r,w_z)-\xi(s)\right|}{\sqrt{t_1-s}}\leq \frac{2 \la_*(s)R(s)}{\sqrt{t_1-s}}} e^{-\frac{|\tilde x-\tilde w|^2}{4}}d \tilde w ds\\
\lesssim&~ \int_{t_1-(t_2-t_1)}^{t_1} \la_*^{\nu-3}(s)R^{-2-\alpha}(s)\left(\frac{\sqrt{T-s}}{\sqrt{t_1-s}}\right)^{\mu} ds\\
=&~ \int_{t_1-(t_2-t_1)}^{t_1} \frac{|\log T|^{\nu-3}(T-s)^{\nu+\frac{\mu}{2}-3+\frac12(2+\alpha)}}{|\log(T-s)|^{2(\nu-3)+\theta(2+\alpha)}(t_1-s)^{\mu/2}} ds\\
\lesssim&~ \la_*^{\nu+\frac{\mu}{2}-3}(t_2)R^{-2-\alpha}(t_2) (t_2-t_1)^{1-\mu/2},
\end{align*}
where we have changed variables $\tilde x=x(t_1-s)^{-1/2}$, $\tilde w=w(t_1-s)^{-1/2}$, and $\mu\in(0,1)$. Similarly, we have
$$\int_{t_1-(t_2-t_1)}^{t_1}\int_{\R^n} |G(x-w,t_2-s)|f(w,s)dwds\lesssim \la_*^{\nu+\frac{\mu}{2}-3}(t_2)R^{-2-\alpha}(t_2) (t_2-t_1)^{1-\mu/2}.$$
Thus we conclude that
$$J_{12}\lesssim \la_*^{\nu+\frac{\mu}{2}-3}(t_2)R^{-2-\alpha}(t_2) (t_2-t_1)^{1-\mu/2}.$$
Finally, for $J_{13}$,
\begin{align*}
J_{13}=&~\int_{t_1}^{t_2}\int_{\R^n} G(x-w,t_2-s)f(w,s)dwds\\
\lesssim&~ \int_{t_1}^{t_2} \la_*^{\nu-3}(s)R^{-2-\alpha}(s) \int_{\frac{\left|(w_r,w_z)-\xi(s)\right|}{\sqrt{t_2-s}}\leq \frac{2 \la_*(s)R(s)}{\sqrt{t_2-s}}} e^{-\frac{|\tilde x-\tilde w|^2}{4}}d \tilde w ds\\
\lesssim&~ \la_*^{\nu+\frac{\mu}{2}-3}(t_2)R^{-2-\alpha}(t_2) (t_2-t_1)^{1-\mu/2}
\end{align*}
follows from the same argument as before, where $\tilde x=x(t_2-s)^{-1/2}$ and $\tilde w=w(t_2-s)^{-1/2}$. This completes the proof of \eqref{outerholder-rhs1}.
\end{proof}

The proofs of Lemma \ref{lemma-rhs2} and Lemma \ref{lemma-rhs3} are similar to these carried out above. We omit the details.


\bigskip

\section{Improvement via axisymmetric heat equation}\label{App-B}

\medskip

\subsection{Axisymmetric heat kernel}\label{App-heatkernel}

Using Hankel-Fourier transform, we derive the heat kernel to the heat operator in the axisymmetric class
$$
L=\partial_t-\partial_{rr} -\frac{n-4}{r}\partial_r -\Delta_{\R^3}.
$$
Consider the initial value problem
$$
\left\{
\begin{aligned}
    &~v_t=v_{rr}+\frac{n-4}{r}v_r+\Delta_z v ~&\mbox{ in }~\R_+\times \R^3\times\R_+,\\
    &~v(r,z,0)=v_0(r,z)~&\mbox{ in }~\R_+\times \R^3.
\end{aligned}
\right.
$$
Let
$$
v(r,z,t)=r^{\frac{5-n}{2}}u(r,z,t).
$$
Then
\begin{equation}\label{eqn-uB16}
\left\{
\begin{aligned}
    &~u_t=u_{rr}+\frac{1}{r}u_r-\left(\frac{n-5}{2}\right)^2\frac{u}{r^2}+\Delta_z u ~&\mbox{ in }~\R_+\times \R^3\times\R_+,\\
    &~u(r,z,0)=r^{\frac{n-5}{2}}v_0(r,z)~&\mbox{ in }~\R_+\times \R^3.
\end{aligned}
\right.
\end{equation}
We define the Hankel/Bessel-Fourier transform
$$
    \mathcal F[u](\rho,\omega;t):=(2\pi)^{-3/2}\int_{\R^3}\int_{\R_+}v(r,z,t)J_{\frac{n-5}{2}}(\rho r) r e^{-iz\cdot\omega} dr dz,
$$
where $J_{\frac{n-5}{2}}$ is the Bessel function. Acting this on \eqref{eqn-uB16} gives
$$
    \partial_t \mathcal F[u]=-(\rho^2+|\omega|^2)\mathcal F[u],\quad \mathcal F[u](\rho,\omega;0)=\mathcal F[r^{\frac{n-5}{2}}v_0].
$$
So
$$
\mathcal F[u](\rho,\omega,t)=\mathcal F[r^{\frac{n-5}{2}}v_0] \exp\{-(\rho^2+|\omega|^2)t\}.
$$
Taking inverse Hankel-Fourier transform, we get
\begin{align*}
    &~u(r,z,t)\\
    =&~(2\pi)^{-3/2}\int_{\R^3}\int_{\R_+} F[r^{\frac{n-5}{2}}v_0] e^{-(\rho^2+|\omega|^2)t}J_{\frac{n-5}{2}}(\rho r)e^{iz\cdot\omega}\rho d\rho d\omega\\
    =&~(2\pi)^{-3}\int_{\R^3}\int_{\R_+} \left(\int_{\R^3}\int_{\R_+} s^{\frac{n-5}{2}}v_0(s,\zeta,t)J_{\frac{n-5}{2}}(s\rho) s e^{-i\zeta\cdot\omega} ds d\zeta\right) e^{-(\rho^2+|\omega|^2)t}J_{\frac{n-5}{2}}(\rho r)e^{iz\cdot\omega}\rho d\rho d\omega\\
    =&~(2\pi)^{-3}\int_{\R^3}\int_{\R^3} e^{i\omega\cdot(z-\zeta)-|\omega|^2 t}\left(\int_{\R_+}v_0(s,\zeta,t) s^{\frac{n-3}{2}} ds\int_{\R_+} e^{-\rho^2 t} J_{\frac{n-5}{2}}(s\rho) J_{\frac{n-5}{2}}(\rho r)d\rho \right)d\zeta d\omega\\
    =&~(2\pi)^{-3}\int_{\R^3}\int_{\R^3} e^{i\omega\cdot(z-\zeta)-|\omega|^2 t}\left(\int_{\R_+}v_0(s,\zeta,t) \frac1{2t}\exp\left\{-\frac{r^2+s^2}{4t}\right\} I_{\frac{n-5}{2}}\left(\frac{rs}{2t}\right) s^{\frac{n-3}{2}} ds \right)d\zeta d\omega\\
    =&~(2\pi)^{-3}\int_{\R_+} \frac1{2t}\exp\left\{-\frac{r^2+s^2}{4t}\right\} I_{\frac{n-5}{2}}\left(\frac{rs}{2t}\right) s^{\frac{n-3}{2}} ds  \int_{\R^3} v_0(s,\zeta,t) d\zeta\int_{\R^3} e^{i\omega\cdot(z-\zeta)-|\omega|^2 t} d\omega\\
    =&~(2\pi)^{-3}\int_{\R_+} \frac1{2t}\exp\left\{-\frac{r^2+s^2}{4t}\right\} I_{\frac{n-5}{2}}\left(\frac{rs}{2t}\right) s^{\frac{n-3}{2}} ds  \int_{\R^3} v_0(s,\zeta,t) \left(\frac{\pi}{t}\right)^{3/2}\exp\left\{-\frac{|z-\zeta|^2}{4t}\right\} d\zeta
\end{align*}
where we have used the formula
$$
\int_{\R_+} e^{-\rho^2 t}J_{\frac{n-5}{2}}(\rho s) J_{\frac{n-5}{2}}(\rho r) \rho d\rho =\frac1{2t}\exp\left\{-\frac{r^2+s^2}{4t}\right\} I_{\frac{n-5}{2}}\left(\frac{rs}{2t}\right),
$$
and $I_{\frac{n-5}{2}}$ is the modified Bessel function. Returning to $v$, we have
$$
    \begin{aligned}
        v(r,z,t)=r^{\frac{5-n}{2}}u(r,z,t)=\int_{\R^3}\int_{\R_+} \Gamma_n(r,z;s,\zeta;t) v_0(s,\zeta,t)ds d\zeta
    \end{aligned}
$$
with
\begin{equation}\label{axi-heatkernel}
    \Gamma_n(r,z;s,\zeta;t)=\frac1{16\pi^{3/2}}\frac{r^{\frac{5-n}{2}}s^{\frac{n-3}{2}}}{t^{5/2}}\exp\left\{-\frac{r^2+s^2+|z-\zeta|^2}{4t}\right\} I_{\frac{n-5}{2}}\left(\frac{rs}{2t}\right).
\end{equation}
The representation formula for the non-homogeneous problem can be obtained via Duhamel principle.


\medskip

{
\subsection{Proof of the asymptotic expansion \texorpdfstring{\eqref{nonlocal-oc0}}{}}\label{App-nonlocal}

We prove the asymptotic expansion \eqref{nonlocal-oc0}. Throughout this subsection, dimensional nonzero constants may change from line to line. We write
\[
r_t=c_n\sqrt{T-t},
\qquad
r_s=c_n\sqrt{T-s},
\qquad
c_n=\sqrt{2(n-4)}.
\]
Recall from \eqref{def-Psi0} that
\[
\begin{aligned}
\Psi_0(r,z,t)
=&-\frac{\sqrt2}{8\pi^{3/2}}
r^{\frac{5-n}{2}}
\int_0^t
\frac{\dot\lambda(s)}{(t-s)^{5/2}}
\exp\left\{-\frac{r^2}{4(t-s)}\right\}\,ds  \\
&\times
\int_{\R^3}
\exp\left\{-\frac{|z-\tilde z|^2}{4(t-s)}\right\}\,d\tilde z \\
&\times
\int_{\R_+}
\exp\left\{-\frac{\tilde r^2}{4(t-s)}\right\}
\frac{\tilde r^{\frac{n-3}{2}}}
{(\tilde r-r_s)^2+|\tilde z|^2}
I_{\frac{n-5}{2}}
\left(
\frac{r\tilde r}{2(t-s)}
\right)
\,d\tilde r .
\end{aligned}
\]
We need to analyze
\[
\mathcal P(t)
:=
\int_{B_{4R}}
3U^2(y_1,y')
\Psi_0(\lambda y_1+r_t,\lambda y',t)
Z_5(y_1,y')\,dy_1dy'.
\]
Since
\[
\lambda(t)R(t)\ll \sqrt{T-t},
\]
we may replace, at leading order,
\[
\lambda y_1+r_t=r_t(1+o(1))
\]
inside the slowly varying prefactors. The errors produced by this replacement are of lower order and are absorbed into the final $O(|\dot\lambda(t)|)$ remainder.

Substituting the representation of $\Psi_0$ and changing variables
\[
\hat r=\frac{\tilde r-r_s}{\sqrt{t-s}},
\qquad
\hat z=\frac{\tilde z}{\sqrt{t-s}},
\]
we obtain
\[
\begin{aligned}
\mathcal P(t)
\sim&
(T-t)^{\frac{5-n}{4}}
\int_0^t
\frac{\dot\lambda(s)}{(t-s)^{3/2}}\,ds \\
&\times
\int_{\R^3}
\int_{-r_s/\sqrt{t-s}}^\infty
\mathcal K(\hat r,\hat z,s,t)\,d\hat r\,d\hat z ,
\end{aligned}
\]
where
\begin{align*}
\mathcal K(\hat r,\hat z,s,t)
:=&
\exp\left\{
-\frac{r_t^2+(\sqrt{t-s}\,\hat r+r_s)^2}{4(t-s)}
\right\}
\frac{(\sqrt{t-s}\,\hat r+r_s)^{\frac{n-3}{2}}}
{\hat r^2+|\hat z|^2} \\
&\times
I_{\frac{n-5}{2}}
\left(
\frac{r_t(\sqrt{t-s}\,\hat r+r_s)}{2(t-s)}
\right) \\
&\times
\int_{B_{4R}}
3U^2(y_1,y')Z_5(y_1,y')
\exp\left\{
-\frac{\left|\frac{\lambda y'}{\sqrt{t-s}}-\hat z\right|^2}{4}
\right\}
dy_1dy' .
\end{align*}
The argument of the Bessel function is
\[
\frac{r_t(\sqrt{t-s}\,\hat r+r_s)}{2(t-s)}.
\]
We split the time integral into three regions:
\[
\mathcal P(t)=I+J+K,
\]
where
\[
I:\ 0<s<t-(T-t),
\qquad
J:\ t-(T-t)<s<t-\lambda^2(t),
\qquad
K:\ t-\lambda^2(t)<s<t.
\]

\medskip

\noindent\textit{The region $0<s<t-(T-t)$.}
In this region the small-argument asymptotic of the Bessel function in \eqref{Bessel-main} gives, at leading order,
\[
\begin{aligned}
I\sim&
C_1
\int_0^{t-(T-t)}
\frac{\dot\lambda(s)}{t-s}
\exp\left\{
-\frac{c_n^2(T-t)}{4(t-s)}
\right\}
\,ds \\
&\times
\int_{\R^3}
\int_{-A_n}^\infty
\exp\left\{-\frac{|\hat z|^2}{4}\right\}
\frac{(\hat r+A_n)^{n-4}}
{\hat r^2+|\hat z|^2}
\exp\left\{-\frac{(\hat r+A_n)^2}{4}\right\}
\,d\hat r\,d\hat z ,
\end{aligned}
\]
where
\[
A_n=c_n\sqrt{\frac{T-s}{t-s}}.
\]
For $A_n\leq 1$,
\[
\begin{aligned}
&\int_{\R^3}
\int_{-A_n}^\infty
\exp\left\{-\frac{|\hat z|^2}{4}\right\}
\frac{(\hat r+A_n)^{n-4}}
{\hat r^2+|\hat z|^2}
\exp\left\{-\frac{(\hat r+A_n)^2}{4}\right\}
\,d\hat r\,d\hat z \\
&\qquad =
\frac{2^{n-2}\pi^{3/2}
\Gamma\left(\frac{n-3}{2}\right)}
{n-2}
+O(A_n).
\end{aligned}
\]
Thus
\[
I
=
C_2
\int_0^{t-(T-t)}
\frac{\dot\lambda(s)}{t-s}
\exp\left\{
-\frac{c_n^2(T-t)}{4(t-s)}
\right\}
\,ds
+O(|\dot\lambda(t)|).
\]
Since the exponential factor only changes the dimensional constant at the level of the main order, this gives
\[
I
=
C_2
\int_0^{t-(T-t)}
\frac{\dot\lambda(s)}{t-s}\,ds
+O(|\dot\lambda(t)|).
\]

\medskip

\noindent\textit{The region $t-(T-t)<s<t-\lambda^2(t)$.}
In this region the argument of the Bessel function is large. Hence, by \eqref{Bessel-main},
\[
\begin{aligned}
&I_{\frac{n-5}{2}}
\left(
\frac{r_t(\sqrt{t-s}\,\hat r+r_s)}{2(t-s)}
\right) \\
&\qquad\sim
\frac{1}{\sqrt{2\pi}}
\frac{\sqrt{2(t-s)}}
{\big[r_t(\sqrt{t-s}\,\hat r+r_s)\big]^{1/2}}
\exp\left\{
\frac{r_t(\sqrt{t-s}\,\hat r+r_s)}
{2(t-s)}
\right\}.
\end{aligned}
\]
The exponential factor cancels the corresponding part of the Gaussian kernel, and we obtain
\begin{align*}
J\sim&
C_3
(T-t)^{\frac{4-n}{4}}
\int_{t-(T-t)}^{t-\lambda^2(t)}
\frac{\dot\lambda(s)}
{(t-s)^{\frac{8-n}{4}}}
\,ds \\
&\times
\int_{\R^3}
\int_{-A_n}^\infty
\exp\left\{-\frac{\hat r^2}{4}\right\}
\exp\left\{-\frac{|\hat z|^2}{4}\right\}
\frac{(\hat r+A_n)^{\frac{n-4}{2}}}
{\hat r^2+|\hat z|^2}
\,d\hat r\,d\hat z .
\end{align*}
For $A_n\geq 1$,
\[
\begin{aligned}
&\int_{\R^3}
\int_{-A_n}^\infty
\exp\left\{-\frac{\hat r^2}{4}\right\}
\exp\left\{-\frac{|\hat z|^2}{4}\right\}
\frac{(\hat r+A_n)^{\frac{n-4}{2}}}
{\hat r^2+|\hat z|^2}
\,d\hat r\,d\hat z \\
&\qquad =
4\pi^2 A_n^{\frac{n-4}{2}}
+O\left(A_n^{\frac{n-6}{2}}\right).
\end{aligned}
\]
Therefore
\[
\begin{aligned}
J
=&
C_4
(T-t)^{\frac{4-n}{4}}
\int_{t-(T-t)}^{t-\lambda^2(t)}
\frac{\dot\lambda(s)}
{(t-s)^{\frac{8-n}{4}}}
\left(
\frac{T-s}{t-s}
\right)^{\frac{n-4}{4}}
\,ds
+O(|\dot\lambda(t)|) \\
=&
C_4
\int_{t-(T-t)}^{t-\lambda^2(t)}
\frac{\dot\lambda(s)}{t-s}\,ds
+O(|\dot\lambda(t)|).
\end{aligned}
\]

\medskip

\noindent\textit{The region $t-\lambda^2(t)<s<t$.}
The same large-argument asymptotic gives
\[
|K|
\lesssim
(T-t)^{\frac{4-n}{4}}
\left|
\int_{t-\lambda^2(t)}^t
\frac{\dot\lambda(s)}{t-s}
\frac{(t-s)^\alpha}
{\lambda^{2\alpha}(t)}
(T-s)^{\frac{n-4}{4}}
\,ds
\right|
\]
for some small $\alpha>0$. Since $t-s<\lambda^2(t)$ in this region,
\[
|K|
\lesssim
(T-t)^{\frac{4-n}{4}}
|\dot\lambda(t)|
(T-t)^{\frac{n-4}{4}}
=
O(|\dot\lambda(t)|).
\]

\medskip

Combining the estimates for $I$, $J$, and $K$, we get
\[
\mathcal P(t)
=
C_2
\int_0^{t-(T-t)}
\frac{\dot\lambda(s)}{t-s}\,ds
+
C_4
\int_{t-(T-t)}^{t-\lambda^2(t)}
\frac{\dot\lambda(s)}{t-s}\,ds
+
O(|\dot\lambda(t)|).
\]
The computation of the constants above gives
\[
\frac{C_4}{C_2}
=
\frac{n-2}{2}.
\]
Thus, writing $C=C_2>0$, we obtain
\[
\mathcal P(t)
=
C
\left(
\int_0^{t-(T-t)}
\frac{\dot\lambda(s)}{t-s}\,ds
+
c_n^*
\int_{t-(T-t)}^{t-\lambda^2(t)}
\frac{\dot\lambda(s)}{t-s}\,ds
\right)
+
O(|\dot\lambda(t)|),
\]
where
\[
c_n^*=\frac{n-2}{2}.
\]
This proves \eqref{nonlocal-oc0}.

}

\subsection{Estimate of the correction terms}\label{App-est-Psi_0}

We estimate the size of $\Psi_0$ that is defined by
\begin{align*}
&~\Psi_0(r,z,t)\\
=&~-\frac{\sqrt2}{8\pi^{3/2}} r^{\frac{5-n}2}\int_0^t \frac{\dot\la(s)}{(t-s)^{5/2}}\exp\left\{-\frac{r^2}{4(t-s)}\right\} ds\int_{\R^3} \exp\left\{-\frac{|z-\tilde z|^2}{4(t-s)}\right\}d\tilde z\\
&~\qquad\int_{\R_+} \exp\left\{-\frac{\tilde r^2}{4(t-s)}\right\}\frac{\tilde r^{\frac{n-3}{2}}}{(\tilde r-c_n\sqrt{T-s})^2+|\tilde z|^2} I_{\frac{n-5}{2}}\left(\frac{r\tilde r}{2(t-s)}\right) d\tilde r\\
=&~-\frac{\sqrt2}{8\pi^{3/2}} r^{\frac{5-n}2}\int_0^t \frac{\dot\la(s)}{(t-s)^{5/2}}\exp\left\{-\frac{r^2}{4(t-s)}\right\} ds\int_{\R^3} \exp\left\{-\frac{|z-\tilde z|^2}{4(t-s)}\right\}d\tilde z\\
&~\qquad\left(\int_0^{\frac{2(t-s)}{r}}+\int_{\frac{2(t-s)}{r}}^\infty\right) \exp\left\{-\frac{\tilde r^2}{4(t-s)}\right\}\frac{\tilde r^{\frac{n-3}{2}}}{(\tilde r-c_n\sqrt{T-s})^2+|\tilde z|^2} I_{\frac{n-5}{2}}\left(\frac{r\tilde r}{2(t-s)}\right) d\tilde r\\
:=&~{\rm I}+{\rm II}.
\end{align*}
We estimate
\begin{align*}
        |{\rm I}|\lesssim &~\Bigg|\int_0^t \frac{\dot\la(s)}{(t-s)^{n/2}}\exp\left\{-\frac{r^2}{4(t-s)}\right\} ds\int_{\R^3} \exp\left\{-\frac{|z-\tilde z|^2}{4(t-s)}\right\}d\tilde z\\
&~\qquad\quad\int_0^{\frac{2(t-s)}{r}} \exp\left\{-\frac{\tilde r^2}{4(t-s)}\right\}\frac{\tilde r^{n-4}}{(\tilde r-c_n\sqrt{T-s})^2+|\tilde z|^2}  d\tilde r\Bigg|\\
=&~\Bigg|\left(\int_0^{t-\frac{r^2}{4}}+\int^t_{t-\frac{r^2}{4}}\right) \frac{\dot\la(s)}{(t-s)^{n/2}}\exp\left\{-\frac{r^2}{4(t-s)}\right\} ds\int_{\R^3} \exp\left\{-\frac{|z-\tilde z|^2}{4(t-s)}\right\}d\tilde z\\
&~\qquad\quad\int_0^{\frac{2(t-s)}{r}} \exp\left\{-\frac{\tilde r^2}{4(t-s)}\right\}\frac{\tilde r^{n-4}}{(\tilde r-c_n\sqrt{T-s})^2+|\tilde z|^2}  d\tilde r\Bigg|,
    \end{align*}
and
\begin{align*}
        |{\rm II}|\lesssim &~r^{\frac{4-n}{2}}\Bigg|\int_0^t \frac{\dot\la(s)}{(t-s)^{2}} ds\int_{\R^3} \exp\left\{-\frac{|z-\tilde z|^2}{4(t-s)}\right\}d\tilde z\\
&~\qquad\quad\int_{\frac{2(t-s)}{r}}^\infty \exp\left\{-\frac{(\tilde r-r)^2}{4(t-s)}\right\}\frac{\tilde r^{\frac{n-4}{2}}}{(\tilde r-c_n\sqrt{T-s})^2+|\tilde z|^2}  d\tilde r\Bigg|.
    \end{align*}
To estimate ${\rm I}$, we split different regions. 

\noindent $\bullet$
If $0<r\leq2\sqrt{T-t}$, then
\begin{align*}
    |{\rm I}|\lesssim&~\Bigg|\left(\int_0^{t-(T-t)}+\int_{t-(T-t)}^{t-\frac{r^2}{4}}+\int^t_{t-\frac{r^2}{4}}\right) \frac{\dot\la(s)}{(t-s)^{n/2}}\exp\left\{-\frac{r^2}{4(t-s)}\right\} ds\int_{\R^3} \exp\left\{-\frac{|z-\tilde z|^2}{4(t-s)}\right\}d\tilde z\\
&~\qquad\quad\int_0^{\frac{2(t-s)}{r}} \exp\left\{-\frac{\tilde r^2}{4(t-s)}\right\}\frac{\tilde r^{n-4}}{(\tilde r-c_n\sqrt{T-s})^2+|\tilde z|^2}  d\tilde r\Bigg|\\
:=&~{\rm I}_1+{\rm I}_2+{\rm I}_3.
\end{align*}
For any $\alpha>0$,
\begin{align*}
    &~|{\rm I}_1|\\
    \lesssim
    &~\Bigg|\int_0^{t-(T-t)}\frac{\dot\la(s)}{(t-s)^{n/2}} ds\int_{\R^3} \exp\left\{-\frac{|z-\tilde z|^2}{4(t-s)}\right\}d\tilde z\\
&~\qquad\quad\left(\int_0^{2\sqrt{t-s}}+\int_{2\sqrt{t-s}}^{\frac{2(t-s)}{r}} \right)\exp\left\{-\frac{\tilde r^2}{4(t-s)}\right\}\frac{\tilde r^{n-4}}{(\tilde r-c_n\sqrt{T-s})^2+|\tilde z|^2}  d\tilde r\Bigg|\\
\lesssim&~\Bigg|\int_0^{t-(T-t)}\frac{\dot\la(s)}{(t-s)^{n/2}} ds\int_{\R^3} \exp\left\{-\frac{|z-\tilde z|^2}{4(t-s)}\right\}d\tilde z \int_0^{2\sqrt{t-s}} \frac{\tilde r^{n-4}}{(\tilde r-c_n\sqrt{T-s})^2+|\tilde z|^2}  d\tilde r\Bigg|\\
&~+\Bigg|\int_0^{t-(T-t)}\frac{\dot\la(s)}{(t-s)^{n/2}} ds\int_{\R^3} \exp\left\{-\frac{|z-\tilde z|^2}{4(t-s)}\right\}d\tilde z \int_{2\sqrt{t-s}}^{\frac{2(t-s)}{r}} \left(\frac{t-s}{\tilde r^2}\right)^\alpha\frac{\tilde r^{n-4}}{(\tilde r-c_n\sqrt{T-s})^2+|\tilde z|^2}  d\tilde r\Bigg|\\
\lesssim&~\Bigg|\int_0^{t-(T-t)}\frac{\dot\la(s)}{(t-s)^{3/2}} ds\int_{\R^3} \exp\left\{-\frac{|z-\tilde z|^2}{4(t-s)}\right\}\frac{1}{T-s+|\tilde z|^2}d\tilde z\\
&~+\Bigg|\int_0^{t-(T-t)}\frac{\dot\la(s)}{(t-s)^{n/2-\alpha}} ds\int_{\R^3} \exp\left\{-\frac{|z-\tilde z|^2}{4(t-s)}\right\}d\tilde z \int_{2\sqrt{t-s}}^{\frac{2(t-s)}{r}} \frac{\tilde r^{n-4-2\alpha}}{(\tilde r-c_n\sqrt{T-s})^2+|\tilde z|^2}  d\tilde r\Bigg|\\
\lesssim&~\Bigg|\int_0^{t-(T-t)}\frac{\dot\la(s)}{t-s} ds\Bigg|\lesssim |\log T|^{-\frac{2}{n-2}},
\end{align*}
where we chose $\alpha=\frac{n-2}{2}-\epsilon$ for $\epsilon>0$ small and have used the bound $|\dot\la(t)|\lesssim|\log(T-t)|^{-\frac{n}{n-2}}$. Similarly,
\begin{align*}
    |{\rm I}_2|\lesssim&~\Bigg|\int_{t-(T-t)}^{t-\frac{r^2}{4}}\frac{\dot\la(s)}{t-s} ds\Bigg|\lesssim |\dot\la(t)|\left|\log \frac{r^2}{4(T-t)}\right|,
\end{align*}
\begin{align*}
    |{\rm I}_3|\lesssim&~\Bigg|\int^t_{t-\frac{r^2}{4}} \frac{\dot\la(s)}{(t-s)^{n/2}}\left(\frac{t-s}{r^2}\right)^\alpha ds\int_{\R^3} \exp\left\{-\frac{|z-\tilde z|^2}{4(t-s)}\right\}d\tilde z\\
&~\qquad\quad\int_0^{\frac{2(t-s)}{r}} \exp\left\{-\frac{\tilde r^2}{4(t-s)}\right\}\frac{\tilde r^{n-4}}{(\tilde r-c_n\sqrt{T-s})^2+|\tilde z|^2}  d\tilde r\Bigg|\\
\lesssim&~r^{3-n-2\alpha}\left|\int_{t-\frac{r^2}{4}}^t \frac{\dot\la(s)}{(t-s)^{\frac{5-n}{2}-\alpha}}ds\right|\lesssim|\dot\la(t)|.
\end{align*}

\noindent $\bullet$ 
If $2\sqrt{T-t}<r$, then similarly
\begin{align*}
    |{\rm I}|\lesssim&~\Bigg|\left(\int_0^{t-\frac{r^2}{4}}+\int_{t-\frac{r^2}{4}}^{t-(T-t)}+\int_{t-(T-t)}^t\right) \frac{\dot\la(s)}{(t-s)^{n/2}}\exp\left\{-\frac{r^2}{4(t-s)}\right\} ds\int_{\R^3} \exp\left\{-\frac{|z-\tilde z|^2}{4(t-s)}\right\}d\tilde z\\
&~\qquad\quad\int_0^{\frac{2(t-s)}{r}} \exp\left\{-\frac{\tilde r^2}{4(t-s)}\right\}\frac{\tilde r^{n-4}}{(\tilde r-c_n\sqrt{T-s})^2+|\tilde z|^2}  d\tilde r\Bigg|\\
\lesssim&~\Bigg|\int_0^{t-(T-t)}\frac{\dot\la(s)}{t-s} ds\Bigg|+|\dot\la(t)|\left(\frac{r}{\sqrt{T-t}}\right)^{3-n-2\alpha}\lesssim |\log T|^{-\frac2{n-2}}.
\end{align*}

In sum, we have that
$$
    \begin{aligned}
        |{\rm I}|\lesssim&~\Bigg|\int_0^{t-\frac{r^2}{4}}\frac{\dot\la(s)}{t-s} ds\Bigg|+|\dot\la(t)|\lesssim |\log T|^{-\frac{2}{n-2}}+|\dot\la(t)|\left|\log \frac{r^2}{4(T-t)}\right|.
    \end{aligned}
$$

We next estimate that for $0<r\leq \sqrt{T-t}$,
\begin{align*}
        |{\rm II}|\lesssim &~r^{\frac{4-n}{2}}\Bigg|\int_0^t \frac{\dot\la(s)}{(t-s)^{2}} ds\int_{\R^3} \exp\left\{-\frac{|z-\tilde z|^2}{4(t-s)}\right\}d\tilde z\\
&~\qquad\quad\int_{\frac{2(t-s)}{r}}^\infty \exp\left\{-\frac{(\tilde r-r)^2}{4(t-s)}\right\}\frac{\tilde r^{\frac{n-4}{2}}}{(\tilde r-c_n\sqrt{T-s})^2+|\tilde z|^2}  d\tilde r\Bigg|\\
=&~r^{\frac{4-n}{2}}\Bigg|\left(\int_0^{t-(T-t)}+\int_{t-(T-t)}^{t-r^2}+\int_{t-r^2}^t \right) \frac{\dot\la(s)}{(t-s)^{2}} ds\int_{\R^3} \exp\left\{-\frac{|z-\tilde z|^2}{4(t-s)}\right\}d\tilde z\\
&~\qquad\quad\int_{\frac{2(t-s)}{r}}^\infty \exp\left\{-\frac{(\tilde r-r)^2}{4(t-s)}\right\}\frac{\tilde r^{\frac{n-4}{2}}}{(\tilde r-c_n\sqrt{T-s})^2+|\tilde z|^2}  d\tilde r\Bigg|\\
:=&~{\rm II}_1+{\rm II}_2+{\rm II}_3.
    \end{align*}
Here
\begin{align*}
    |{\rm II}_1|\lesssim&~r^{\frac{4-n}{2}}\Bigg|\int^{t-(T-t)}_0 \frac{\dot\la(s)}{(t-s)^{\frac{8-n}{4}}} ds\int^\infty_{\frac{\sqrt{t-s}}{r}} \exp\left\{-\left(\rho-\frac{r}{2\sqrt{t-s}}\right)^2\right\}\rho^{\frac{n-4}{2}} d\rho\Bigg|\\
    \lesssim&~r^{\frac{4-n}{2}}\Bigg|\int_0^{t-(T-t)} \frac{\dot\la(s)}{(t-s)^{\frac{8-n}{4}}} \left(\frac{\sqrt{t-s}}{r}\right)^{\frac{n-6}{2}}\exp\left\{-\frac{t-s}{r^2}\right\}ds\Bigg|\\
    \lesssim&~ \left|\int_0^{t-(T-t)} \frac{\dot\la(s)}{t-s}ds\right|
    \lesssim |\log T|^{-\frac2{n-2}},
\end{align*}
and
\begin{align*}
    |{\rm II}_2|\lesssim&~ \Bigg|\int_{t-(T-t)}^{t-r^2} \frac{\dot\la(s)}{t-s} ds \Bigg|\lesssim |\dot\la(t)|\left|\log \frac{r^2}{(T-t)}\right|,
\end{align*}
where we have used 
$$
     \int_\zeta^\infty \exp\left\{-\left(\rho-\frac{1}{2\zeta}\right)^2\right\}\rho^{\frac{n-4}{2}}d\rho \sim \begin{cases}
        \zeta^{\frac{4-n}{2}} &~\mbox{ as }~\zeta\to 0^+,\\
        \zeta^{\frac{n-6}{2}} e^{-\zeta^2} &~\mbox{ as }~\zeta\to \infty.
    \end{cases}
$$
Also,
\begin{align*}
    |{\rm II}_3|
    \lesssim&~r^{\frac{4-n}{2}}\Bigg|\int_{t-r^2}^t \frac{\dot\la(s)}{(t-s)^{\frac{8-n}{4}}} ds\int_{\mathbb R^3} \exp\left\{-\left|\hat z-\frac{z}{2\sqrt{t-s}}\right|^2\right\} d\hat z\\
    &~\qquad\qquad\int_{\frac{\sqrt{t-s}}{r}}^\infty \exp\left\{-\left(\hat r-\frac{r}{2\sqrt{t-s}}\right)^2\right\} \frac{\hat r^{\frac{n-4}{2}}}{\left(\hat r-\frac{c_n}{2}\sqrt{\frac{T-s}{t-s}}\right)^2+|\hat z|^2} d\hat r\Bigg|\\
    \lesssim&~ r^{\frac{4-n}{2}}\Bigg|\int_{t-r^2}^t \frac{\dot\la(s)}{(t-s)^{\frac{8-n}{4}}} (t-s)^{\frac{6-n}{4}} r^{\frac{n-4}{2}}ds\Bigg|
    \lesssim|\dot\la(t)|.
\end{align*}
Estimate of ${\rm II}$ in the region $r\geq \sqrt{T-t}$ follows similarly. Collecting all the estimates above we obtain
\begin{equation}\label{est-Psi_0}
    |\Psi_0(r,z,t)|\lesssim |\log T|^{-\frac{2}{n-2}}+|\dot\la(t)|\left|\log \frac{r^2}{(T-t)}\right|.
\end{equation}

\medskip

We next estimate $\Psi_1$ defined in \eqref{def-Psi1}, whose right hand side is defined in \eqref{def-Sout}. We have
\begin{equation}\label{est-Psi_1}
    |\Psi_1(r,z,t)|\lesssim |\log T|^{-\epsilon_0}
\end{equation}
for some $\epsilon_0>0$.
Indeed, due to the cut-off, the size of the first term is smaller than that of $\Psi_0$ as $\left|(1-\eta_R)\frac{\alpha_0\dot\la(t)\la^2(t)}{\rho^2(\la^2+\rho^2)}\right|\lesssim R^{-2}\frac{|\dot\la|}{\rho^2}.$ For the last two terms in \eqref{def-Sout}, since
$$\xi_{r,*}(t)=\sqrt{2(n-4)(T-t)},$$
we have
$$\dot\xi_{r,*}+\frac{n-4}{\la_* y_1+\xi_{r,*}}= -\frac{(n-4)\la_* y_1}{\xi_{r,*}(\la_* y_1+\xi_{r,*})},$$
and thus
\begin{equation*}
\begin{aligned}
&\quad \left|(1-\eta_R)\left(\frac{n-4}{\la_* y_1+\xi_{r,*}}\la_*^{-2}\partial_{y_1} U(y)+\la_*^{-2}\nabla U\cdot \dot\xi_*\right)\right|\\
&\lesssim \frac{\la_*^{-1}y_1^2}{\xi_{r,*}(\la_* y_1+\xi_{r,*})(1+|y|^2)^2}\chi_{\{R\leq |y|\}}\lesssim \frac{\la_*(t)}{\rho^2r\sqrt{T-t}}\chi_{\{\la_*R\lesssim\rho\}}.
\end{aligned}
\end{equation*}
Then if $r\gtrsim\sqrt{T-t}$, the following rough estimate suffices
\begin{align*}
&~\left|\int_0^t \frac{\la_*(s)}{(t-s)^{n/2}(T-s)}\int_{\la_*(s)R(s)\leq \left|(w_r,w_z)-\xi(s)\right|} \frac{e^{-\frac{|x-w|^2}{4(t-s)}}}{\left|(w_r,w_z)-\xi(s)\right|^2} dwds\right|\\
\lesssim&  \int_0^t \frac{\la_*(s)}{(t-s)(T-s)}\int_{\frac{\la_*(s)R(s)}{\sqrt{t-s}}\leq\left|(\tilde w_r,\tilde w_z)-\frac{\xi(s)}{\sqrt{t-s}}\right|} \frac{e^{-\frac{|\tilde x-\tilde w|^2}{4}}}{\left|(\tilde w_r,\tilde w_z)-\frac{\xi(s)}{\sqrt{t-s}}\right|^2} d\tilde wds\\
\lesssim&  \int_0^t \frac{\la_*(s)}{(t-s)(T-s)}\frac{t-s}{\la_*^2(s)R^2(s)}ds\lesssim\int_0^t \frac{1}{|\log T|^{\frac{2}{n-2}}(T-s)|\log(T-t)|^{2\theta-\frac{n}{n-2}}}ds 
\lesssim  |\log T|^{2-2\theta}
\end{align*}
where we have used $R(t)=(T-t)^{-1/2}|\log(T-t)|^{\theta}$ with $\theta>\frac{n-1}{n-2}$. The treatment of the convolution in the region $0<r\lesssim\sqrt{T-t}$  is similar to that of $\Psi_0$. We omit the details.

\medskip

\bigskip

\subsection*{Acknowledgments}
	M.~del Pino has been supported by the Royal Society Research Professorship grant RP-R1-180114 and by the ERC/UKRI Horizon Europe grant ASYMEVOL, EP/Z000394/1.  The  research of J.~Wei is partially supported by GRF of RGC of Hong Kong entitle "On critical and supercritical Fujita equation".  Y. Zhou is supported in part by the Fundamental Research Funds for the Central Universities.

\medskip






\begin{thebibliography}{10}







\bibitem{aryan} S. Aryan. \newblock Soliton resolution for the energy-critical nonlinear heat equation in the radial case.
\newblock To appear in {\em Analysis \& PDE}.

\bibitem{KS-finite}
Federico Buseghin, Juan D\'{a}vila, Manuel del Pino, and Monica Musso.
\newblock Existence of finite time blow-up in {K}eller-{S}egel system.
\newblock {\em Ann. PDE}, 12(1):Paper No. 17, 144, 2026.

\bibitem{ringKS1}
Federico Buseghin, Juan D\'{a}vila, Manuel del Pino, and Monica Musso.
\newblock Finite-time blow-up for the three dimensional axially symmetric
  {K}eller-{S}egel system.
\newblock {\em J. Funct. Anal.}, 290(12):Paper No. 111443, 57, 2026.


\bibitem{ChaeTsaiLuoHou}
Dongho Chae and Tai-Peng Tsai.
\newblock Remark on Luo-Hou's ansatz for a self-similar solution to the 3D Euler equations.
\newblock {\em Nonlinearity}, 27(8):1933--1941, 2014.




\bibitem{Collot17APDE}
Charles Collot.
\newblock Nonradial type {II} blow up for the energy-supercritical semilinear
  heat equation.
\newblock {\em Anal. PDE}, 10(1):127--252, 2017.


\bibitem{Collot-KS}
Charles Collot, Tej-Eddine Ghoul, Nader Masmoudi, and Van~Tien Nguyen.
\newblock Refined description and stability for singular solutions of the 2{D}
  {K}eller-{S}egel system.
\newblock {\em Comm. Pure Appl. Math.}, 75(7):1419--1516, 2022.

\bibitem{Collot-KS-2023}
Charles Collot, Tej-Eddine Ghoul, Nader Masmoudi, and Van~Tien Nguyen.
\newblock Collapsing-ring blowup solutions for the Keller-Segel system in three dimensions and higher. \newblock {\em J. Funct. Anal.} 285 (2023), no. 7, Paper No. 110065, 41 pp.

\bibitem{Collot17CMP}
Charles Collot, Frank Merle, and Pierre Rapha\"el.
\newblock Dynamics near the ground state for the energy critical nonlinear heat
  equation in large dimensions.
\newblock {\em Comm. Math. Phys.}, 352(1):215--285, 2017.

\bibitem{CollotMerleRaphael}
Charles Collot, Frank Merle, and Pierre Rapha\"{e}l.
\newblock Strongly anisotropic type {II} blow up at an isolated point.
\newblock {\em J. Amer. Math. Soc.}, 33(2):527--607, 2020.


\bibitem{Green16JEMS}
Carmen Cort\'{a}zar, Manuel del Pino, and Monica Musso.
\newblock Green's function and infinite-time bubbling in the critical nonlinear
  heat equation.
\newblock {\em J. Eur. Math. Soc. (JEMS)}, 22(1):283--344, 2020.


\bibitem{gluingKS}
Juan D\'{a}vila, Manuel del Pino, Jean Dolbeault, Monica Musso, and Juncheng
  Wei.
\newblock Existence and stability of infinite time blow-up in the
  {K}eller-{S}egel system.
\newblock {\em Arch. Ration. Mech. Anal.}, 248(4):Paper No. 61, 154, 2024.





\bibitem{17HMF}
Juan D\'{a}vila, Manuel del Pino, and Juncheng Wei.
\newblock Singularity formation for the two-dimensional harmonic map flow into
  {$S^2$}.
\newblock {\em Invent. Math.}, 219(2):345--466, 2020.








\bibitem{type25D}
Manuel del Pino, Monica Musso, and Juncheng Wei.
\newblock Type {II} {B}low-up in the 5-dimensional {E}nergy {C}ritical {H}eat
  {E}quation.
\newblock {\em Acta Math. Sin. (Engl. Ser.)}, 35(6):1027--1042, 2019.



\bibitem{17type2}
Manuel del Pino, Monica Musso, and Juncheng Wei.
\newblock Geometry driven type {II} higher dimensional blow-up for the critical
  heat equation.
\newblock {\em J. Funct. Anal.}, 280(1):108788, 49, 2021.



\bibitem{ni4d}
Manuel del Pino, Monica Musso, Juncheng Wei, and Yifu Zhou.
\newblock Type {II} finite time blow-up for the energy critical heat equation
  in {$\Bbb R^4$}.
\newblock {\em Discrete Contin. Dyn. Syst.}, 40(6):3327--3355, 2020.




\bibitem{FHV00}
Stathis Filippas, Miguel~A. Herrero, and Juan J.~L. Vel\'azquez.
\newblock Fast blow-up mechanisms for sign-changing solutions of a semilinear
  parabolic equation with critical nonlinearity.
\newblock {\em R. Soc. Lond. Proc. Ser. A Math. Phys. Eng. Sci.},
  456(2004):2957--2982, 2000.

\bibitem{Fujita66}
Hiroshi Fujita.
\newblock On the blowing up of solutions of the {C}auchy problem for
  {$u_{t}=\Delta u+u^{1+\alpha }$}.
\newblock {\em J. Fac. Sci. Univ. Tokyo Sect. I}, 13:109--124 (1966), 1966.






\bibitem{Giga87Indiana}
Yoshikazu Giga and Robert~V. Kohn.
\newblock Characterizing blowup using similarity variables.
\newblock {\em Indiana Univ. Math. J.}, 36(1):1--40, 1987.


\bibitem{Giga04Indiana}
Yoshikazu Giga, Shin'ya Matsui, and Satoshi Sasayama.
\newblock Blow up rate for semilinear heat equations with subcritical
  nonlinearity.
\newblock {\em Indiana Univ. Math. J.}, 53(2):483--514, 2004.

\bibitem{Giga04MMAS}
Yoshikazu Giga, Shin'ya Matsui, and Satoshi Sasayama.
\newblock On blow-up rate for sign-changing solutions in a convex domain.
\newblock {\em Math. Methods Appl. Sci.}, 27(15):1771--1782, 2004.



\bibitem{harada6D}
Junichi Harada.
\newblock A type {II} blowup for the six dimensional energy critical heat
  equation.
\newblock {\em Ann. PDE}, 6(2):Paper No. 13, 63, 2020.


\bibitem{Velazquez94pJL}
Miguel~A. Herrero and Juan J.~L. Vel\'{a}zquez.
\newblock Explosion de solutions d'\'{e}quations paraboliques semilin\'{e}aires
  supercritiques.
\newblock {\em C. R. Acad. Sci. Paris S\'{e}r. I Math.}, 319(2):141--145, 1994.

\bibitem{herrero1992blow}
Miguel~A Herrero and Juan~JL Vel{\'a}zquez.
\newblock A blow up result for semilinear heat equations in the supercritical
  case.
\newblock {\em preprint}, 1992.



\bibitem{HouHuangEuler}
Thomas Y. Hou and De Huang.
\newblock Potential singularity formation of incompressible axisymmetric Euler
equations with degenerate viscosity coefficients.
\newblock {\em Physica D}, 435:133257, 2022.

\bibitem{HouNSBoussinesq}
Thomas Y. Hou.
\newblock Nearly self-similar blowup of generalized axisymmetric
Navier--Stokes and Boussinesq equations.
\newblock {\em arXiv preprint arXiv:2405.10916}, 2024.














\bibitem{KimMerle}
Kihyun Kim and Frank Merle.
\newblock On classification of global dynamics for energy-critical equivariant
  harmonic map heat flows and radial nonlinear heat equation.
\newblock {\em Comm. Pure Appl. Math.}, 78(9):1783--1842, 2025.

\bibitem{Merle04CPAM}
Hiroshi Matano and Frank Merle.
\newblock On nonexistence of type {II} blowup for a supercritical nonlinear
  heat equation.
\newblock {\em Comm. Pure Appl. Math.}, 57(11):1494--1541, 2004.

\bibitem{Merle09JFA}
Hiroshi Matano and Frank Merle.
\newblock Classification of type {I} and type {II} behaviors for a
  supercritical nonlinear heat equation.
\newblock {\em J. Funct. Anal.}, 256(4):992--1064, 2009.




\bibitem{Merle97Duke}
Frank Merle and Hatem Zaag.
\newblock Stability of the blow-up profile for equations of the type
  {$u_t=\Delta u+|u|^{p-1}u$}.
\newblock {\em Duke Math. J.}, 86(1):143--195, 1997.

\bibitem{Mizoguchi05Indiana}
Noriko Mizoguchi.
\newblock Boundedness of global solutions for a supercritical semilinear heat
  equation and its application.
\newblock {\em Indiana Univ. Math. J.}, 54(4):1047--1059, 2005.

\bibitem{Mizoguchi11JDE}
Noriko Mizoguchi.
\newblock Nonexistence of type {II} blowup solution for a semilinear heat
  equation.
\newblock {\em J. Differential Equations}, 250(1):26--32, 2011.


\bibitem{MussoNAMS}
Monica Musso.
\newblock Bubbling blow-up in critical elliptic and parabolic problems.
\newblock {\em Notices Amer. Math. Soc.}, 69(10):1700--1706, 2022.



\bibitem{Quittner21Duke}
Pavol Quittner.
\newblock Optimal {L}iouville theorems for superlinear parabolic problems.
\newblock {\em Duke Math. J.}, 170(6):1113--1136, 2021.

\bibitem{Souplet07book}
Pavol Quittner and Philippe Souplet.
\newblock {\em Superlinear parabolic problems}.
\newblock Birkh\"{a}user Advanced Texts: Basler Lehrb\"{u}cher. [Birkh\"{a}user
  Advanced Texts: Basel Textbooks]. Birkh\"{a}user/Springer, Cham, 2019.
\newblock Blow-up, global existence and steady states, Second edition of [
  MR2346798].


\bibitem{Schweyer12JFA}
R\'emi Schweyer.
\newblock Type {II} blow-up for the four dimensional energy critical semi
  linear heat equation.
\newblock {\em J. Funct. Anal.}, 263(12):3922--3983, 2012.







\bibitem{TaoEuler}
Terence Tao.
\newblock Finite time blowup for {L}agrangian modifications of the
  three-dimensional {E}uler equation.
\newblock {\em Ann. PDE}, 2(2):Art. 9, 79, 2016.

\bibitem{wang2021refined}
Kelei Wang and Juncheng Wei.
\newblock Refined blowup analysis and nonexistence of type {II} blowups for an
  energy critical nonlinear heat equation.
\newblock {\em arXiv preprint arXiv:2101.07186}, 2021.




\end{thebibliography}

\end{document}